\documentclass[12pt]{amsart}

\usepackage{epsfig}
\usepackage{enumerate}

\newtheorem{thm}{Theorem}[section]

\newtheorem{defn}[thm]{Definition}
\newtheorem{lemma}[thm]{Lemma}

\newtheorem{cor}[thm]{Corollary}
\newtheorem{remark}[thm]{Remark}
\newtheorem{example}[thm]{Example}

\usepackage{amsmath}
\usepackage{amsxtra}
\usepackage{amscd}
\usepackage{amsthm}
\usepackage{amsfonts}
\usepackage{amssymb}
\usepackage{eucal}

\newcommand{\Z}{{\mathbb Z}}

\newcommand{\N}{{\mathbb N}}

\newcommand{\ba}{{\mathbf a}}
\newcommand{\bm}{{\mathbf m}}

\newcommand{\bu}{{\mathbf u}}

\newcommand{\bk}{{\mathbf k}}
\newcommand{\bt}{{\mathbf t}}

\numberwithin{equation}{section}

\begin{document}
\title{$T$-systems with boundaries from network solutions}
\author{Philippe Di Francesco} 
\address{
Institut de Physique Th\'eorique du Commissariat \`a l'Energie Atomique, 
Unit\'e de Recherche associ\'ee du CNRS,
CEA Saclay/IPhT/Bat 774, F-91191 Gif sur Yvette Cedex, 
FRANCE. e-mail: philippe.di-francesco@cea.fr}
\author{Rinat Kedem}
\address{Department of Mathematics, University of Illinois MC-382, 
Urbana, IL 61821, U.S.A. e-mail: rinat@illinois.edu}

\date{\today}

\begin{abstract}
In this paper, we use the network solution of the $A_r$
$T$-system to derive that of the unrestricted $A_\infty$ $T$-system, equivalent to the octahedron relation. 
We then present a method for implementing various boundary conditions on this system,
which consists of picking initial data with suitable symmetries. The corresponding restricted
$T$-systems are solved exactly in terms of networks. This gives a simple explanation for 
phenomena such as the Zamolodchikov periodicity property 
for $T$-systems (corresponding to the case $A_\ell\times A_r$) and a combinatorial
interpretation for the positive Laurent property 
for the variables of the associated cluster algebra. 
We also explain the relation between the $T$-system wrapped on a torus
and the higher pentagram maps of Gekhtman et al.
\end{abstract}

\maketitle
\tableofcontents
%
%

\section{Introduction}
\subsection{Discrete Integrable systems, positivity and periodicity}

Discrete integrable systems are evolution equations in a discrete time variable $k\in \Z$ that
admit a sufficient number of conservation laws or integrals of motion, in the Liouville sense.
In this note we concentrate essentially on the so-called $T$-system, which first arose in the context
of integrable quantum spin chains, as a system of equations satisfied by the eigenvalues of transfer matrices
of generalized Heisenberg magnets, with the symmetry of a given Lie algebra \cite{KNS}. In the case of type A, the
$T$-system equation is also often referred to as the octahedron recurrence, and appears to be central
in a number of combinatorial objects, such as: the $\lambda$-deformed determinant introduced by Robbins and Rumsey
and its interplay with Alternating Sign Matrices \cite{RR}; the puzzles leading to the proof of positivity of 
Littlewood-Richardson coefficients for $sl_n$ \cite{KTW}; the partition function of domino tilings of 
the Aztec diamond \cite{EKLP,SPY}. Finally the $T$-system
plays a central role as discrete integrable system, where it is 
referred to as the discrete Hirota equation \cite{KLWZ}. Note that an interesting deformation of the T-system was considered
by Nakajima \cite{Nakajima}; the corresponding system is obeyed by the so-called q,t-characters of quantum affine
algebras.

A new interpretation for the $T$-system arose from realizing that the corresponding discrete
evolution could be viewed as a particular mutation in a suitably defined cluster algebra \cite{DFK08}. 
As such, it must satisfy the Laurent property, namely that any solution is a Laurent polynomial
of any set of admissible initial data \cite{FZI}. Moreover, the general positivity conjecture for cluster
algebras would also imply that these Laurent polynomials have {\it non-negative} integer coefficients.
Positivity for the unrestricted $T$-system expressed in terms of ``flat" initial data
follows from the interpretation of the solution as a positively weighted
partition function for domino tilings of the Aztec diamond \cite{SPY}. In the present paper, we first generalize this 
result to an explicit network solution for arbitrary initial data which we then adapt to include various types of
boundary conditions. The network solutions display in particular the positive Laurent phenomenon.

Another fundamental property of $T$-systems was conjectured by Zamolodchikov \cite{ZAMO}
in the form
of periodicity properties of the so-called $Y$-systems in the presence of special boundary conditions, which
is a result of similar periodicity properties for the $T$-system. 
In a more general setting, the $Y$-system is attached to the Dynkin diagram $G$ of a Lie algebra, 
and the special boundaries are coded by another Dynkin diagram $G'$, while the period of the system
is given by $2(h_G+h_{G'})$, where $h_G$ is the Coxeter number of the corresponding algebra.
This periodicity has been proved by many authors \cite{Volkov,Szenes,Henriques,IIKNS,IIKKN} 
for the case when either or both $G,G'$ are of $A$-type, culminating in the general proof of Keller \cite{Keller} 
using category theory, for the case of any pair of Dynkin diagrams
$G,G'$. Note that the various methods of proof used in these works
do not imply the positive Laurent phenomenon. 
The method presented in this paper for $(A_r,A_\ell)$, based on the explicit network solution, 
provides a simple combinatorial explanation for this property. 

By analogy with the solutions of the so-called $Q$-systems \cite{DFK3}, based on an 
explicit construction of conserved quantities,
a first solution of $A_r$ $T$-systems for particular periodic initial data surface
was produced in terms of partition functions of paths with time- and space-dependent weights 
on some target graphs \cite{DFK09a}. 
Finally the $A_r$ $T$-system was 
explicitly solved \cite{DF} for arbitrary admissible initial conditions
in terms of weighted path models on specific networks, coded by the geometry of the initial data surface.

The aim of this paper is to use the network solution of the $A_r$ case to derive properties of solutions of 
$T$-systems with different kinds of boundary conditions. We start with the unrestricted $T$-system:
by using the formulas for $A_r$ for $r$ large enough, we show that
any fixed unrestricted $T$-system solution can be expressed in a compact form, as a principal
minor of a positive network matrix coded by the geometry of the initial conditions. 

To address other boundary conditions, our strategy consists in identifying 
suitable initial data for the $A_r$ system, that {\it imply} the presence of the desired boundaries,
such as walls along which the values of the $T$-system solution must be equal to $1$. 
Once these are identified, we must plug them into the network solution of the $A_r$ system.
The network solution happens to behave nicely under these symmetries, and can be reduced to 
explicit positive expressions in all cases. As a result, we obtain closed formulas for the solutions of the
$T$-system with various boundary conditions.

\subsection{$T$-system: definitions}

Let us now give a few definitions regarding the $T$-system and its various boundary conditions.

\subsubsection{The unrestricted $A_\infty$ $T$-system}

The unrestricted $A_\infty$ $T$-system, also called octahedron recurrence, 
is the following system for formal variables $T_{i,j,k}$, $i,j,k\in \Z$:
\begin{equation}\label{tsys}
T_{i,j,k+1}T_{i,j,k-1}=T_{i,j+1,k}T_{i,j-1,k}+T_{i+1,j,k}T_{i-1,j,k} \qquad (i,j,k\in \Z)\, .
\end{equation}
The system splits into two independent systems 
corresponding to a fixed parity of $i+j+k$. From now on we restrict ourselves to $T_{i,j,k}$ with $i+j+k=0$ mod 2.

This system can be considered as a three-term recursion relation in $k$.
As such it has the following sets of admissible initial conditions, attached to a {\it stepped surface} defined as follows.

\begin{defn}\label{steppeddef}
A stepped surface in the variables $(i,j,k)\in \Z^3$ is a set:
\begin{equation}\label{stepsurf}
\bk = \{(i,j,k_{i,j})\in\Z^3:i+j+k_{i,j}=0\, {\rm mod}\, 2\,  {\rm and}\,  |k_{i,j}-k_{i',j'}|=1 \, {\rm if}\, |i-i'| + |j-j'|=1\}\, .
\end{equation}
\end{defn}

To any such stepped surface, we attach the initial condition:
\begin{equation}\label{initcond} X_\bk(\bt) : \left\{ T_{i,j,k_{i,j}}=t_{i,j}\quad (i,j\in \Z) \right\} \end{equation}
for some formal variables $\bt=\{t_{i,j}\}_{i,j\in \Z}$, which we refer to as initial data/values along the surface $\bk$.

The interplay between these admissible initial conditions
is best understood if we interpret the relation \eqref{tsys} as a {\it mutation} relation
for the cluster algebra related to the $T$-system \cite{DFK08}. In this setting, the admissible initial data
are cluster variables $x_\bk=(t_{i,j})_{i,j\in\Z}$ in a seed of the cluster algebra, and a mutation $\mu_{i,j}$
is simply one application of the relation \eqref{tsys} where $k=k_{i,j+1}=k_{i,j-1}=k_{i+1,j}=k_{i-1,j}$
and either $k_{i,j}=k-1$ (forward mutation) or $k_{i,j}=k+1$ (backward mutation). The mutation $\mu_{i,j}$
sends the surface $\bk$ to a new surface $\bk'$ such that $k_{a,b}'=k_{a,b} +2\epsilon \delta_{a,i}\delta_{b,j}$
with $\epsilon=1$ for a forward mutation, and $\epsilon=-1$ for a backward mutation. Accordingly, the initial
data along the surface $\bk$ is transformed into initial data along $\bk'$ 
by keeping the same values $t_{a,b}'=t_{a,b}$
except for $a=i$ and $b=j$, where 
\begin{equation}\label{mutat} t_{i,j}'={t_{i,j-1}t_{i,j+1}+t_{i-1,j}t_{i+1,j}\over t_{i,j} }\, .\end{equation}
The following is a pictorial representation of a forward mutation:
\begin{equation}\label{octa} \raisebox{-1.cm}{\hbox{\epsfxsize=15.cm \epsfbox{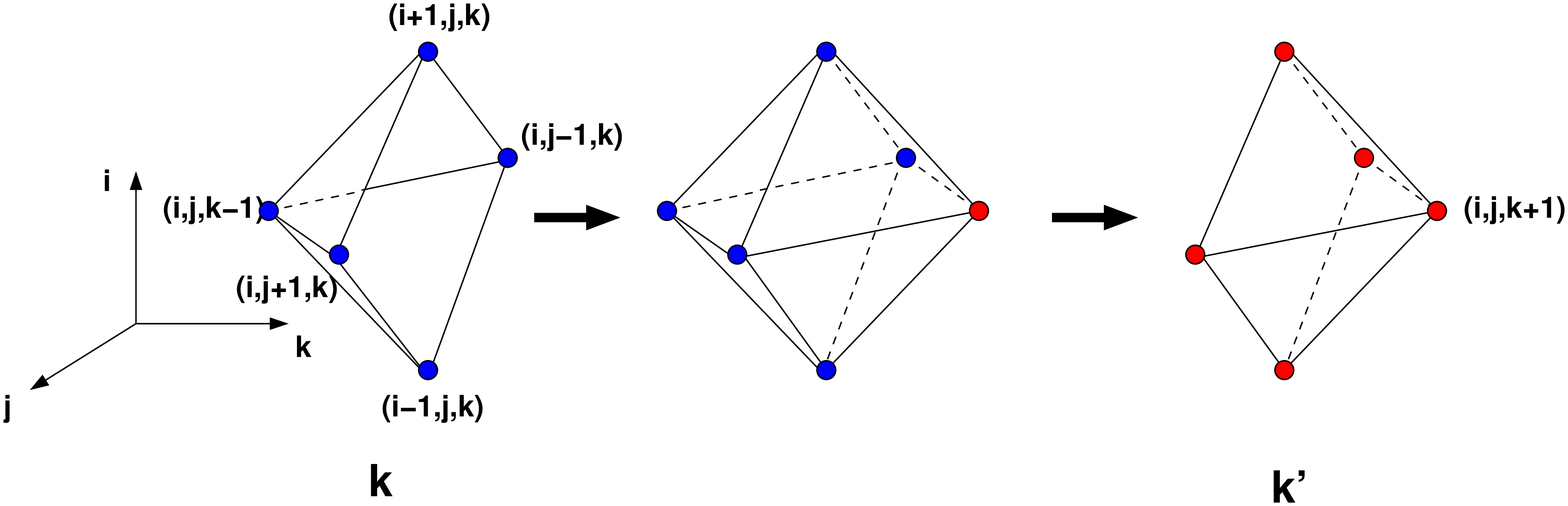}}}  
\end{equation}
It shows how the mutated surface $\bk'$ differs from $\bk$ by one point, which is the sixth point $(i,j,k+1)$
of the incomplete octahedron $(i,j,k-1),(i,j+1,k),(i,j-1,k),(i+1,j,k),(i-1,j,k)$, hence the name ``octahedron" 
equation often used for \eqref{tsys}. Iterating mutations on a given stepped surface $\bk$, we may 
attain any other stepped surface $\bk'$.

In the following, unless otherwise stated, we will refer to the fundamental stepped surface as ``flat" stepped 
surface $\bk_0$ with $k_{i,j}^{(0)}=i+j$ mod 2.

\subsubsection{The $T$-system for $A_r$}

\begin{figure}
\centering
\includegraphics[width=16.cm]{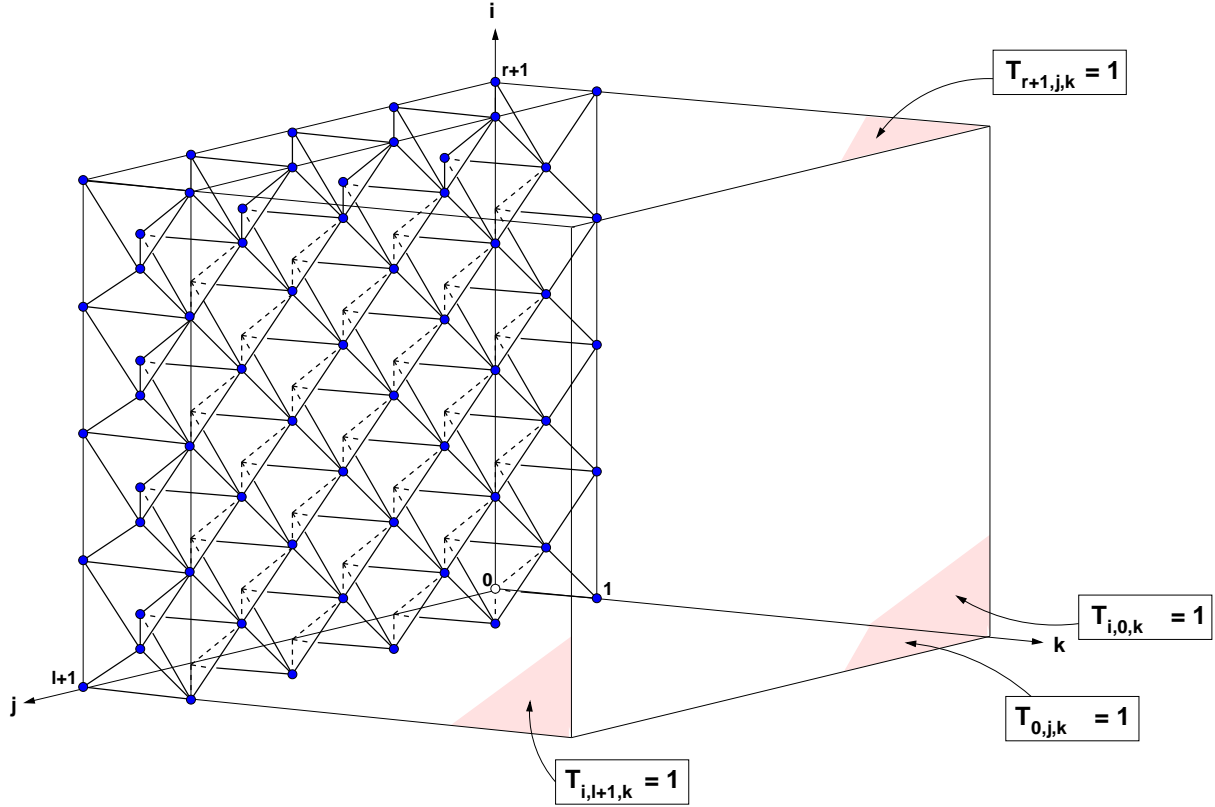}
\caption{\small The $\ell$-restricted $T$-system geometry. We have represented in $(i,j,k)$-space
the four walls along which we set $T_{i,j,k}=1$, as well as the flat initial data stepped surface $\bk_0$
with vertices in the planes $k=0$ and $k=1$.}
\label{fig:boundat}
\end{figure}

In the case of the $A_r$ Lie algebra, the $T$-system \eqref{tsys}
is restricted to values of $i\in [1,r]$ and is subject to the
boundary condition 
\begin{equation}\label{arboundary}
T_{0,j,k}=T_{r+1,j,k}=1\qquad  {\rm for}\, \,  {\rm all}\, \,  j,k\in \Z\, .
\end{equation} 

The system (\ref{tsys}-\ref{arboundary}) can still be considered as a three-term recursion relation in $k$.
The corresponding admissible initial data are attached to infinite strip-like stepped surfaces
$\bk=(i,j,k_{i,j})_{i\in [1,r];j\in \Z}$ such that $k_{i,j}\in \Z$, $i+j+k_{i,j}=0$ mod 2,  
and $|k_{i+1,j}-k_{i,j}|=|k_{i,j+1}-k_{i,j}|=1$
for all $i,j$. To each such $\bk$, we associate the initial conditions:
\begin{equation}\label{initcondar}X_\bk(\bt): \left\{  T_{i,j,k_{i,j}}=t_{i,j}\quad (i\in [1,r];j\in \Z) \right\} \end{equation}
for some formal variables $\bt=\{t_{i,j}\}_{i\in[1,r];j\in\Z}$.

The fundamental ``flat" stepped 
surface, still denoted by $\bk_0$ now has $k_{i,j}^{(0)}=i+j$ mod 2, for $i\in [1,r],j\in \Z$.

The following useful lemma allows to eliminate the $T_{i,j,k}$ for $i>1$ in terms of the $T_{1,j',k'}$'s.

\begin{lemma}{\cite{BR}}\label{elimdet}
The solutions $T_{i,j,k}$ to the $A_r$ $T$-system may be expressed for $i=1,...,r$ as the following 
``discrete Wronskian" determinants involving only $T_{1,j',k'}$'s:
\begin{equation}\label{wronsk}
T_{i,j,k}=\det_{1\leq a,b\leq i} \left( T_{1,j+a-b,k+a+b-i-1} \right)\, .
\end{equation}
\end{lemma}
\begin{proof}
This is a direct consequence of the Desnanot-Jacobi identity relating the determinant $\vert M\vert$ of any
$N\times N$ matrix $M$ to its minors $\vert M\vert _{i_1}^{j_1}$ and $\vert M\vert _{i_1,i_2}^{j_1,j_2}$
obtained respectively by erasing row $i_1$ and column $j_1$ of $M$ or rows $i_1,i_2$ and
columns $j_1,j_2$ of $M$:
\begin{equation}\label{desnajac} \vert M \vert \times \vert M\vert _{1,N}^{1,N}
= \vert M\vert _{1}^{1} \times  \vert M\vert _{N}^{N}-
\vert M\vert _{N}^{1} \times  \vert M\vert _{1}^{N} \, ,\end{equation}
with the convention that the determinant of a $0\times 0$ matrix is $1$. The lemma follows by taking
the $(i+1)\times (i+1)$  matrix $M$ with entries $M_{a,b}=T_{1,j+a-b,k+a+b-i-2}$, $a,b=1,2,...,i+1$.
\end{proof}

We will also consider the $A_r$ $T$-system with so-called
{\it $\ell$-restricted boundary conditions}, in which we restrict the range of $j\in [0,\ell+1]$ and we impose 
\begin{equation}\label{restrictell}
T_{i,0,k}=T_{i,\ell+1,k}=1\qquad  (i\in[1,r];k\in \Z)\, .
\end{equation}
In this case the initial data is also restricted
to a finite sequence $\bt=\{t_{i,j}\}_{i\in [1,r];j\in [1,\ell]}$, and the associated initial conditions read:
\begin{equation}\label{initcondseg}
X_\bk(\bt) : \left\{ T_{i,j,k_{i,j}}=t_{i,j} \quad (i\in[1,r];j\in[1,\ell]) \right\} \, .
\end{equation}
The boundary conditions and flat surface initial data for the $\ell$-restricted $T$-system 
are sketched in Fig.\ref{fig:boundat}.

\subsection{Main results}

In this paper we explore the effect of imposing various boundary conditions of the $T$-system.

For the unrestricted $T$-system, we derive a compact explicit expression for the solution, first  in terms of
the initial data $X_{\bk_0}$ \eqref{initcond} along the flat stepped surface $\bk_0$ 
(Theorem \ref{kok} and Corollary \ref{corTex}). This is then generalized to arbitrary initial
data $X_k$ (Theorem \ref{arbibound}). As a consequence, we have:

\begin{thm}\label{unposiT}
The solution $T_{i,j,k}$ of the unrestricted $T$-system \eqref{tsys} with arbitrary initial conditions $X_\bk$ 
\eqref{initcond} is a Laurent polynomial of the initial values $\{t_{i,j}\}$ with non-negative integer coefficients.
\end{thm}

This extends the result of \cite{SPY}, corresponding to $\bk=\bk_0$ in our language.

Next we consider the $A_r$ $T$-system in different geometries, first in a right or left half-plane
bordered by a ``wall" $j=$constant, along which the value of $T_{i,j,k}$ is fixed to $1$. We show that 
the solutions of such systems coincide with that of the one without a wall, provided we pick initial data 
obeying certain symmetry relations (Theorems \ref{plusimple} and \ref{corolla}). We also show that 
the solutions of the two-wall $\ell$-restricted $A_r$ $T$-system coincide with that of the system without walls
but with initial data obeying multiple reflection symmetries inherited from the two half-plane cases
(Theorem \ref{inithm}). In all cases, we have an explicit formula for the solution
in terms of the initial data.
This will allow us in particular to establish the following two results on the solutions 
of the $\ell$-restricted $A_r$ $T$-system.

\begin{thm}\label{priociT}
The solution of the $\ell$-restricted $A_r$ $T$ system with arbitrary initial 
conditions $X_\bk$ satisfies the following periodicity condition:
$$ T_{i,j,k+N}=T_{i,j,k} \qquad (i\in [1,r];j,k\in \Z)$$
with period $N=2(\ell+r+2)$.
\end{thm}

\begin{thm} \label {positiT}
The solution $T_{i,j,k}$ of the $\ell$-restricted $A_r$ $T$ system with initial conditions $X_{\bk_0}$
\eqref{initcondseg} along the ``flat" stepped surface $\bk_0$
is a Laurent polynomial of the initial values $\{t_{i,j}\}$
with non-negative integer coefficients. 
\end{thm}

In all cases, the positivity of the coefficients will arise from a combinatorial interpretation,
as counting families of non-intersecting paths on 
suitable network graphs.

\subsection{Outline}

The paper is organized as follows. 

In Sect.\ref{netsolsec}, we recall the solution of the 
$A_r$ $T$-system for an arbitrary initial data stepped surface $\bk$. The solution is
expressed in terms of paths on networks. The latter are made of elementary
``chips" associated to $2\times 2$ matrices $U,V$ whose arrangement is coded by the initial data stepped
surface $\bk$, and whose entries are Laurent monomials of the initial data values along $\bk$.

This solution is exploited in Sect.\ref{unsec} to derive the solution of the unrestricted $T$-system for
an arbitrary initial data stepped surface $\bk$ (Theorem \ref{arbibound}). We find that $T_{i,j,k}$ is equal, up to simple
factors of the initial data, to a principal minor of a network matrix corresponding to the {\it shadow}
of the point $(i,j,k)$ onto the stepped surface $\bk$, namely the intersection of $\bk$ and the pyramid
$\{ (x,y,z)\ {\rm such}\ {\rm that}\ |i-x|+|j-y|\leq |k-z|\}$. Theorem \ref{unposiT} follows from this expression.

Sects.\ref{warmupsec} and \ref{ellsec} are devoted to the study of the $\ell$-restricted $A_r$ $T$ system solutions. 
For pedagogical reasons, we first treat the case $r=1$ completely in Sect.\ref{warmupsec}, where we derive
network formulas for the general solution of the $A_1$ $T$-system.  We first treat the case
of the right (resp. left) half-plane $A_1$ $T$-system, which correspond to
imposing a wall-type boundary condition on $T_{1,j,k}$ along the ``wall" $j=0$ (resp. $j=\ell+1$) 
and restricting the range of $j$ to the half-plane $j>0$ (resp. $j<\ell+1$). 
The general strategy is to consider a full plane $A_1$ $T$-system, and to engineer both its initial data stepped surface
and initial values to ensure that the solution coincides with that of the half-plane in the relevant range of $j$. 
This allows to use the general full plane network solution to derive results
in the half-plane geometry, in particular to establish the positive Laurent property of the solution, 
first for the ``flat" initial data stepped surface $\bk_0$ (Theorems \ref{rightbounone} and \ref{leftbounone}),
and then for general $\bk$ (Theorem \ref{arbihalf}). 
Superimposing both half-plane conditions leads to the $\ell$-restricted $A_1$ $T$-system,
whose network solution leads to the $A_1$ version of the periodicity property of Theorem \ref{priociT}
(Theorem \ref{periodone}). This solution allows to prove the $A_1$ version of the Laurent positivity
of Theorem \ref{positiT} (Theorem \ref{ones} for the flat stepped surface $\bk_0$ and Theorem \ref{twos}
for the general stepped surface $\bk$).

The same strategy is then applied to the case of general $r$ in Sect. \ref{ellsec}, namely we impose special restrictions
to the initial data stepped surface and values of the full space $A_r$ $T$-system so as to mimic
wall-type boundaries (left or right half-space) geometries (Theorem \ref{plusimple} and Corollary \ref{corolla}
for the stepped surface $\bk_0$). Finally,
by superimposing the two, we obtain the $\ell$-restricted two-wall boundary geometry (Theorem \ref{inithm} for $\bk_0$). 
These special restrictions however impose the vanishing of the initial values $t_{i,j}$ within square domains of the form
$(i,j)\in [1,r]\times( [-r,-1]\, {\rm mod}\, \ell+r+2)$, which create potential singularities in the 
corresponding network matrices. To repair this, we use a regularization procedure
detailed in Sect.\ref{artwo}, by assigning special non-zero values
within these squares, to be sent to zero in the end. With this trick, all formulas are well-defined
and the relevant limits yield the solutions in half-space (Sect.\ref{arthree}) and $\ell$-restricted geometries
(Sects.\ref{arfour} and \ref{arfive}) .

In the concluding Section 6 we consider other types of boundary conditions
on the $T$-system related to
Frieze patterns of the plane \cite{FRISES}, pentagram \cite{GLICK} and higher pentagram maps \cite{GEKH}. 
We show that the latter are connected to $T$-systems wrapped on a torus, namely with doubly-periodic initial data.
Finally we discuss generalized cut-like boundary conditions and formulate some further positivity conjectures.

\section{Networks and the $A_r$ $T$-system solution}\label{netsolsec}
In this section we recall the network solution \cite{DF} of the infinite $A_r$ $T$-system, not subject to the $\ell$-restriction. The basic building blocks are
matrices $U$ and $V$,  which form the elementary ``chips" of a network. 

\subsection{Definitions and properties of the matrices $U$ and $V$}\label{UVsec}

Define the  $2\times 2$ matrices
\begin{equation}
U(a,b,c)=\begin{pmatrix} 1 & 0 \\ {c \over b} & {a\over b} \end{pmatrix},\qquad 
V(a,b,c)=\begin{pmatrix} {b \over c} & {a\over c}  \\0 & 1 \end{pmatrix}.
\end{equation}
These are embedded in $GL_{r+1}$ in the standard way:
Given $i\in [1,r]$, define $U_i(a,b,c)$ as the $(r+1)\times (r+1)$ matrix with entries
\begin{equation}
\left( U_i(a,b,c) \right)_{k,\ell}=\left\{ \begin{array}{ll}\left( U(a,b,c) \right)_{k-i+1,\ell-i+1} ,& {\rm if} \ k,\ell\in\{i,i+1\} ;\\
\delta_{k,\ell} & {\rm otherwise}, \end{array} \right. 
\end{equation}
and similarly for $V_i(a,b,c)$.

These elementary matrices have the following important 
properties:
\begin{eqnarray}
\label{propuv} U_i(a,b,c)V_{i+1}(b,c,d) &=&V_{i+1}(a,c,d)U_i(a,b,d)\\
\label{propvu} V_i(a,b,c) U_{i+1}(d,e,f)&=&U_{i+1}(d,e,f)V_i(a,b,c)\\
\label{muta} U(a,b,u)V(v,b,c)&=&V(v,a,b') U(b',c,u) \qquad {\rm iff}\qquad b b'=uv+a c
\end{eqnarray}
The third relation is crucial: It gives a representation of the mutation relation of the $T$-system 
(\ref{mutat}-\ref{octa}) via a matrix exchange identity.

\subsection{Pictorial representations}
In this paper we will use two pictorial representations of the elementary network matrices.
\subsubsection{Pictorial representation I}
The matrices $U_i(a,b,c)$ and $V_i(a,b,c)$ are represented as bicolored lozenges:
\begin{equation}\label{repUV}
U_i(a,b,c)=
\raisebox{-1.cm}{\hbox{\epsfxsize=2.5cm \epsfbox{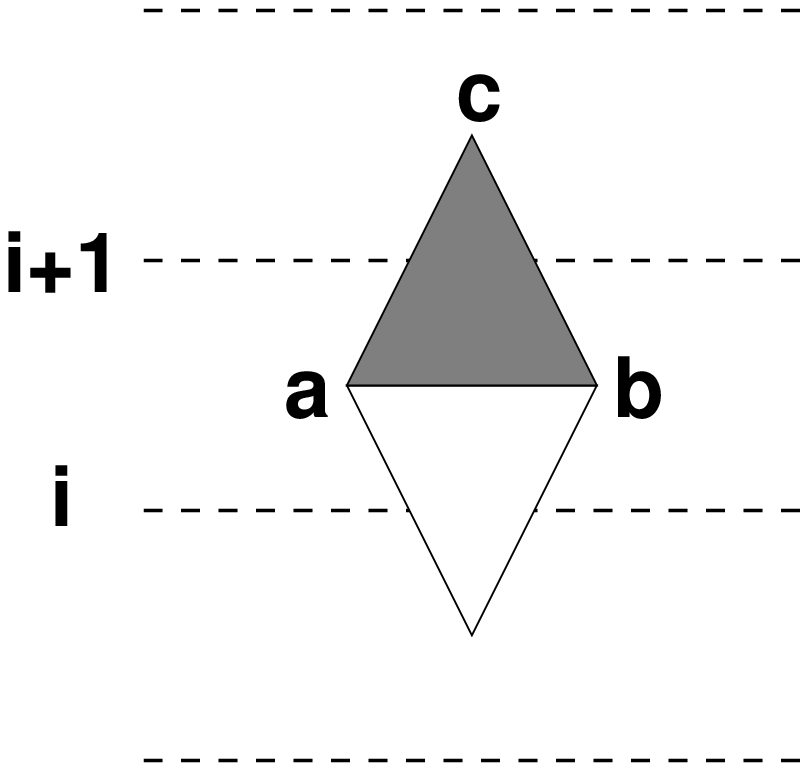}}}\qquad V_i(a,b,c)=
\raisebox{-1.cm}{\hbox{\epsfxsize=2.5cm \epsfbox{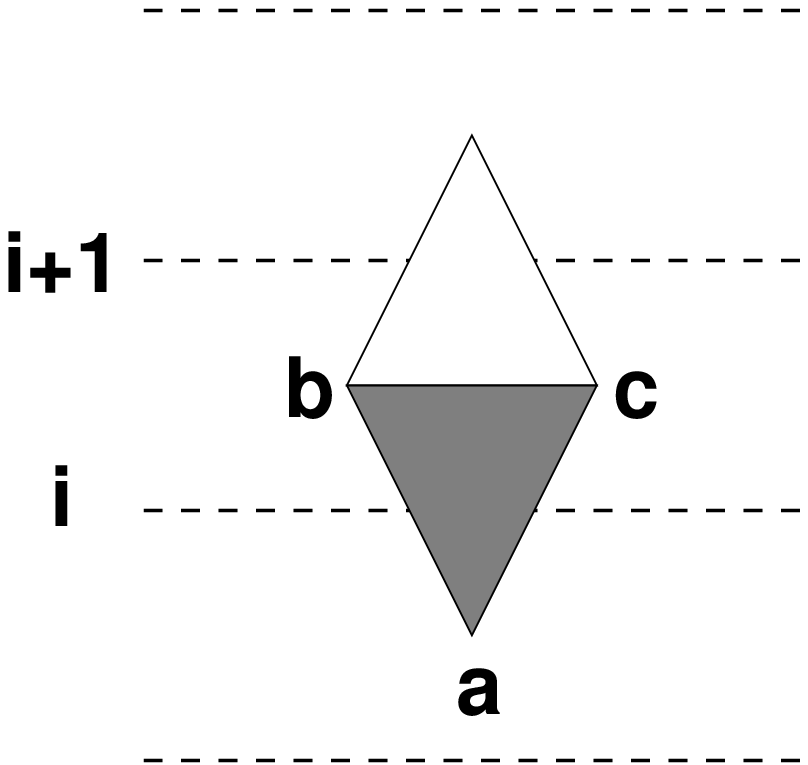}}}
\end{equation}
A product of matrices of $U$ and $V$ type is represented by 
drawing the corresponding lozenges in the same order from left to right, and identifying the edges of the triangles
whenever no other object sits inbetween. This forms a triangulation of some region in the plane.

The relations satisfied by the elementary network matrices can be represented pictorially as follows. 
Property \eqref{propuv} is (we allow the triangles to be slightly deformed):
\begin{equation}\label{repUVVU}
\raisebox{-1.cm}{\hbox{\epsfxsize=3.cm \epsfbox{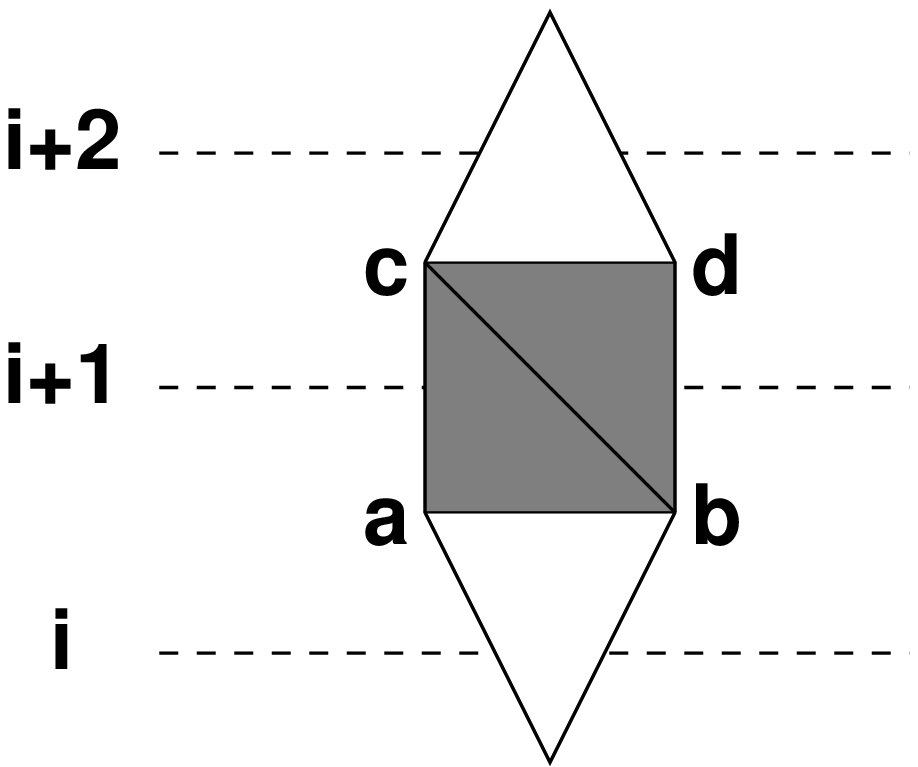}}}\qquad =\qquad 
\raisebox{-1.cm}{\hbox{\epsfxsize=3.cm \epsfbox{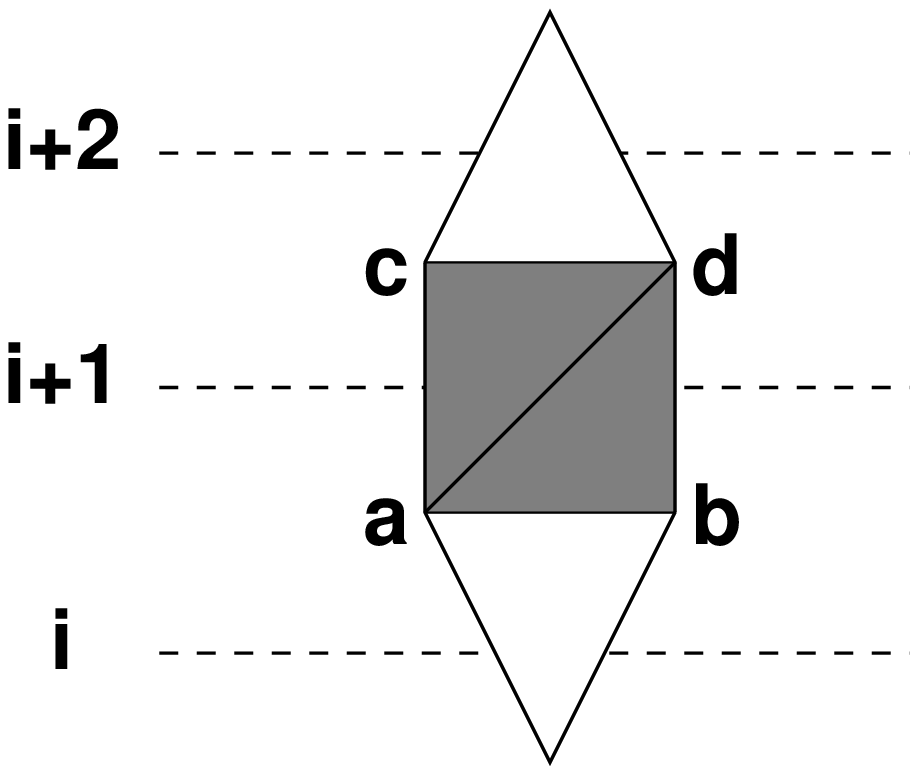}}}
\end{equation}
Property \eqref{propvu} is
\begin{equation}\label{repVUUV}
\raisebox{-1.cm}{\hbox{\epsfxsize=3.cm \epsfbox{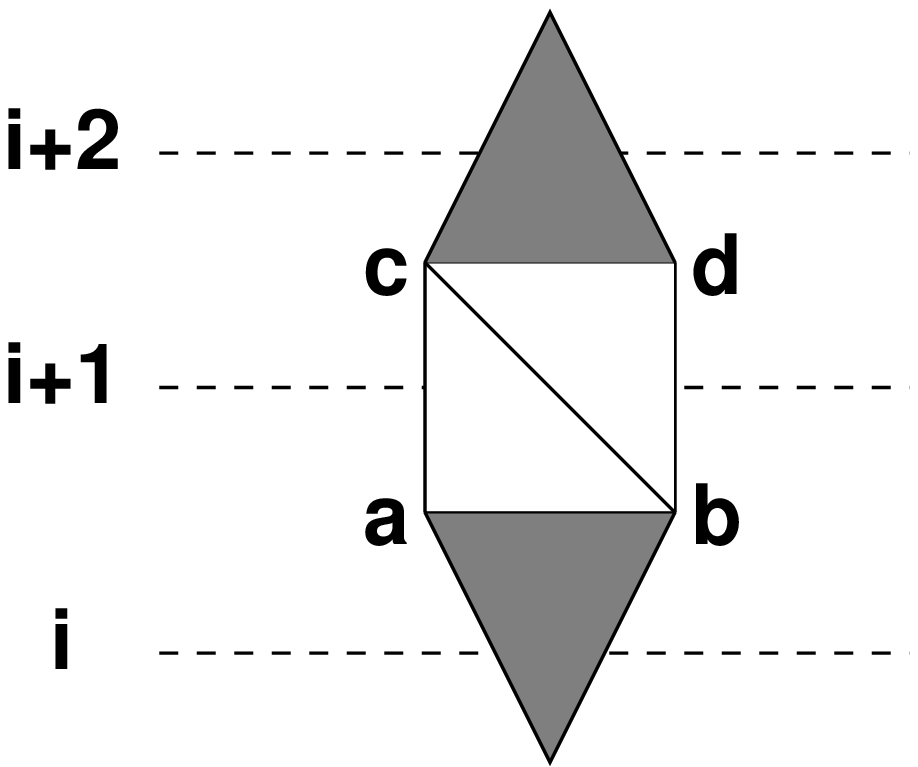}}}\qquad =\qquad 
\raisebox{-1.cm}{\hbox{\epsfxsize=3.cm \epsfbox{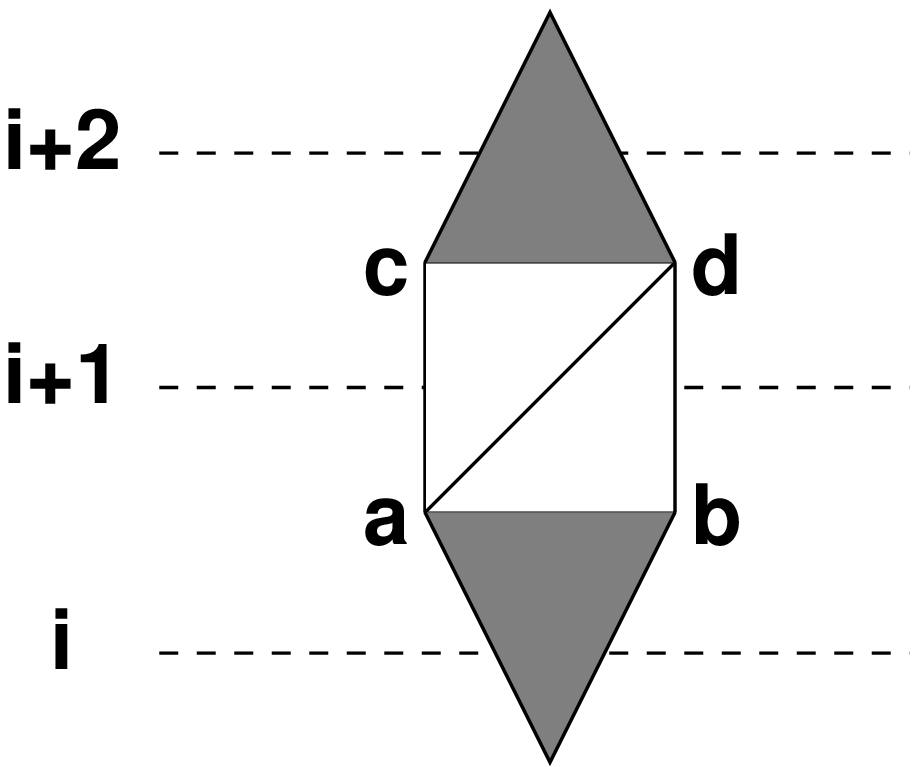}}}
\end{equation}
and the mutation \eqref{muta} is
\begin{equation}\label{repmuta}
\raisebox{-1.1cm}{\hbox{\epsfxsize=2.5cm \epsfbox{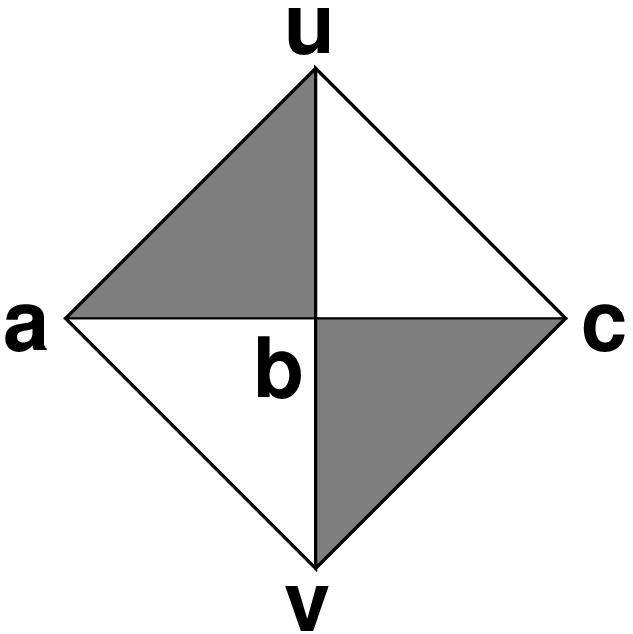}}}\qquad =\qquad 
\raisebox{-1.1cm}{\hbox{\epsfxsize=2.5cm \epsfbox{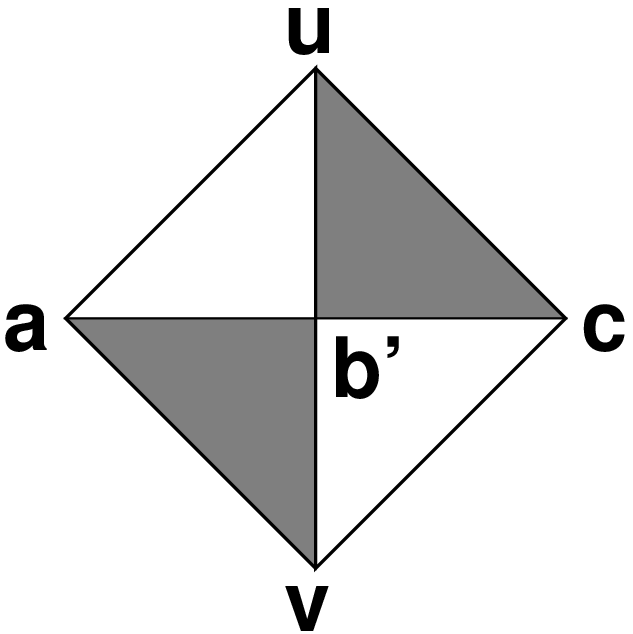}}}
\end{equation}

\subsubsection{Pictorial representation II}
The second useful pictorial representation is as network chips. The picture is as follows:
\begin{equation}
\begin{array}{rll}
U_i(a,b,c)&=
\raisebox{-1.6cm}{\hbox{\epsfxsize=2.cm \epsfbox{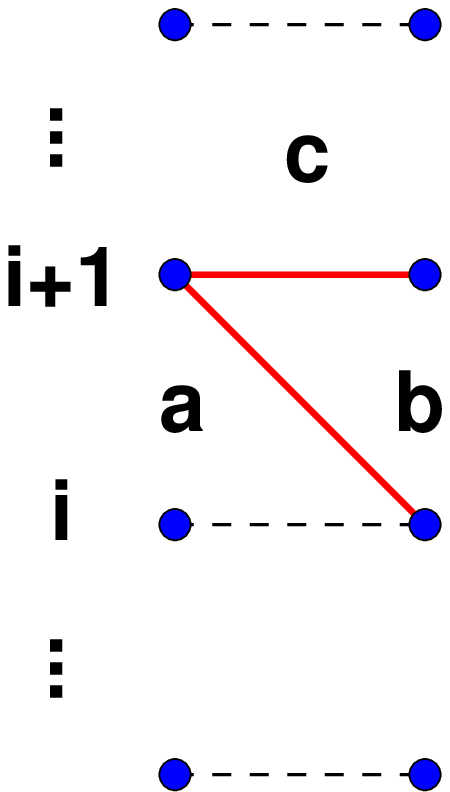}}}&=\raisebox{-1.6cm}{\hbox{\epsfxsize=2.cm \epsfbox{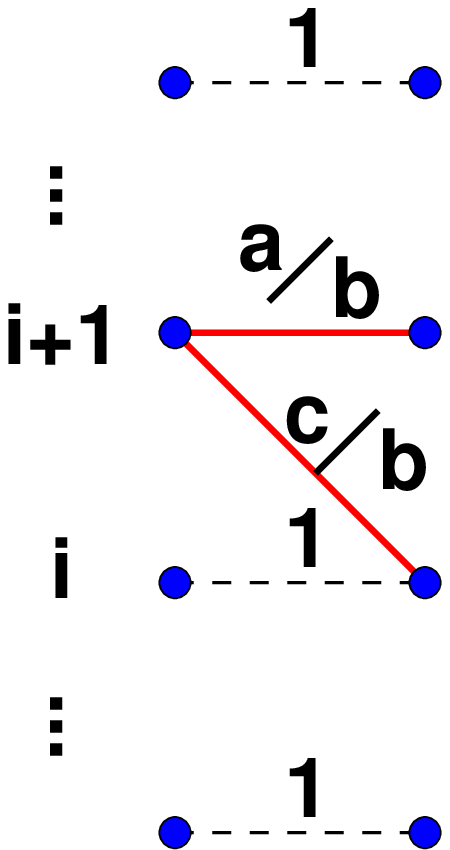}}} \\ \\ \\
\qquad V_i(a,b,c)&=
\raisebox{-1.6cm}{\hbox{\epsfxsize=2cm \epsfbox{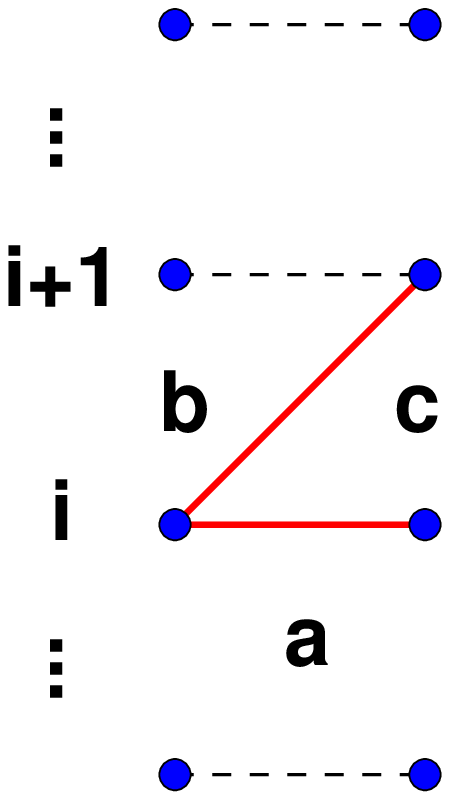}}}&=\raisebox{-1.9cm}{\hbox{\epsfxsize=2.2cm \epsfbox{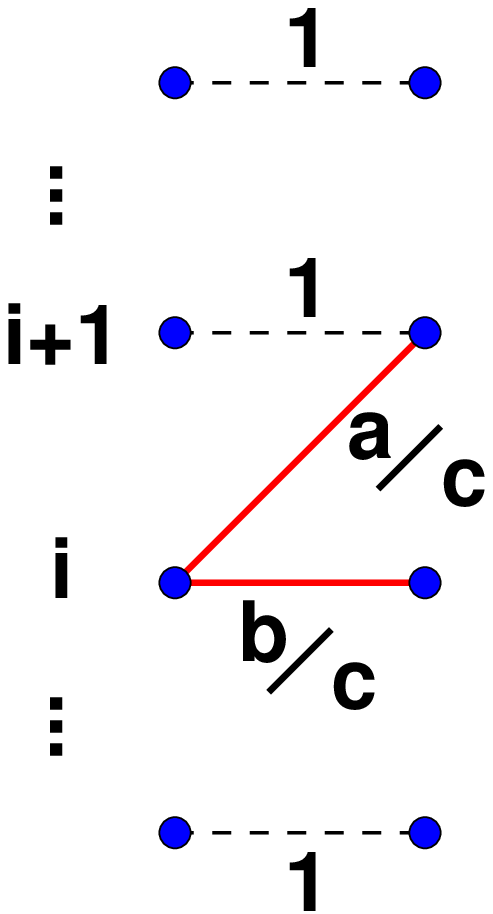}}}\end{array}\label{rep2UV}
\end{equation}
where $a,b,c$ are  represented on the faces of a graph in the left picture. The edges, which are all oriented from left to right (for simplicity, this orientation is omitted in the pictures)
are weighted by the  matrix element $(j,k)$ for an edge from $j$ to $k$.
By convention, dashed edges carry the weight $1$. We have indicated the weights of the edges in the right picture.

In terms of the network chips, property \eqref{propuv} can be illustrated as 
\begin{equation}\label{repnetUVVU}
\raisebox{-1.cm}{\hbox{\epsfxsize=3.cm \epsfbox{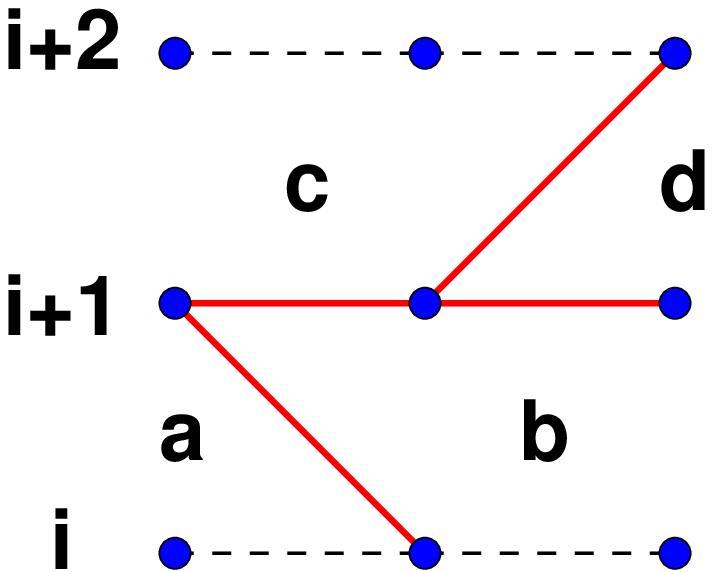}}}\qquad =\qquad 
\raisebox{-1.cm}{\hbox{\epsfxsize=3.cm \epsfbox{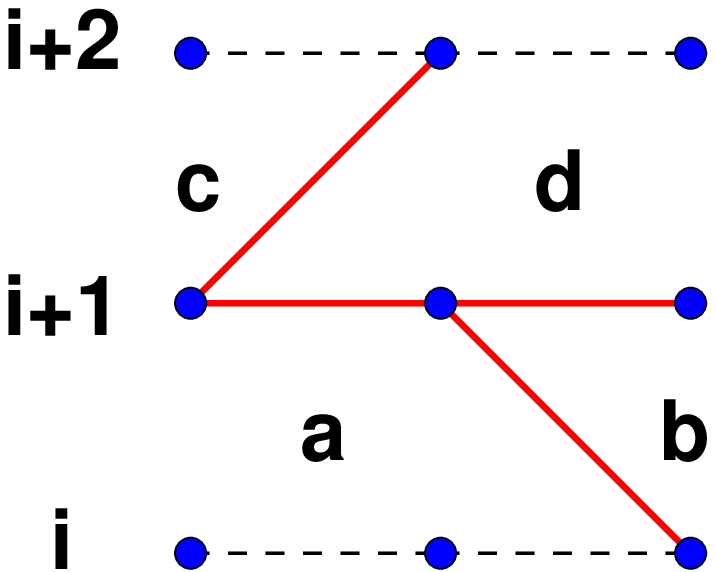}}}
\end{equation}
Property \eqref{propvu} is
\begin{equation}\label{repnetVUUV}
\raisebox{-1.cm}{\hbox{\epsfxsize=3.cm \epsfbox{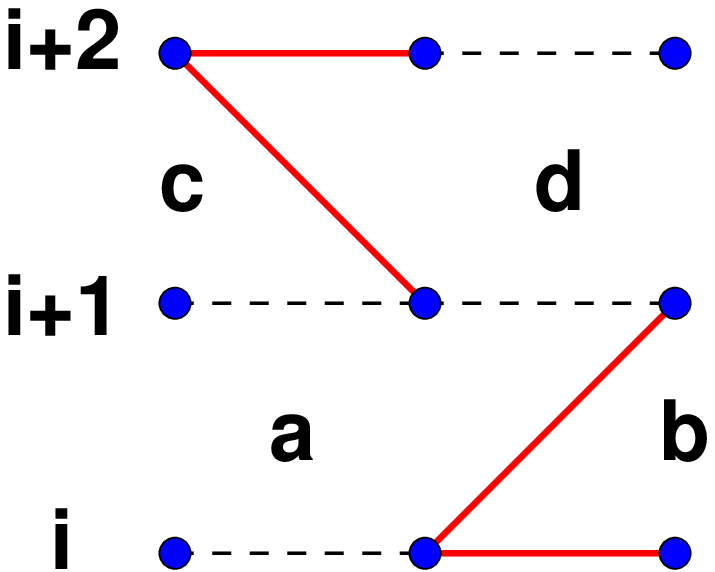}}}\qquad =\qquad 
\raisebox{-1.cm}{\hbox{\epsfxsize=3.cm \epsfbox{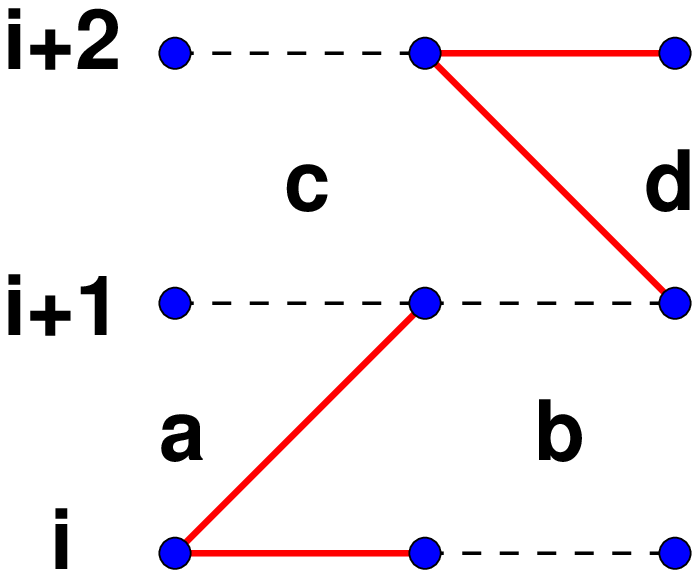}}}
\end{equation}
The mutation \eqref{muta} is illustrated as
\begin{equation}\label{mutnetVUUV}
\raisebox{-1.3cm}{\hbox{\epsfxsize=2.5cm \epsfbox{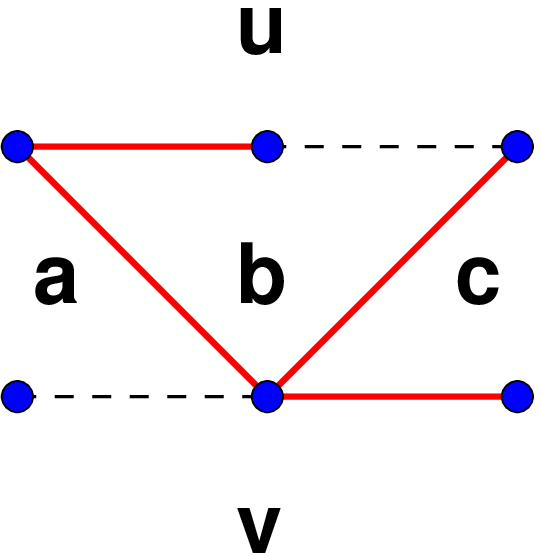}}}\qquad =\qquad 
\raisebox{-1.3cm}{\hbox{\epsfxsize=2.5cm \epsfbox{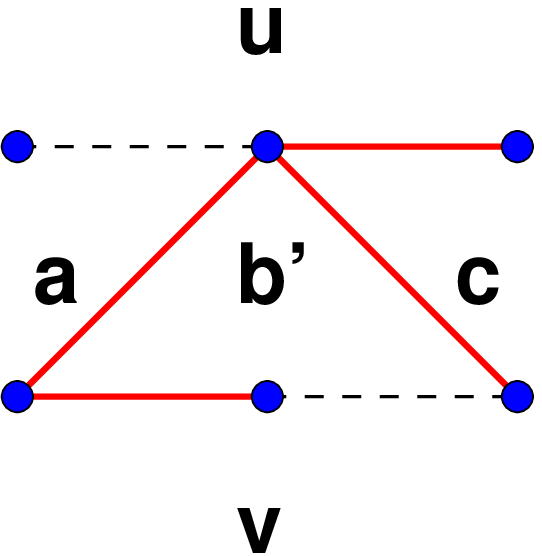}}}
\end{equation}

\subsection{Network Matrix}

\subsubsection{Definition}

Given some initial condition $X_{\bf k}$ as in \eqref{initcond}, define the $(r+1)\times(r+1)$ matrices 
\begin{eqnarray*}
N_{i,j}&=&\left\{ \begin{array}{ll} U_i(t_{i,j-1},t_{i,j},t_{i+1,j-1}), & {\rm if} \ k_{i,j-1}=k_{i,j}-1;\\
V_i(t_{i-1,j},t_{i,j-1},t_{i,j}) & {\rm otherwise}. \end{array}\right. \qquad (i\in [1,r];j\in \Z) \nonumber  \\
M_{0,j}&=& \mathbb I \nonumber \\
M_{i,j}&=& \left\{ \begin{array}{ll} N_{i,j} M_{i-1,j},&{\rm if}\ k_{i,j}=k_{i-1,j-1}, \ k_{i-1,j}\neq k_{i,j-1} ;\\
M_{i-1,j} N_{i,j} & {\rm otherwise.}
\end{array}\right.  \qquad (i\in [1,r];j\in \Z) \nonumber \\
N_j&=& M_{r,j}   
\end{eqnarray*}
The network matrix corresponding to the initial condition $X_{\bk}$ is
\begin{equation}\label{networkmatrix}
N(j_0,j_1)= \prod_{j=j_0+1}^{j_1} N_j \qquad (j_0 \leq j_1)
\end{equation}
with the convention that $t_{0,j}=t_{r+1,j}=1$ and 
$N(j,j)={\mathbb I}$. The order of multiplication in \eqref{networkmatrix} is according to increasing values of $j$.
We may think of $N(j_0,j_1)$ as a network matrix corresponding
to a slice of the initial data surface, containing the points $(i,j_0,k_{i,j_0}),(i,j_0+1,k_{i,j_0+1}),\cdots ,(i,j_1,k_{i,j_1})$ 
for $i\in [1,r]$. 

To make the definition more transparent, let us translate it in the language of the pictorial representation I above.
Each matrix $N_{i,j}$ corresponds to a lozenge made of two triangles (one grey, one white)
sharing the horizontal edge $(i,j-1,k_{i,j-1})-(i,j,k_{i,j})$. The grey triangle is above ($U$ matrix) if $k_{i,j}=k_{i,j-1}+1$,
below ($V$ matrix) if $k_{i,j}=k_{i,j-1}-1$. 
Moreover, the order in which the $\{N_{i,j}\}_{i\in[1,r]}$ are multiplied to form the ``slice" network matrix $N(j-1,j)$
exactly corresponds to a choice of diagonal in each square $(i-1,j-1,k_{i-1,j-1})-(i-1,j,k_{i-1,j})-(i,j-1,k_{i,j-1})-(i,j,k_{i,j})$,
for $i=2,3,...,r$,
with the rule that the diagonal should connect two opposite vertices with the {\it same} value of $k$. This gives rise to six possible vertical configurations of two lozenges:
\begin{equation}\label{rules}
\raisebox{-1.1cm}{\hbox{\epsfxsize=14.cm \epsfbox{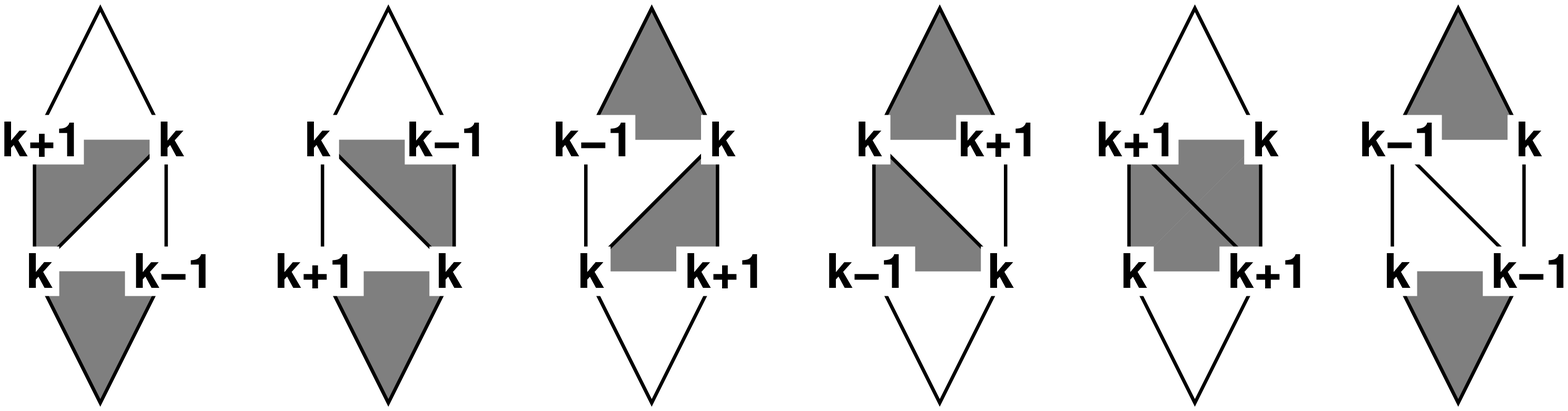}}} 
\end{equation}

Note that when both pairs
of diagonally opposite vertices have the same value of $k$, the choice of diagonal is not fixed. This ambiguity is immaterial, due to
the identities \eqref{repUVVU} and \eqref{repVUUV}. We have chosen the NW-SE diagonal by convention.
We call this construction the {\it $U,V$ decomposition} of the stepped surface $\bk$.

\begin{remark}
There is a direct bijection between the $U,V$ decomposition of the stepped surfaces $\bk$ and the 
quiver representing the $B$-matrix of the cluster algebra associated to the $T$-system \cite{DFK08}. 
Let $Q_\bk$ be
the (infinite) quiver encoded by the exchange matrix $B_\bk$ at the node labeled by 
$\bk$ in the cluster graph. We may represent $Q_\bk$ with its vertices
$(i,j)$ at the nodes of a square lattice $\Z\times \Z$ as a planar oriented graph with only square 
and triangular faces. Shading in grey the faces whose
edges are oriented couterclockwise yields a tessellation with white and grey squares and triangles
and corresponds to the $U,V$ decomposition of $\bk$ described above. In particular, the six face
configurations \eqref{rules} correspond to the six following local quiver configurations:
\begin{equation}\label{quivrules}
\raisebox{-1.5cm}{\hbox{\epsfxsize=14.cm \epsfbox{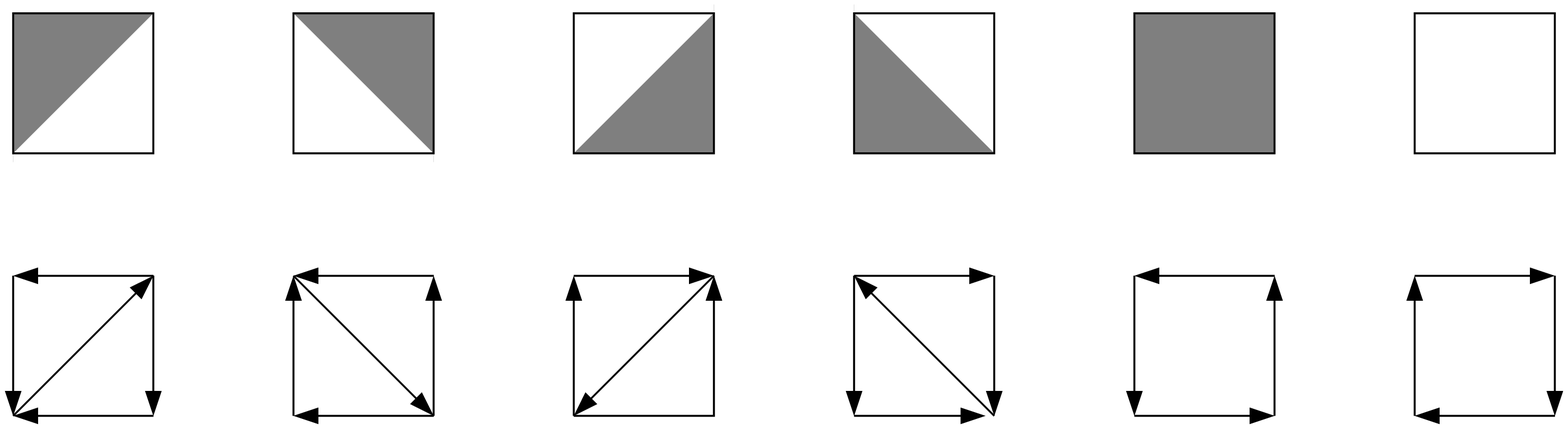}}} 
\end{equation}
\end{remark}

\begin{remark}\label{remrules}
For notational simplicity the rules \eqref{rules} are to be understood as seen from
{\it behind} the initial data surface (i.e. from the opposite side of the surface to where the point $(i,j,k)$ lies), 
namely from an observer sitting at a point $(i,j,k')$ with $k'<k_{i,j}$. This allows to read
expressions such as products of $U,V$ matrices from left to right.
\end{remark}

In the particular case of the fundamental stepped surface $\bk_0$, as $k_{i,j}^{(0)}\in \{0,1\}$, the network matrices simplify to
\begin{eqnarray}
N_{i,j}&=&\left\{ \begin{array}{ll} U_i(t_{i,j-1},t_{i,j},t_{i+1,j-1}), & {\rm if} \ i+j=0\ {\rm mod}\ 2; \\ 
V_i(t_{i-1,j},t_{i,j-1},t_{i,j}) & {\rm otherwise.} \end{array}\right.\, (i\in [1,r];j\in \Z) \nonumber  \\
N_j&=& \prod_{i=1}^r N_{i,j} ,\nonumber \\
N(j_0,j_1)&=& \prod_{j=j_0+1}^{j_1} N_j, \qquad (j_0 \leq j_1), \label{flatnetmat}
\end{eqnarray}
still with the convention that $t_{0,j}=t_{r+1,j}=1$ and 
$N(j,j)={\mathbb I}$.

The matrix $N(j,j')$ for the fundamental stepped surface $\bk_0$
(we choose odd $j<j'\in \Z$ in this example) is represented as follows:
$$\raisebox{-1.1cm}{\hbox{\epsfxsize=9.cm \epsfbox{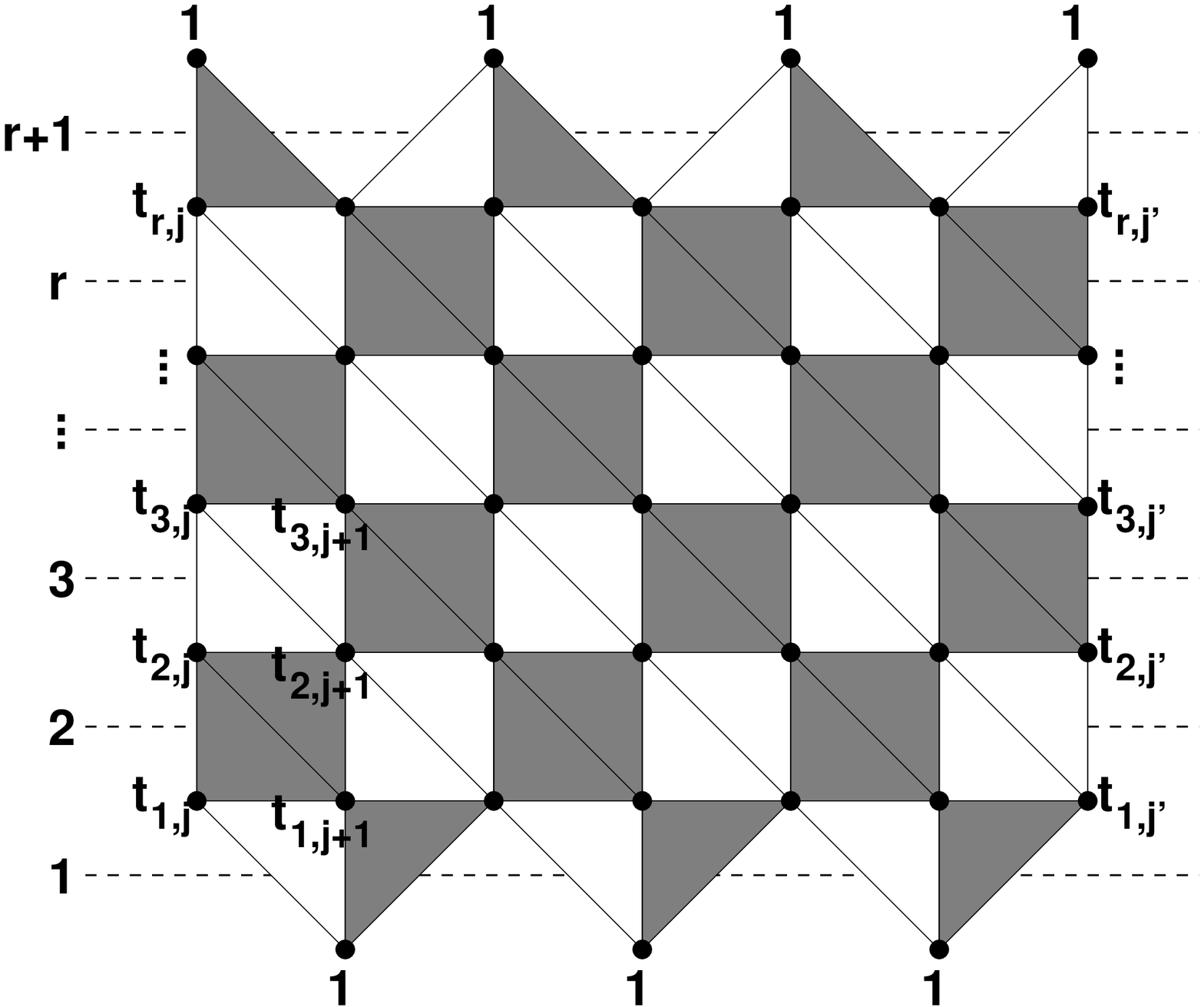}}} $$

This corresponds to the following portion of the cluster algebra quiver (strictly speaking the bottom and top row of 
fixed values $T_{0,j,k}=T_{r+1,j,k}=1$ are not part of the cluster, and the corresponding nodes are not
vertices of the quiver):
$$\raisebox{-1.1cm}{\hbox{\epsfxsize=8.cm \epsfbox{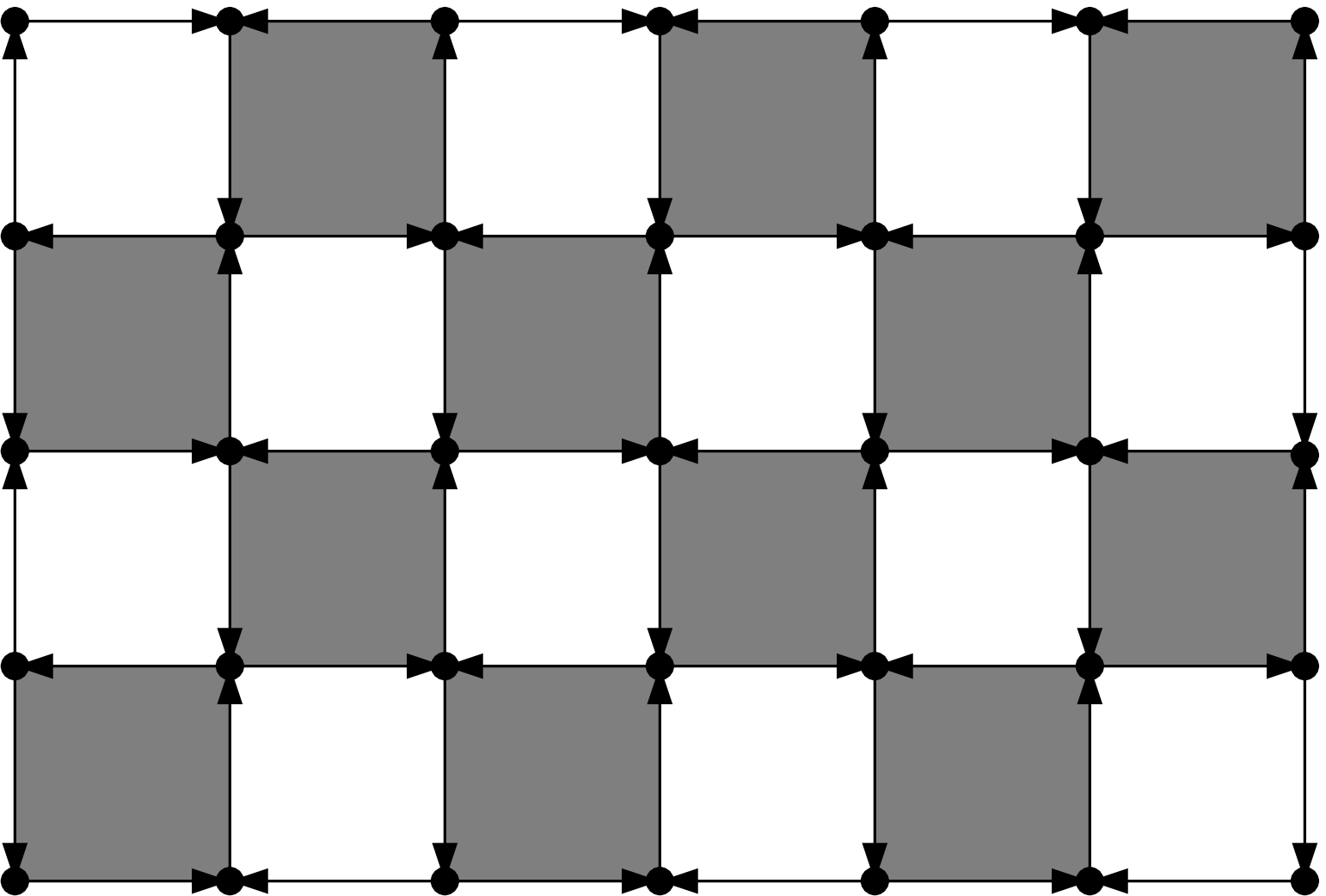}}} $$
In the network picture, the matrix   $N(j,j')$ is
$$\raisebox{-1.1cm}{\hbox{\epsfxsize=14.cm \epsfbox{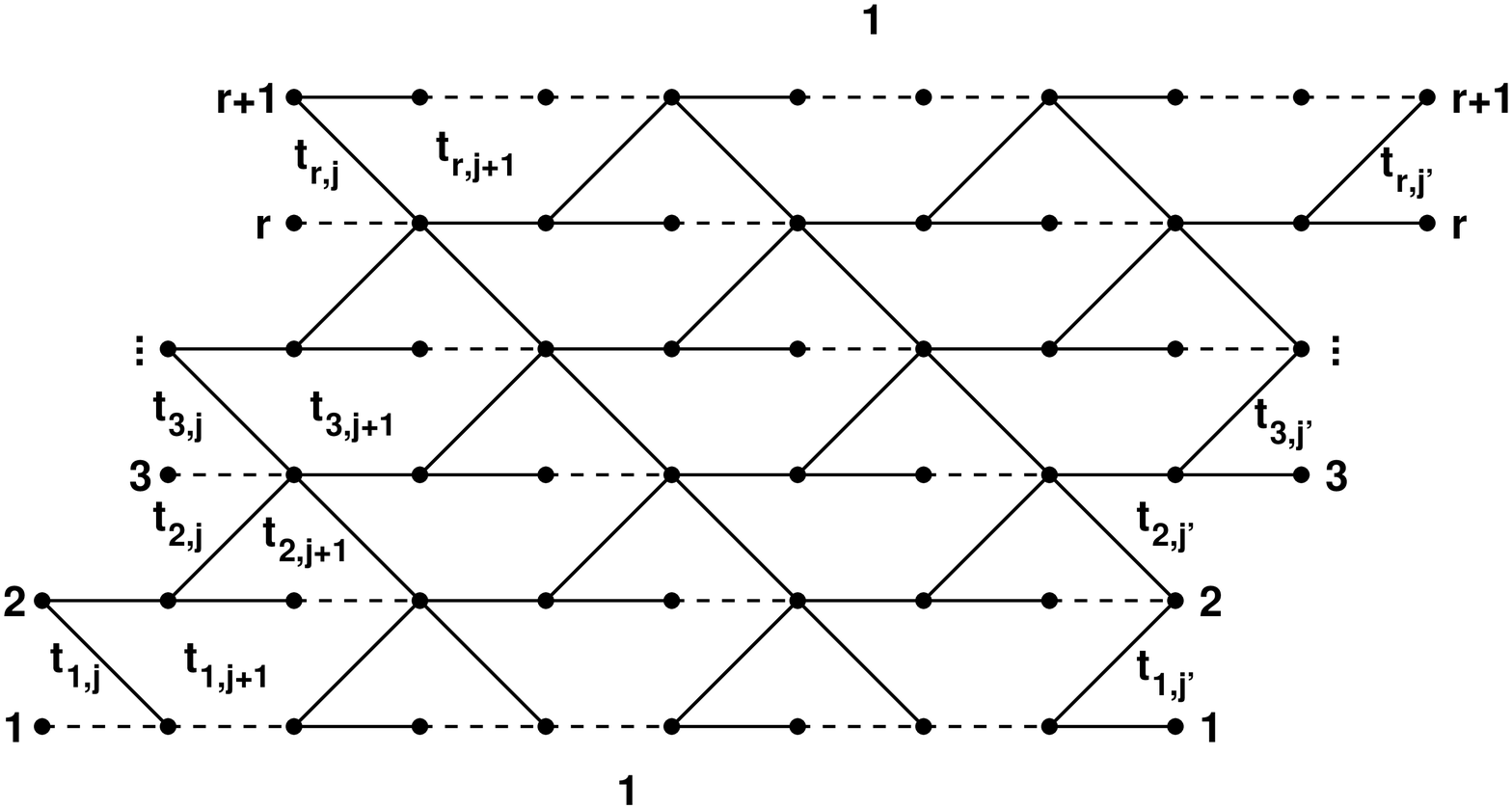}}} $$

An example of a non-flat stepped surface $\bk$, together with a pictorial
representation of type I of its network matrix, is
$$\raisebox{-1.1cm}{\hbox{\epsfxsize=16.cm \epsfbox{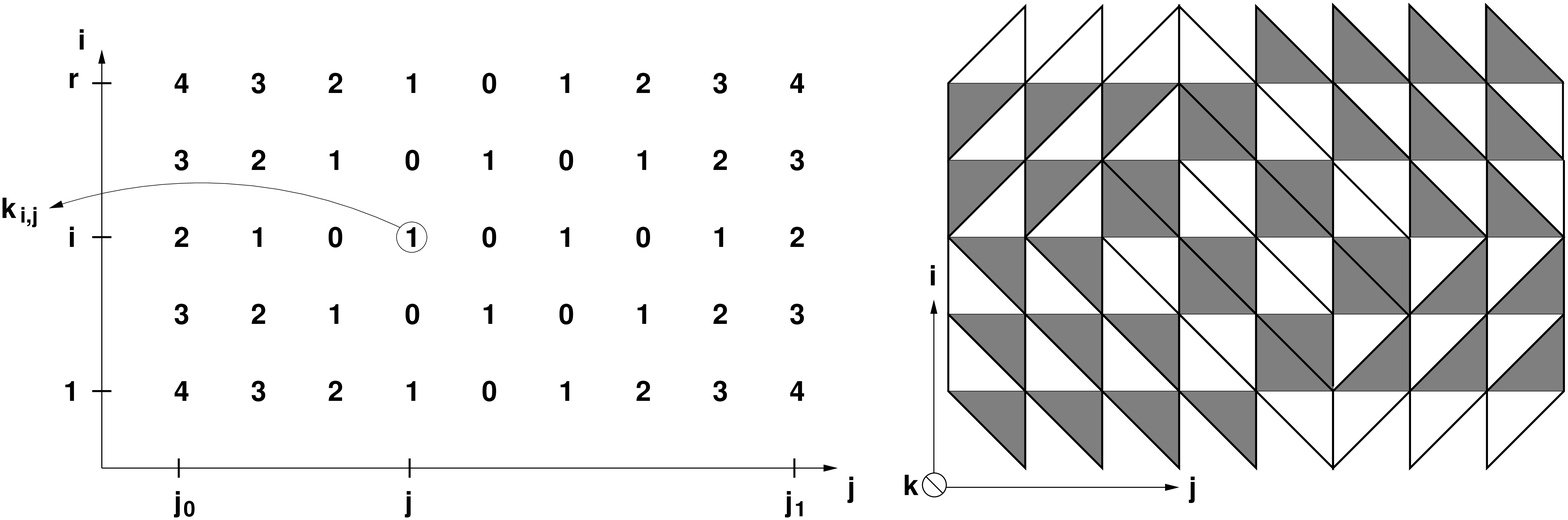}}} $$

\subsection{$A_r$ $T$-system solution}

Following \cite{DF}, we can now write an explicit expression for the variable $T_{1,j,k}$ 
in terms of the initial conditions $X_\bk$ for any stepped surface $\bk$. Without loss of generality,
we may assume that the point $(1,j,k)$ is above $\bk$, namely that $k\geq k_{1,j}$ (otherwise we simply
reflect $k\to -k$ and $\bk\to -\bk$).

\begin{defn}\label{projdef}
The projection of $(1,j,k)$
onto a given stepped surface $\bk=\{(i,j,k_{i,j})\}$ is the finite subset $\{(i,j,k_{i,j})_{i\in[1,r];j\in [j_0,j_1]}\}\subset \bk$. Here,  
$j_0$ is defined as largest integer such that  $$k-j=k_{1,j_0}-j_0,$$ and  $j_1$ defined as the smallest integer such that
 $$k+j=k_{1,j_1}+j_1.$$ We call $j_0$ the minimum of the projection, and $j_1$ its maximum.
\end{defn}

\begin{thm}{\cite{DF}}\label{soluT}
The solution $T_{1,j,k}$ of the $A_r$ $T$-system \eqref{tsys} in terms of the initial conditions $X_\bk$
on a given stepped surface $\bk$ is
\begin{equation}\label{Tsolu}
T_{1,j,k}=\left[N(j_0,j_1)\right]_{1,1}\, T_{1,j_1,k_{1,j_1}} 
\end{equation}
\end{thm}
For the proof, we refer the reader to \cite{DF}.

The expression for $T_{i,j,k}$ with $i>1$ is obtained from the Wronskian expressions \eqref{wronsk}. 
Combinatorially, this determinant is easily interpreted via the Lindstr\"om-Gessel-Viennot theorem \cite{LGV1,LGV2}
as the partition function of a family
of $i$  non-intersecting paths on the weighted network corresponding to $N(j_0,j_1)$, where $j_0$ and $j_1$
are respectively the smallest minimum and largest maximum of the projections of the $T_{1,j',k'}$ involved 
in the discrete Wronskian expression \eqref{wronsk},
namely with $(j',k')=(j+a-b,k+a+b-i-1)$, $a,b\in[1,i]$. 
\begin{thm}{\cite{DF}}\label{finsol}
As a function of $X_\bk$, the solution $T_{i,j,k}$ of the $A_r$ $T$-system is
$$T_{i,j,k}=Z_{j_0(1),...,j_0(i)}^{j_1(1),...,j_1(i)}(j_0,j_1)\, \prod_{a=1}^i T_{1,j_1(a),k_{1,j_1(a)}}, $$
where $Z_{j_0(1),...,j_0(i)}^{j_1(1),...,j_1(i)}(j_0,j_1)$ is the partition function  of $i$
non-intersecting weighted paths on the network corresponding to $N(j_0,j_1)$, starting at the points
$(1,j_0(1)),(1,j_0(2),...,(1,j_0(i))$ and ending at the points $(1,j_1(1)),(1,j_1(2),...,(1,j_1(i))$. These points are
respectively the minima
of the projections of the points $(1,j+b-i,k-b)$, $b=1,2,...,i$, with coordinates $(1,j_0(b),k_{1,j_0(b)})$, 
onto $\bk$, and the maxima
of the projections of the points $(1,j+a-1,k+a-i)$, $a=1,2,...,i$ onto $\bk$, 
with coordinates $(1,j_1(a),k_{1,j_1(a)})$.
\end{thm}

\section{Unrestricted $A_\infty$ $T$-system}\label{unsec}
In this section, we study the solutions $T_{i,j,k}$ of the octahedron equation or the unrestricted 
$A_\infty$ $T$-system \eqref{tsys}, not subject to the restriction \eqref{arboundary}. The idea is that for given $i,j,k$, the solutions of the $A_\infty$ system are given by those of some $A_r$ system for sufficiently large $r$.

We wish to write the solution explicitly in terms of initial conditions 
$X_\bk$ \eqref{initcond} specified along some 
stepped surface $\bk$ \eqref{stepsurf}.

As a preliminary remark, we note that the substitutions $k\to -k$ as well as $(i,j,k)\to (i+a,j+b,k+c)$
for any $a,b,c\in\Z$ with $a+b+c$ even leave the $T$-system equation \eqref{tsys} invariant:

\begin{lemma}\label{refk}
The solution $T_{i,j,k}$ of the unrestricted 
$A_\infty$ $T$-system \eqref{tsys} with initial conditions $X_\bk$ is the same function of the initial
values $\{t_{x,y}\}$ as $T_{i,j,-k}$ with initial condition $X_{-\bk}$, where by $-\bk$ we
mean the stepped surface $\bk'$ with $k_{i,j}'=-k_{i,j}$ for all $i,j$.
\end{lemma}

\begin{lemma}\label{transla}
The unrestricted $A_\infty$ $T$-system solution
$T_{i,j,k}$ with initial conditions $X_\bk$ is the same function of the initial values $t_{x,y}$ 
as $T_{i+a,j+b,k+c}$ is of the initial values $u_{x,y}=t_{x-a,y-b}$ for the initial conditions 
$X_{c+\bk}$, for any $a,b,c\in\Z$, such that $a+b+c=0$ mod 2, and where by $c+\bk$ we
mean the stepped surface $\bk'$ with $k_{i,j}'=k_{i,j}+c$ for all $i,j$.
\end{lemma}

As an immediate consequence of Lemmas \ref{refk} and \ref{transla},
we may assume without loss of generality
that the point $(i,j,k)$ is ``above" $\bk$, that is,  $k\geq k_{i,j}$, as all the results for 
$k\geq k_{i,j}$ may be transferred to the case $k<k_{i,j}$ as well. 

\begin{defn}
Let ${\mathcal D}_{\bk}(i,j,k)=\{(x,y,k_{x,y})\in \Z^3\, :\,  |x-i|+|y-j|\leq {|k-k_{x,y}|}\}\subset \bk$.
We call ${\mathcal D}_{\bk}(i,j,k)$ the {\it shadow} of the point $(i,j,k)$ on the initial data stepped surface $\bk$.
\end{defn}

Note that for $\bk=\bk_0$,
the boundary points $\partial{\mathcal D}_{\bk_0}(i,j,k)=\{(x,y,k_{x,y}^{(0)})\,: \, 
|x-i|+|y-j|={k-1}\}$ all have $k_{x,y}^{(0)}=1$.
For later purposes, we also define the interior domain
${\mathcal D'}_{\bk}(i,j,k)=\{(x,y,k_{x,y})\, :\,  |x-i|+|y-j|< |k-k_{x,y}|\}$. 

\subsection{Solution for the fundamental stepped surface $\bk_0$}

We start with the case when $\bk=\bk_0$. Let us consider a point $(i,j,k)$
with $i+j+k=0$ mod 2, so that  $k$ and $k_{i,j}^{(0)}\in\{0,1\}$ have the same parity. 
The following statement is clear from the form of the octahedron equation:
\begin{lemma}\label{redinit}
The solution $T_{i,j,k}$ of the unrestricted $A_\infty$ $T$-system \eqref{tsys} with initial 
conditions $X_{\bk_0}$
depends only on the initial values $t_{x,y}$ associated with points 
$(x,y,k_{x,y}^{(0)})\in \mathcal D_{\bk_0}(i,j,k)$.
\end{lemma}

Lemma \ref{transla} has the following immediate consequence:

\begin{lemma}\label{translatwo}
The unrestricted $A_\infty$ $T$-system solution
$T_{i,j,k}$ is the same function of the initial values $t_{x,y}$ on ${\mathcal D}_{\bk_0}(i,j,k)$
as $T_{i+a,j+b,k+c}$ is of $u_{x,y}=t_{x-a,y-b}$ on ${\mathcal D}_{c+\bk_0}(i+a,j+b,c+k)$, 
for any $a,b,c\in\Z$, such that $a+b+c=0$ mod 2.
\end{lemma}

In view of the above Lemmas, we may immerse the domain ${\mathcal D}_{\bk_0}(i,j,k)$ 
of initial data surface into a different initial data surface, pertaining to the $A_r$ case with
sufficiently large $r$, so that the domain does not feel the $A_r$ boundary. 
More precisely, using the above-mentioned translational invariance, we have
the following.

\begin{lemma}\label{chgsurf}
The solution $T_{i,j,k}(\{t_{x,y}\})$ of the $A_\infty$ $T$-system in terms of the initial values $t_{x,y}$
on ${\mathcal D}_{\bk_0}(i,j,k)$
coincides with the solution $T_{k,0,k}(\{u_{x,y}\})$ of the $A_r$ $T$-system, with $r=2k-1$, 
and with initial data $u_{x,y}$ on any stepped surface $\bk$ such that $k_{x,y}=x+y$ mod 2
for $(x,y,k_{x,y})\in {\mathcal D}_{\bk_0}(k,0,k)$,
on which $u_{x,y}=t_{x+i-k,y+j}$.
\end{lemma}
\begin{proof}
We use Lemma \ref{redinit} to compare the solution $T_{i,j,k}(\{t_{x,y}\})$ of the $A_\infty$ $T$-system
to that, $T_{k,0,k}(\{u_{x,y}\})$ of the $A_r$ $T$-system with $r=2k-1$. The latter only depends on
the values $u_{x,y}$ on the shadow of $(k,0,k)$ onto the stepped surface $\bk$, which was engineered to be 
${\mathcal D}_{\bk_0}(k,0,k)$. The lemma then follows from the 
translational invariance of Lemma \ref{translatwo},
with $a=k-i$, $b=-j$ and $c=0$.
\end{proof}

To compute the solution $T_{i,j,k}$ of the unrestricted $A_\infty$ $T$-system,
we simply have to compute the solution 
$T_{k,0,k}(\{u_{x,y}\})$ of the $A_{2k-1}$ $T$-system. 

\begin{defn}
The network matrix associated to the domain ${\mathcal D}_{\bk_0}(k,0,k)$, denoted by
$N\left({\mathcal D}_{\bk_0}(k,0,k)\right)$ is the product of the $2k-2\times 2k-2$ 
$U$ and $V$ matrices corresponding
to the $U,V$ decomposition of the domain ${\mathcal D}_{\bk_0}(k,0,k)$, according to the rules
of eq.\eqref{rules}.
\end{defn}

\begin{example}\label{troisex}
For $k=3$, we have in the pictorial representation I: 
\begin{eqnarray*}N\left({\mathcal D}_{\bk_0}(3,0,3)\right)&=& \quad 
\raisebox{-2.cm}{\hbox{\epsfxsize=4.cm \epsfbox{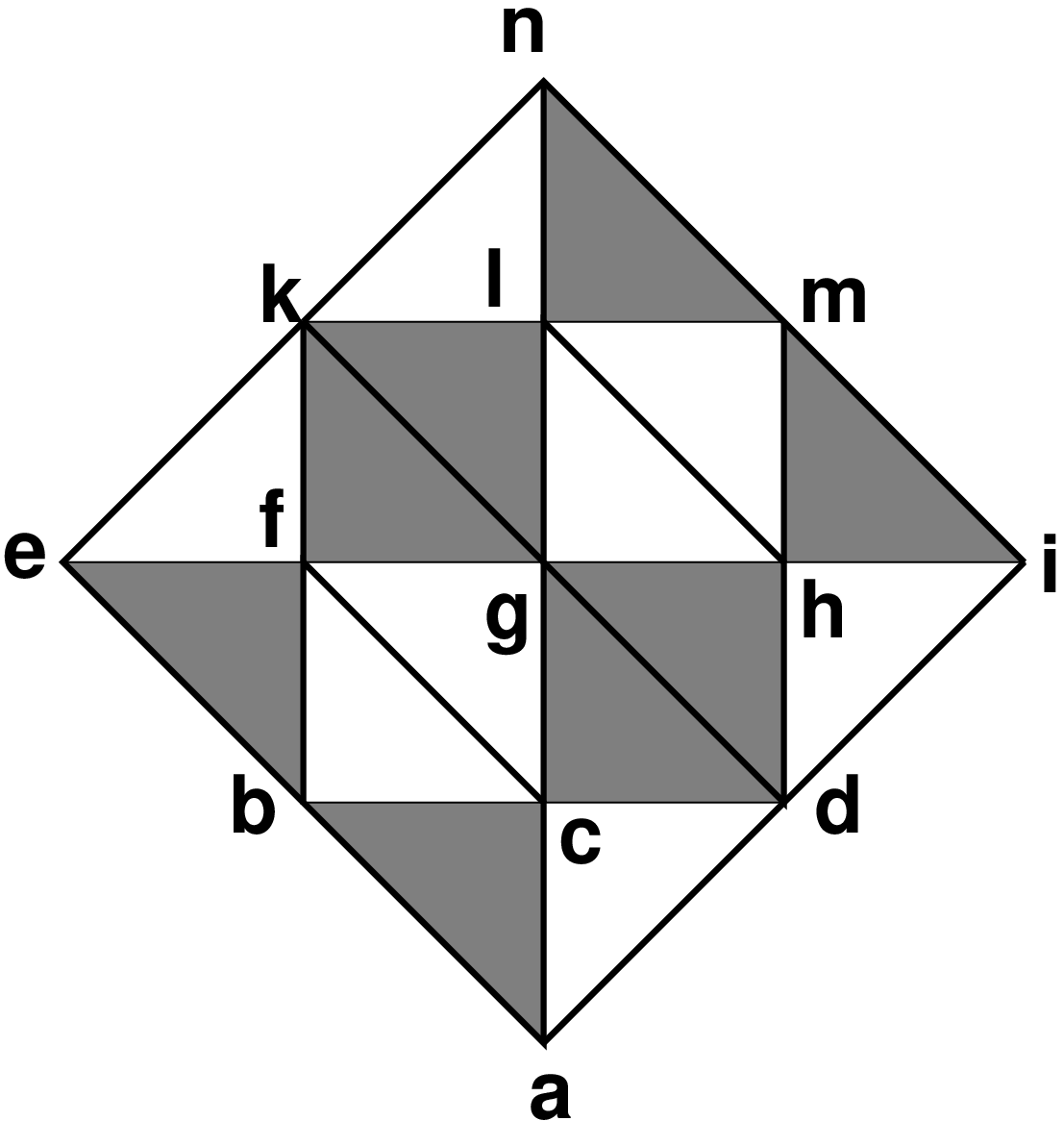}}}
\\
&=& V_{2}(b,e,f) V_{1}(a,b,c)U_{2}(f,g,j)V_{3}(g,k,l)\times \\
&&\qquad  \times U_{1}(c,d,g)V_{2}(d,g,h)U_{3}(l,m,n)U_{2}(h,i,m)
\end{eqnarray*}
where we have used shorthands for the variables
\begin{eqnarray*}
 u_{5,0}&=&n\\
u_{4,-1}=k \quad u_{4,0}&=&l \quad u_{4,1}=m\\
u_{3,-2}=e \quad u_{3,-1}=f\quad u_{3,0}&=&g 
\quad u_{3,1}=h\quad u_{3,2}=i\\
u_{2,-1}=b \quad u_{2,0}&=&c \quad u_{2,1}=d\\
u_{1,0}&=&a 
\end{eqnarray*}
\end{example}

We have the following.

\begin{thm}\label{kok}
The solution $T_{k,0,k}(\{u_{x,y}\})$ of the $A_{2k-1}$ $T$-system is given by:
$$ T_{k,0,k} =\left\vert N\left({\mathcal D}_{\bk_0}(k,0,k)\right)_{1,2,...,k-1}^{1,2,...,k-1} \right\vert
\ \prod_{a=1}^{k-1}u_{a,1-a}^{-1} \ \prod_{b=1}^{k}u_{b,b-1}$$
where for any matrix $M$ the notation $|M_{1,2...,m}^{1,2,...,m}|$ stands for the $m\times m$ principal minor of $M$.
\end{thm}
\begin{proof}
We apply Theorem \ref{finsol} to the case of the following particular stepped surface $\bk$.
We assume that $\bk$ satisfies the conditions of Lemma \ref{chgsurf}, and that, moreover,
outside of ${\mathcal D'}_{\bk_0}(k,0,k)$,
$k_{x,y}$ is a strictly increasing function of $y$ for $y\geq 0$ and strictly decreasing for $y\leq 0$,
while $u_{x,y}$ is arbitrary outside of ${\mathcal D}_{\bk_0}(k,0,k)$.
As $\bk$ and $\bk_0$
coincide along ${\mathcal D}_{\bk_0}(k,0,k)$ we have: 
\begin{equation}\label{formuT}
T_{k,0,k}(\{u_{x,y}\})=Z_{j_0(1),...,j_0(k)}^{j_1(1),...,j_1(k)}(j_0,j_1)\,
 \prod_{a=1}^{k}u_{1,j_1(a)}\end{equation}
where we have used the initial data $T_{1,j_1(a),k_{1,j_1(a)}}=u_{1,j_1(a)}$, and
$$\left\{ \begin{matrix}  j_0(a)=a-k & j_1(a)=a-1 \\
k_{1,j_0(a)}=a-k+1 & k_{1,j_1(a)}=a  \end{matrix}\right. \qquad (a=1,2,...,k) $$
with $j_0=1-k$ and $j_1=k-1$.
The relevant part of the network involved in the quantity 
$Z_{j_0(1),...,j_0(k)}^{j_1(1),...,j_1(k)}(j_0,j_1)$ is the rectangle corresponding to $N(j_0,j_1)$,
which reads in pictorial representations I and II:
$$\raisebox{-2.cm}{\hbox{\epsfxsize=14.cm \epsfbox{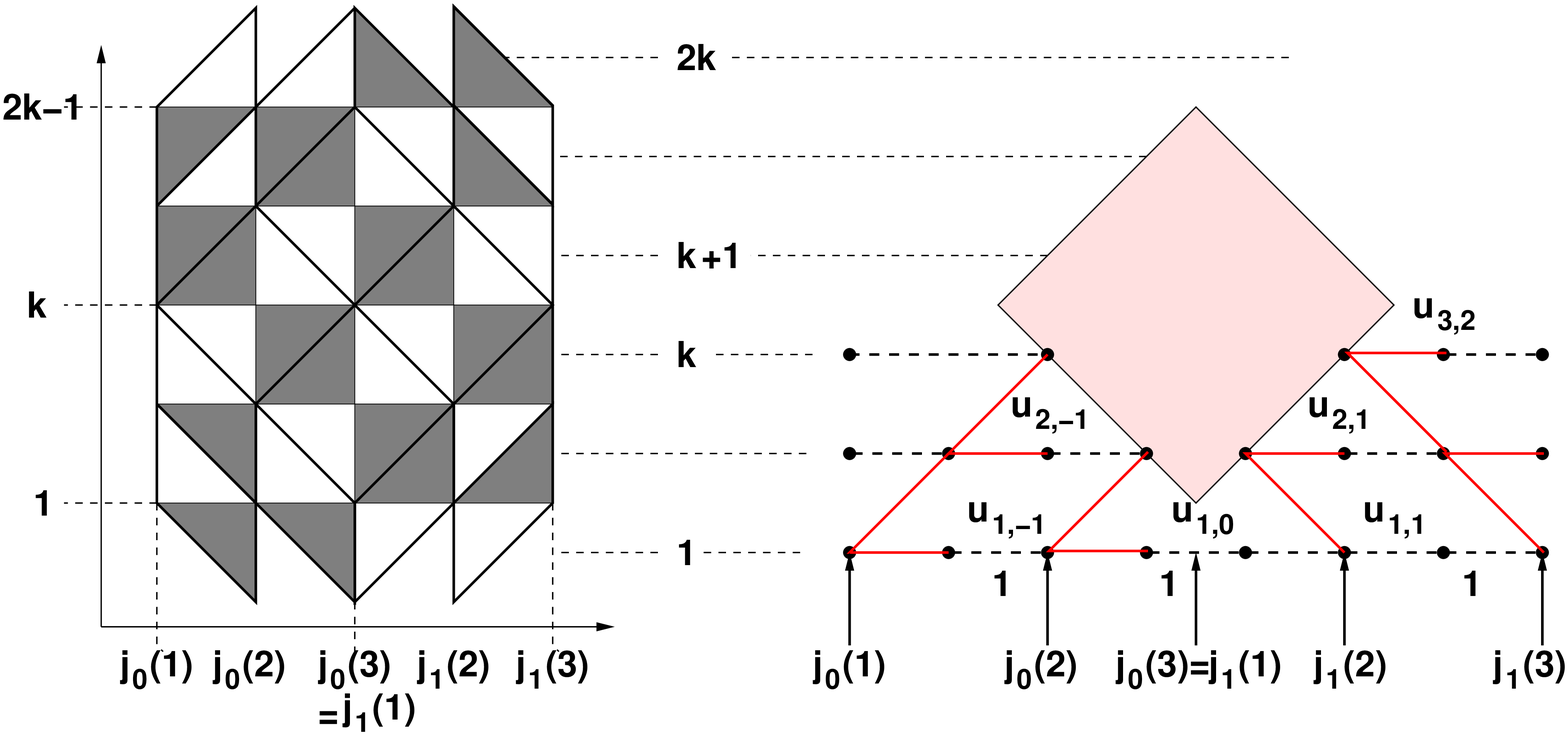}}}$$
where in the second picture we have represented the network for 
$N\left({\mathcal D}_{\bk_0}(k,0,k)\right)$ as a big lozenge with
entry and exit points determined by the unique configuration of non-intersecting paths outside of that domain,
that start at $j_0(a)$ and end at $j_1(k+1-a)$, $a=1,2,...,k$.
The steps of these paths outside of $N\left({\mathcal D}_{\bk_0}(k,0,k)\right)$
are all diagonal steps of the network (going up on the left side, and down on the right side). 
The contribution from the SW$\to$NE steps (on the left side) reads:
$$ \prod_{a=2}^{k} \prod_{m=1}^{k+1-a} {u_{m-1,j_0(a)}\over u_{m,j_0(a)} } =
\prod_{a=2}^{k} {1\over u_{k+1-a,j_0(a)}}
$$
while the NW$\to$SE steps (right side) contribute:
$$
 \prod_{a=1}^{k} \prod_{m=1}^{a} {u_{m+1,j_1(a)}\over u_{m,j_1(a)} } =
\prod_{a=1}^{k} {u_{a,j_1(a)}\over u_{1,j_1(a)}}
$$
Collecting all the weights and substituting them into \eqref{formuT}, the theorem follows, as, by
virtue of the Lindstr\"om-Gessel-Viennot theorem \cite{LGV1,LGV2}, the
quantity $\left\vert N\left({\mathcal D}_{\bk_0}(k,0,k)\right)_{1,2,...,k-1}^{1,2,...,k-1} \right\vert$
is the partition function for families of $k-1$ non-intersecting weighted
paths on the network for the above big lozenge,
starting at all $k-1$ points on the SW border, and ending at all $k-1$ points on the SE border.
\end{proof}

\begin{example}
We continue with the example $k=3$ of Ex.\ref{troisex}. 
We have
$$\vert D_{\bk_0}(3,0,3)\vert_{1,2}^{1,2}=\left \vert 
\begin{matrix}
{b d+a g\over c d} & {a(g i + d m)\over d h i} \\
{e g + b k\over d f} & {c(g i + d m) (e g + b k )\over d f g h i}+ {b(g n + m k)\over g il}
\end{matrix}\right\vert $$
The formula of Theorem \ref{kok} gives:
\begin{eqnarray*}T_{3,0,3}&=& \vert D_{\bk_0}(3,0,3)\vert_{1,2}^{1,2} \, {a d i\over a b} \\
&=& 
{b e g\over d f h} + {b e m\over f h i} + {b^2 k\over d f h} + {b^2 mk\over f g h i} + {b^2 n\over c i l} + 
{a b g n\over c d i l} +{a b m k\over c d il}+ {b^2 m k\over c g i l}
 \end{eqnarray*}
\end{example}

Theorem \ref{kok} has the following immediate consequence.
\begin{cor}\label{corTex}
The solution $T_{i,j,k}$ of the unrestricted $A_\infty$ $T$-system with initial conditions $X_{\bk_0}$ is given by:
\begin{equation}\label{Tsolut}
T_{i,j,k}=\left\vert N\left({\mathcal D}_{\bk_0}(i,j,k)\right)_{1,2,...,k-1}^{1,2,...,k-1} \right\vert
\ \prod_{a=1}^{k-2}t_{i+a-k+1,j-a}^{-1} \ \prod_{b=1}^{k-1}t_{i+b-k+1,j+b}
\end{equation}
\end{cor}

\subsection{Solution for an arbitrary stepped surface $\bk$}

Recall that ${\mathcal D}_{\bk}(i,j,k)$ denotes the shadow of $(i,j,k)$ on $\bk$, 
defined as the intersection of $\bk$ with the pyramid
$\Pi(i,j,k)=\{ (x,y,z) | |x-i|+|y-j|\leq |z-k| \}$. Using Lemma \ref{translatwo}, we may assume 
without loss of generality that ${\mathcal D}_{\bk}(i,j,k)$ is entirely above $\bk_0$. Indeed, 
$T_{i,j,m+k}$ is the same function of the initial data on $m+\bk$ as $T_{i,j,k}$ on $\bk$,
so we may pick $m$ large enough to ensure that $m+k_{i,j}\geq k_{i,j}^{(0)}$ on 
${\mathcal D}_{m+\bk}(i,j,m+k)$.

As explained before, any finite domain of $\bk$ above $\bk_0$
may be obtained by applying a finite number of forward
mutations $\mu_{i,j}$ to $\bk_0$. These correspond to a local transformation
of the surface, in which a vertex $(i,j,k_{i,j}=m-1)$ such that its four neighbors have 
$k_{i-1,j}=k_{i+1,j}=k_{i,j-1}=k_{i,j+1}=m$ is sent to the 6th vertex of the octahedron, 
$(i,j,k_{i,j}')$, with $k_{i,j}'=m+1$, as illustrated in \eqref{octa}.
If we complete ${\mathcal D}_{\bk}(i,j,k)$ with the faces of 
$\Pi(i,j,k)$ until they intersect $\bk_0$,
we obtain a domain $\Delta_\bk(i,j,k)$ that touches $\bk_0$ along the square 
$|x-i|+|y-j|=k-1$. The domain $\Delta_\bk(i,j,k)$
is obtained from ${\mathcal D}_{\bk_0}(i,j,k)$ by a finite number of forward 
mutations of the form $\mu_{x,y}$ with
$(x,y,k_{x,y}^{(0)})$ strictly inside ${\mathcal D}_{\bk_0}(i,j,k)$.

Starting from the expression of Corollary \ref{corTex}, we may implement these mutations by the 
corresponding $VU\leftrightarrow UV$ substitutions according to \eqref{muta}, 
as depicted in \eqref{repmuta}. These mutations are
directly applied on the matrix $N\left({\mathcal D}_{\bk_0}(i,j,k)\right)$, until the matrix is expressed as
$N\left(\Delta_\bk(i,j,k)\right)$. We have consequently:
\begin{equation}\label{debut}
T_{i,j,k}=\left\vert N\left(\Delta_\bk(i,j,k)\right)_{1,2,...,k-1}^{1,2,...,k-1} \right\vert
\ \prod_{a=1}^{k-2}t_{i+a-k+1,j-a}^{-1} \ \prod_{b=1}^{k-1}t_{i+b-k+1,j+b}
\end{equation}
We are left with the simple task of comparing $N\left(\Delta_\bk(i,j,k)\right)$ with $N\left({\mathcal D}_{\bk}(i,j,k)\right)$. 
Let us denote by $L_a=(i_a,j_a)$ and $R_a=(i_a',j_a')$, $a=1,2,...,\kappa$, 
the $(i,j)$ coordinates of the 
vertices of $\partial{\mathcal D}_{\bk}(i,j,k)\cap \Pi(i,j,k)$ with $(i_a,j_a)$ in the bottom left corner $i_a\leq i,j_a\leq j$
and $(i_a',j_a')$ in the bottom right corner $i_a'\leq i,j_a'\geq j$, labeled from bottom to top.
We have:

\begin{thm}\label{arbibound}
The solution $T_{i,j,k}$ of the unrestricted $A_\infty$ $T$-system with initial conditions $X_\bk$ reads:
\begin{equation}\label{theformula}
T_{i,j,k}=\left\vert N\left({\mathcal D}_{\bk}(i,j,k)\right)_{1,2,...,\kappa-1}^{1,2,...,\kappa-1} \right\vert
\ \prod_{a=1}^{\kappa-1}t_{L_a}^{-1} \ \prod_{b=1}^{\kappa}t_{R_b}
\end{equation}
\end{thm}
\begin{proof}
As $\Delta_\bk(i,j,k)$ is a completion of ${\mathcal D}_{\bk}(i,j,k)$ by use of the four faces 
of the pyramid $\Pi(i,j,k)$ until they reach $\bk_0$,
we have a simple pattern for the associated networks.
Here is an example, with $(i,j,k=4)$ and its shadow ${\mathcal D}_{\bk}(i,j,k)$ 
(shaded area) and domain $\Delta_\bk(i,j,k)$
(within the dashed square) for a typical stepped surface whose heights $k_{i,j}$ are displayed on the left diagram:
$$\raisebox{-2.cm}{\hbox{\epsfxsize=15.cm \epsfbox{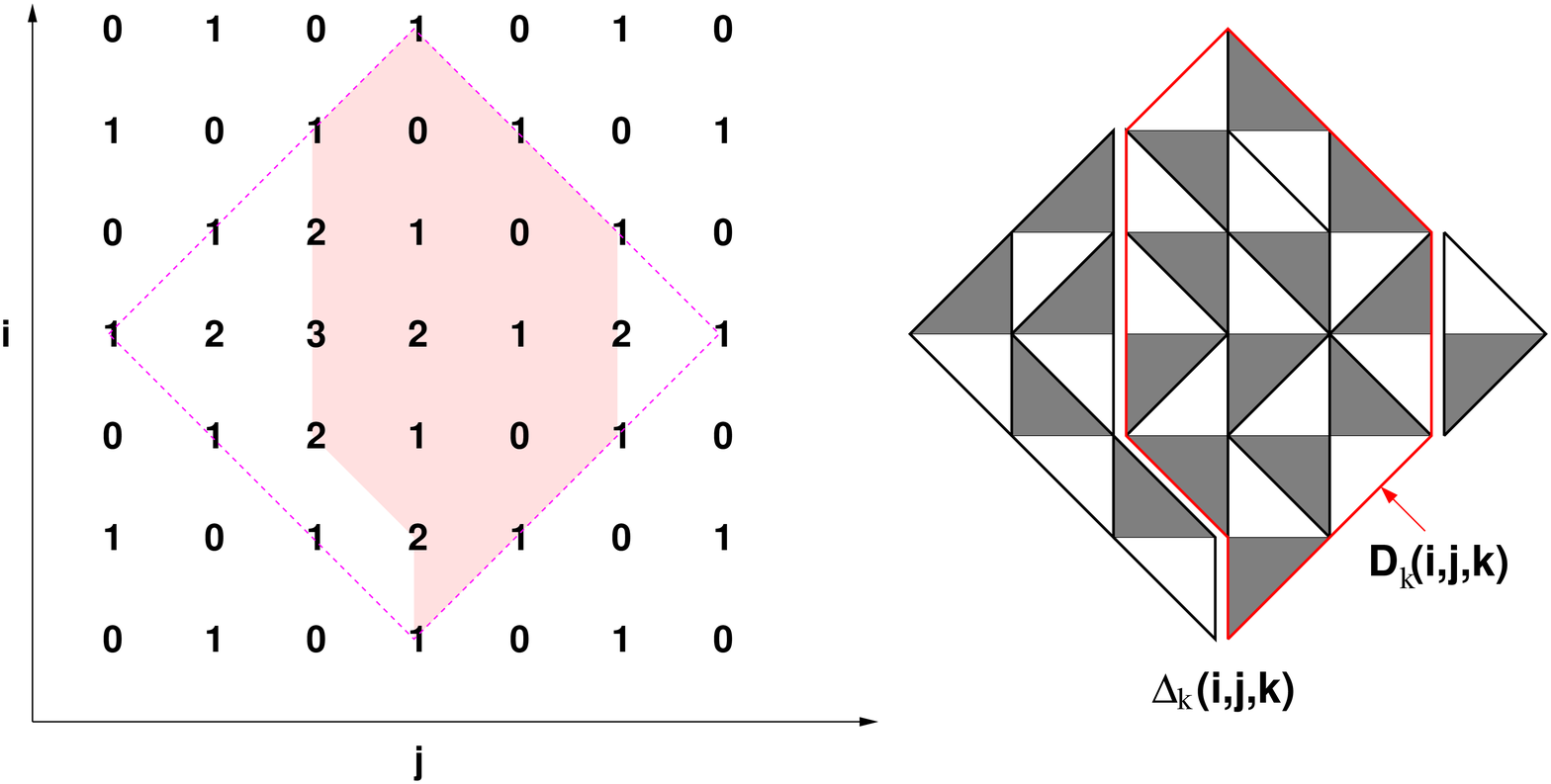}}}$$
We have depicted the corresponding matrix $N\left(\Delta_\bk(i,j,k)\right)$ on the right, while
$N\left({\mathcal D}_{\bk}(i,j,k)\right)$ corresponds to the smaller indicated domain, which matches the 
shaded domain on the left. Note that by construction the four corners between ${\mathcal D}_{\bk}(i,j,k)$ 
and $\Delta_\bk(i,j,k)$
are products of only $U$'s (W corners) or only $V$'s (E corners). The network pictorial representation is:
$$\raisebox{-2.cm}{\hbox{\epsfxsize=15.cm \epsfbox{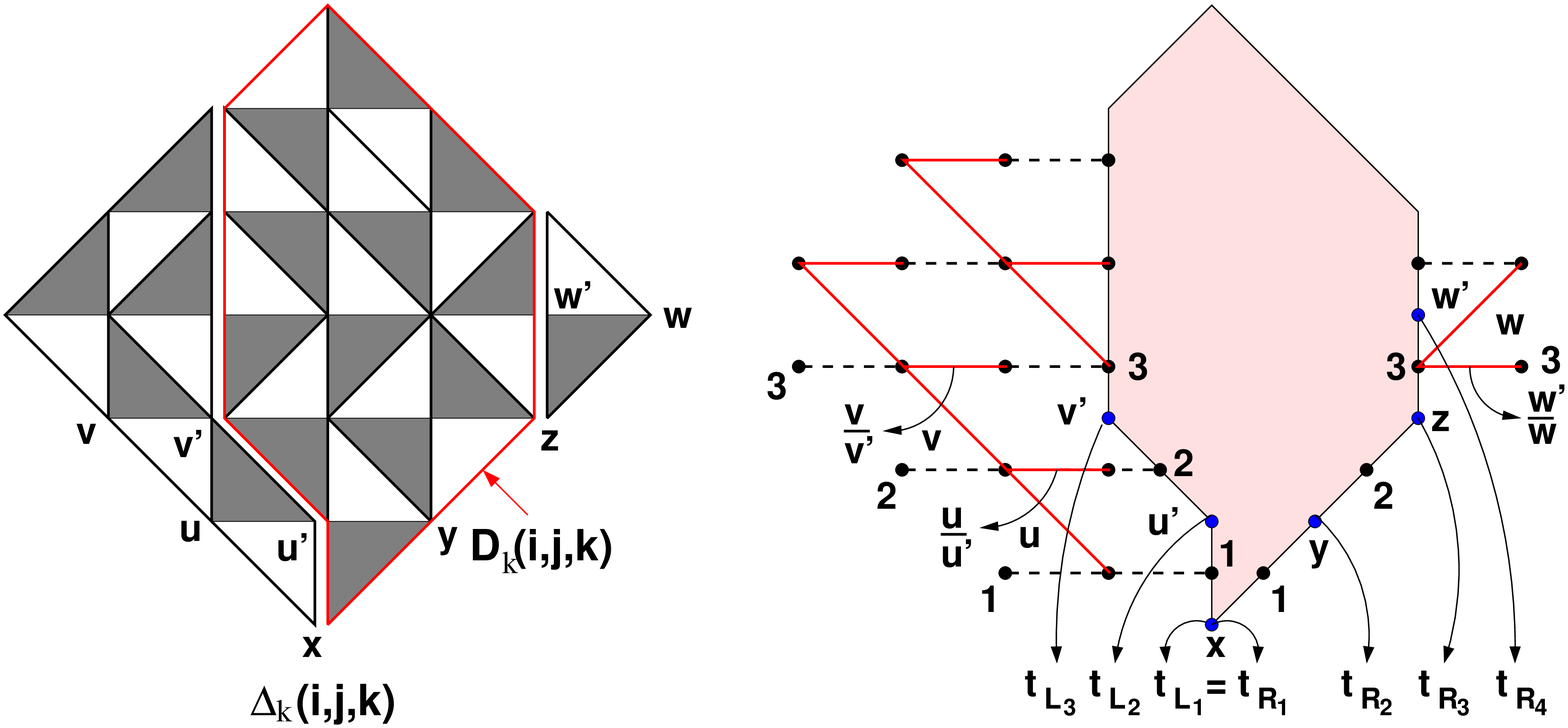}}}$$
The vertex labels correspond to the
actual initial data values, with $t_{L_1}=t_{R_1}=x$, $t_{L_2}=u'$, $t_{L_3}=v'$, $t_{R_2}=y$, $t_{R_3}=z$,
$t_{R_4}=w'$, while $\kappa=4$.
We see that the (non-intersecting) paths contributing to $N\left(\Delta_\bk(i,j,k)\right)$ must go along horizontal edges
throughout the domain $\Delta_\bk(i,j,k) \setminus {\mathcal D}_{\bk}(i,j,k)$, as they correspond to $U$ matrices (W side)
and $V$ matrices (E side). It is now easy to express 
$\left\vert N\left(\Delta_\bk(i,j,k)\right)_{1,2,...,k-1}^{1,2,...,k-1}\right\vert$ in terms
of $\left\vert N\left({\mathcal D}_{\bk}(i,j,k)\right)_{1,2,...,\kappa-1}^{1,2,...,\kappa-1}\right\vert$.
Collecting the contributions of the horizontal steps of these paths, in the form of ratios of face labels
along horizontals, all intermediate terms cancel out, leaving us with only the first and last one.
In the particular example above, the weights of the horizontal steps transform the prefactor for the W side:
${1\over u v x}$ into ${u\over u'}{v\over v'}{1\over u v x}={1\over u' v' x}$, while 
on the E side we have: $x y z w$ transformed into ${w'\over w} x y z w=x y z w'$. In general,
the net result is to replace the factors of $t^{-1}$'s and $t$'s in \eqref{debut} by the products of $t_{L_a}^{-1}$ 
and $t_{R_a}$, and the theorem follows.
\end{proof}

Theorem \ref{unposiT} is now an immediate corollary of Theorem \ref{arbibound}, as $N({\mathcal D}_{\bk}(i,j,k))$
is the matrix of a network with edge weights that are non-negative Laurent monomials of the initial data
$\{t_{i,j}\}$, and by the Lindstr\"om-Gessel-Viennot theorem \cite{LGV1,LGV2} the minor
$N\left({\mathcal D}_{\bk}(i,j,k)\right)_{1,2,...,\kappa-1}^{1,2,...,\kappa-1}$ is the partition function of families of
$\kappa-1$ weighted non-intersecting paths on the network graph, which is a polynomial of
the path weights with non-negative integer coefficients.

\begin{figure}
\centering
\includegraphics[width=14.cm]{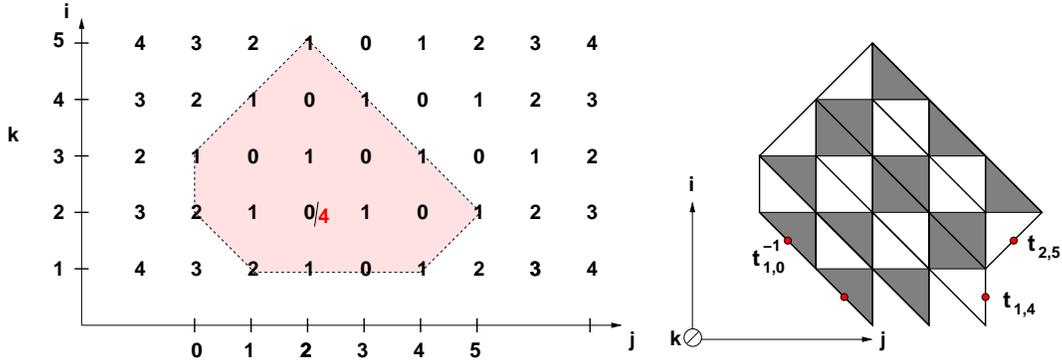}
\caption{A typical application of the truncation of formula \ref{theformula} in the $A_r$ case. We have
$r=5$, $(i,j,k)=(2,2,4)$, and the surface $\bk$ given on the left where we indicate the value of $k$ in the $(i,j)$
coordinate plane. The shaded area is the (truncated) shadow of $(i,j,k)$ on $\bk$. The formula expresses $T_{i,j,k}$
as the partition function for $i=2$ non-intersecting paths on the lattice $N({\mathcal D}_\bk(i,j,k))$
represented on the right, with the indicated prefactors.}
\label{fig:arshadow}
\end{figure}

\begin{remark}
A last remark is in order. In this section, we have used the known solution of the $A_r$ $T$-system
(Theorem \ref{finsol}) to derive the general  formula of Theorem \ref{arbibound} for the unrestricted
$A_\infty$ $T$-system solutions. We may now reverse the logic and extend the formula \eqref{theformula}
to the case of the $A_r$ $T$-system solutions, by viewing the $A_r$ $T$-system as a restriction
of the $A_\infty$ one obtained by impsing the extra $A_r$ boundary condition. 
This is easily done by noting that the $A_r$ boundary
simply truncates the space $(i,j,k)\Z^3$ to the domain $(i,j,k)\in [1,r]\times \Z^2$. Accordingly, the
initial data stepped surfaces $\bk$ are truncated to lie in the same domain, while the shadow
of any given point $(i,j,k)$ on $\bk$ is similarly truncated to a domain ${\mathcal D}_\bk(i,j,k)=\bk \cap \Pi(i,j,k)$.
The formula \eqref{theformula} remains valid with this new definition, while the left and right boundary points
range only over those within the truncated domain (see Fig.\ref{fig:arshadow} for an example). 
This gives a new direct formula for the solution of
the $A_r$ $T$-system which displays manifest Laurent positivity of the solution in terms of 
arbitrary initial data. 
\end{remark}

\section{$\ell$-restricted T-system: the $A_1$ case}\label{ellonesec}\label{warmupsec}

We study the solutions of the $T$-system with $\ell$-restricted boundaries.
For pedagogical reasons, this section is entirely devoted to the case of $A_1$, for which we will 
derive slightly more general results.

\subsection{The $\ell$-restricted $A_1$ $T$-system and its initial conditions}\label{boundasec}

The $A_1$ $T$-system is the $r=1$ version of \eqref{tsys}, with the simplified notation
$T_{1,j,k}=T_{j,k}$ for $j,k\in \Z$ and $j+k=1$ mod 2. Let $S$ be a subset of $\Z$.
We consider the $A_1$ $T$-system with the restriction that $j\in S$:
\begin{equation}\label{restazero}
T_{j,k+1}T_{j,k-1}=T_{j+1,k}T_{j-1,k}+1\qquad (j\in S;k\in \Z)\end{equation}
The general initial conditions for \eqref{restazero} are indexed by stepped surfaces $\bk$ \eqref{stepsurf},
which reduce here to paths 
$$\bk(S)=\{k_j(S)\in \Z,\ |k_{j+1}-k_j|=1\, {\rm and}\, k_j+j=1\,  {\rm mod} \, 2,\, j\in S\}.$$
We consider the system \eqref{restazero}
with possibly additional boundary conditions depending on $S$, 
and an initial condition $X^{S}(\bt)$, which is an assignment of  formal variables $\bt=(t_{j})_{j\in S}$ 
to the points on the surface $\bk(S)$.
We consider the four cases:

\begin{enumerate}[(i)]
\item
Unrestricted $A_1$ $T$-system: $S=\Z$.
The initial condition $X_\bk(\bt)$ is the assignment
\begin{equation}\label{boundari}
X_\bk(\bt): \Big\{ T_{j,k_j}=t_j\, (j\in \Z)\Big\}
\end{equation}
\item
Right half-plane $A_1$ $T$-system: $S=[1,\infty)$.
The additional
boundary conditions are
\begin{equation*}
T_{0,k}=1 \qquad  (k\in\Z) 
\end{equation*}
and the initial condition $X_\bk^+(\bt)$ is the assignment
\begin{equation}\label{boundarii}X_{\bk}^+(\bt): \Big\{ T_{j,k_j}=t_j\, (j\in [1,\infty))\Big\}
\end{equation}
\item
Left half-plane $A_1$ $T$-system: $S=(-\infty,\ell]$.
The additional
boundary conditions are
\begin{equation*}
T_{\ell+1,k}=1 \qquad  (k\in\Z) 
\end{equation*}
and the initial condition $X_{\bk}^-(\bt)$ is the assignment
\begin{equation}\label{boundariii}X_{\bk}^-(\bt): \Big\{ T_{j,k_j}=t_j\, (j\in (-\infty,\ell])\Big\}
\end{equation}
\item
$\ell$-restricted $A_1$ $T$-system: $S=[1,\ell]$.
The additional
boundary conditions are
\begin{equation*}
T_{0,k}=T_{\ell+1,k}=1 \qquad  (k\in\Z) 
\end{equation*}
and the initial condition $X_{\bk}^{[1,\ell]}(\bt)$ is the assignment
\begin{equation}\label{initgenone}X_{\bk}^{[1,\ell]}(\bt): \Big\{ T_{j,k_j}=t_j\, (j\in [1,\ell])\Big\}
\end{equation}
\end{enumerate}

In the following, we will also consider the unrestricted $A_1$ $T$-system (i)
with initial conditions $X_\bk(\bt)$ \eqref{initcondi}, further restricted by imposing extra conditions 
on the initial values $\bt=\{t_{j}\}$ as well as the path $\bk$.
The new initial values $\bt^{+},\bt^{-},\bt^{[1,\ell]}$ and
paths $\bk^{+},\bk^{-},\bk^{[1,\ell]}$ correspond respectively to the
following conditions:

\begin{equation}\label{otsymii}
\bt^{+}:\, \left\{ 
\begin{matrix}
&t_{-j-2}=-t_{i,j}\,  (j\in \Z_+)\\
&t_{0}=1,\,  t_{-1}=0 
\end{matrix} \right. \quad \bk^{+}:\, \left\{ 
\begin{matrix}
&k_{-j-2}=k_{j}\,  (j\in \Z_+)\\
&k_{-1}=k_0-1
\end{matrix} \right.
\end{equation}

\begin{equation}\label{otsymiii} 
\bt^{-}:\, \left\{ 
\begin{matrix}
&t_{j+\ell+3}=-t_{\ell+1-j}\, (j\in \Z_+)\\
&t_{i,\ell+1}=1,\,  t_{i,\ell+2}=0
\end{matrix} \right. \quad \bk^{-}:\, \left\{ 
\begin{matrix}
&k_{j+\ell+3}=k_{\ell+1-j}\,  (j\in \Z_+)\\
&k_{\ell+2}=k_{\ell+1}-1
\end{matrix} \right.
\end{equation}

\begin{equation} \label{otsymiv}
\bt^{[1,\ell]}: \, \left\{ \begin{matrix}
&t_{-j-2}=-t_{i,j}\, (j\in \Z_+)\\
&t_{2(\ell+3)+j}=t_{i,j}\,  (j\in \Z)\\
&t_{0}=t_{\ell+1}=1,\,  t_{-1}=0
\end{matrix} \right. \quad  \bk^{[1,\ell]}:\, \left\{ 
\begin{matrix}
&k_{-j-2}=k_j\, (j\in \Z_+)\\
&k_{j+2(\ell+3})=k_{j}\, (j\in \Z)\\
&k_{-1}=k_0-1,\  k_{\ell+2}=k_{\ell+1}-1
\end{matrix} \right.
\end{equation}

By convention, when $\bk=\bk_0$ we drop the requirements on $\bk$.
We note that the conditions \eqref{otsymiv} are equivalent to imposing
{\it simultaneously} the conditions \eqref{otsymii} and \eqref{otsymiii}.

We wish to study the solutions $T_{j,k}$ of the $A_1$ $T$-system in terms of initial conditions
in all of the above cases $(i-iv)$.
By virtue of Lemmas \ref{refk} and \ref{transla} we may without loss of generality restrict ourselves 
to points $(j,k)$ above the initial data surface $\bk$ in all these cases.

\subsection{Unrestricted system solution}
The unrestricted system subject to initial conditions $X_\bk(\bt)$
reads:
\begin{eqnarray}
T_{j,k+1}T_{j,k-1}&=&T_{j+1,k}T_{j-1,k}+1\qquad (j,k\in \Z;j+k=1\, {\rm mod}\, 2)\label{toneA}\\
T_{j,k_j}&=&t_{j} \qquad \qquad  (j\in \Z)\label{ttwoA}
\end{eqnarray}

Its solution is simply expressed in terms of the following simplified versions of
the $U,V$ matrices of Sect. \ref{UVsec} defined as:
\begin{equation} U(a,b)=U(a,b,1)=\begin{pmatrix} 1 & 0\\ b^{-1} & a b^{-1}\end{pmatrix} 
\qquad V(a,b)=V(1,a,b)=\begin{pmatrix}a b^{-1} & b^{-1}\\ 0 & 1\end{pmatrix}.
\end{equation}

Let us consider a point $(j,k)$ above the path $\bk$, i.e. with $k\geq k_j$. 
Def.\,\ref{projdef} for $r=1$ defines the projection
of the point $(j,k)$ onto $\bk$ as the portion of the path 
$(j,k_j)_{j\in[j_0,j_1]}$ with largest $j_0$ and smallest $j_1$
such that $ k-k_{j_0}=j-j_0, k-k_{j_1}=j_1-j $. Note that $j_0$ and $j_1$ are both even integers.
The {\it cone} of projection of $(j,k)$ is defined by the two lines $k=j+k_{j_0}-j_0$ and $k=k_{j_1}+j_1-j$.

We define the matrix
\begin{equation}\label{defM} 
M_{j}(t_j,t_{j+1})=\left\{ \begin{matrix} 
V(t_j,t_{j+1}) & {\rm if}\ \ k_j=k_{j+1}+1 \\
U(t_j,t_{j+1}) & {\rm if} \ \ k_j=k_{j+1}-1
\end{matrix} \right. 
\end{equation}
We have:
\begin{thm}{\cite{FRISES},\cite{DF}}\label{unrestaone}
The solution $T_{j,k}$ of the system (\ref{toneA}-\ref{ttwoA}) is:
\begin{equation}\label{soloneUV}
T_{j,k}=\Big( \prod_{j=j_0}^{j_1-1} M_{j,j+1}(t_j,t_{j+1})\Big)_{1,1} t_{j_1}
\end{equation}
\end{thm}

Note that this is the $A_1$ version of \eqref{Tsolu}, in which the $2\times 2$ network matrix
$N(j_0,j_1)$ is identified with the $2\times 2$ matrix
product $\prod_{j=j_0}^{j_1-1} M_{j,j+1}(t_j,t_{j+1})$.

The exact solution of Theorem \ref{unrestaone}
was used previously to derive the positive Laurent property for the solution
of the $T$-system, namely that $T_{j,k}$ is a Laurent polynomial of the initial data, with non-negative
integer coefficients. (This is clear from Theorem \ref{unrestaone}, as the entries of $U,V$ are 
themselves Laurent monomials of the initial data with non-negative integer coefficients.).

\subsection{$\ell$-restricted case: equivalent initial data and main theorems}

We now turn to solutions of the $\ell$-restricted system.
The main idea is to realize the $\ell$-restricted boundaries 
within the framework of the unrestricted 
$T$-system, by suitably engineering the initial data $t_k$, $k\in \Z$. 
The following three theorems will be proved in next section.

\begin{thm}\label{periodone}
The solution $T_{j,k}$ of the $\ell$-restricted $A_1$ $T$-system (iv) is periodic in 
the direction $k$:
$$ T_{j,k+N}=T_{j,k} $$
with period $N=2(\ell+3)$.
\end{thm}

\begin{thm}\label{ones}
The solution of the unrestricted $A_1$ $T$-system (i)
with initial conditions $X_{\bk_0}(\bt^{[1,\ell]})$ (\ref{boundari},\ref{otsymiv})
restricts to the solution of the $\ell$-restricted $A_1$ $T$-system (iv)
with the initial conditions $X_{\bm}^{[1,\ell]}(\bu)$ \eqref{initgenone}, 
where $\bu=\bt([1,\ell]),\bm=\bk_0([1,\ell])$ are
the restrictions of $\bt,\bk_0$ to the interval $j\in [1,\ell]$. 
As such, the solution of the $\ell$-restricted $A_1$ $T$-system with initial conditions 
$X_{\bk_0}^{[1,\ell]}(\bt)$
is a positive Laurent polynomial of the initial values $t_1,t_2,...,t_\ell$.
\end{thm}

\begin{thm}\label{twos}
The solution of the $\ell$-restricted $A_1$ $T$-system (iv)
with initial conditions $X_{\bk}(\bt)$ \eqref{initgenone}
along an arbitrary finite path $\bk$
is a positive Laurent polynomial of the initial values $t_1,t_2,...,t_\ell$.
\end{thm}

\subsection{Half-plane solution}

To prove Theorem \ref{ones}, we must show that the $\ell$-restricted boundary is implemented
by the choice of symmetries of the initial data. Concretely, one must show that both
$T_{0,k}=1$ and $T_{\ell+1,k}=1$ as a consequence. It turns out to be instructive to
first consider the case of the $T$-system \eqref{restazero} in a half-plane.
We have:

\begin{thm}\label{leftbounone}
The solution of the unrestricted $A_1$ $T$-system (i)
with initial conditions $X_{\bk_0}(\bt^{+})$  (\ref{boundari},\ref{otsymii})
restricts to that of the right half-plane $A_1$ $T$-system (ii)
with initial condition $X_{\bm}^+(\bu)$ \eqref{boundarii}, where $\bu=\bt([1,\infty)),\bm=\bk_0([1,\infty))$ 
are the restrictions of $\bt,\bk_0$ to  the range $j\in [1,\infty)$.
As such the solutions of the latter are positive Laurent polynomials of the initial data $t_1,t_2,t_3,...$
\end{thm}
\begin{proof} To prove the first statement of the theorem, it is sufficient to show that $T_{0,2k+1}=1$
for all $k\geq 0$ (the case $k<0$ follows from the general reflection 
symmetry argument of Lemmas \ref{refk} and \ref{transla}).  
Indeed, the half-plane solution is uniquely determined in terms of initial conditions of the type 
\eqref{boundarii}, so it must coincide with that of the unrestricted system in the range $j\geq 0$, 
once the boundary condition $T_{0,k}=1$ is guaranteed.

To compute $T_{0,2k+1}$, we wish to use Theorem \ref{unrestaone}, but we cannot plug directly the value $t_{-1}=0$
as some entries of the matrices $U,V$ may diverge. However, only combinations of the form 
$V(t_j,t_{j+1})U(t_{j+1}t_{j+2})$
for even $j$ enters the solution \eqref{soloneUV}. We simply note that
\begin{equation}\label{limP}
\lim_{\epsilon\to 0} V(-1,\epsilon)U(\epsilon,1) =P=\begin{pmatrix} 0 & 1 \\ 1 & 0 \end{pmatrix}
\end{equation}
Provided we take this limit, we may now safely use the formula \eqref{soloneUV} for $T_{0,2k+1}$, for $k\geq 0$
with $j_0=-2k+2$ and $j_1=2k$:
\begin{eqnarray*}T_{0,2k+1} &=& \\
&&\!\!\!\!\!\!\!\!\!\!\!\!\!\!\!\!\!\!\!\!\!\!\!\! \left(V(-t_{2k-2},-t_{2k-3}) \cdots V(-t_2,-t_1)U(-t_1,-1)
P V(1,t_1)U(t_1,t_2)\cdots U(t_{2k-1},t_{2k})\right)_{1,1} t_{2k} 
\end{eqnarray*}
(Here and in the following the $\cdots$ stand for alternating products of $UVUVU...$).
Next, we shall use the following ``collapse" properties of $U,V,P$ matrices, easily derived by direct calculation:
\begin{equation}\label{colla} U(-b,-a)PV(a,b)= P \qquad V(-b,-a)P U(a,b)=P \end{equation}
for all $a,b$.
Applying these iteratively to \eqref{soloneUV} implies:
$$ T_{0,2k+1} =\left(P V(t_{2k-2},t_{2k-1})U(t_{2k-1},t_{2k})\right)_{1,1}t_{2k}={1\over t_{2k}}\times t_{2k}=1$$

We now turn to the Laurent positivity of the solution. Let us compute $T_{j,k}$ for $k\geq 0$ via \eqref{soloneUV}. If
$j_0\geq 0$, this is the same as the solution of the unrestricted system, and the positivity is clear. 
Otherwise, let us denote by ${\bar j}_0=-j_0-2\geq 0$, and compute:
\begin{eqnarray*}
 T_{j,k}&=& \left( V(t_{j_0},t_{j_0+1}) \cdots U(t_{-3},-1)P V(1,t_1) \cdots U(t_{j_1-1},t_{j_1})\right)_{1,1}t_{j_1}\\
 &=&\left( V(-t_{{\bar j}_0},-t_{{\bar j}_0-1}) \cdots U(-t_1,-1)P V(1,t_1) \cdots U(t_{j_1-1},t_{j_1})\right)_{1,1}t_{j_1}\\
 &=&\left( P V(t_{{\bar j}_0},t_{{\bar j}_0+1}) \cdots U(t_{j_1-1},t_{j_1})\right)_{1,1}t_{j_1}
\end{eqnarray*}
where we have used \eqref{colla} repeatedly to eliminate the first ${\bar j}_0$ terms.
This is a product of matrices with entries that are all Laurent monomials of the initial data $(t_j)_{j\geq 1}$
with non-negative integer coefficients. The positive Laurent property follows.
 \end{proof}

\begin{figure}
\centering
\includegraphics[width=8.cm]{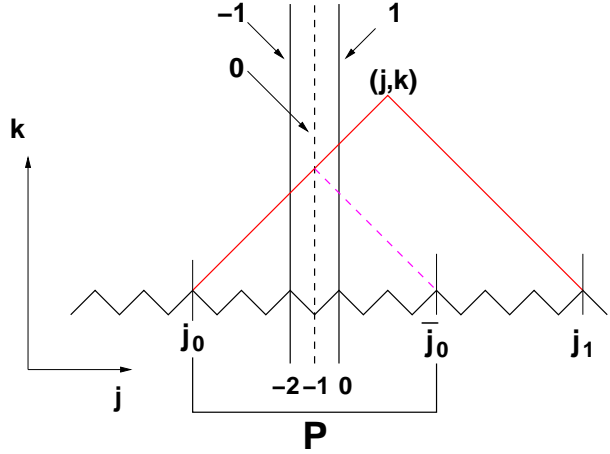}
\caption{The segment $[j_0,j_1]$ of values necessary to express $T_{j,k}$ has its left endpoint reflected by the line $j=-1$.
As a result, only the values of $t_j$ for $j$ between ${\bar j}_0=-j_0-2$ and $j_1$ enter the expression.}
\label{fig:refwall}
\end{figure}

\begin{remark}\label{reflerem}
There is a very simple pictorial interpretation of the computation of $T_{j,k}$. 
The left endpoint of the segment $j\in [j_0,j_1]$ 
of initial values $t_j$ necessary to express
the solution $T_{j,k}$ has been reflected by the line $j=-1$. This is depicted in Fig. \ref{fig:refwall}, along with the 
corresponding cone of projection of $(j,k)$ and its reflection.
\end{remark}

We have the following analogous result for the left half-plane $j\leq \ell$ solution:
 
\begin{thm} \label{rightbounone}
For fixed $\ell\in\Z_{>0}$, 
the solution of the unrestricted $A_1$ $T$-system (i)
with initial conditions $X_{\bk_0}(\bt^{-})$ (\ref{boundari},\ref{otsymiii}),
restricts to that of the left half-plane $A_1$ $T$-system (iii)
with initial condition $X_{\bm}^-(\bu)$ \eqref{boundariii}, where $\bu=\bt((-\infty,\ell]), \bm
=\bk_0((-\infty,\ell])$ 
are the restrictions of 
$\bt,\bk_0$ to the range $j\in (-\infty,\ell]$.
As such the solutions of the latter are positive Laurent polynomials of the 
initial data $t_\ell,t_{\ell-1},t_{\ell-2},...$
\end{thm}
\begin{proof} 
Let us first show the positivity statement.
Imitating the proof of Theorem \ref{leftbounone}, we must ``regularize" the singular value $0$ by introducing:
$$\lim_{\epsilon\to 0} V(1,\epsilon)U(\epsilon,-1)= -P$$
A new feature arises when $\ell$ is even: in that case, the boundary contribution is
$$\lim_{\epsilon\to 0} U(1,\epsilon)V(\epsilon,-1)= -P$$
as well, but if $j+k=\ell+3$, the formula \eqref{soloneUV} for $T_{j,k}$ contains a potential
singularity as $U(1,\epsilon)$
diverges when $\epsilon\to 0$. Fortunately the full formula also has an $\epsilon=T_{j_1,k_{j_1}}$ in factor, leading to
a finite limit:
\begin{eqnarray}
\ \ \ \ \ \ \ \  T_{j,k}&=&\lim_{\epsilon\to 0}\Big( V(t_{j-k},t_{j-k+1})\cdots V(t_{\ell},1)U(1,\epsilon) \Big)_{1,1}\epsilon\nonumber \\
&=&\Big( V(t_{j-k},t_{j-k+1})\cdots V(t_{\ell},1) \Big)_{1,2}=
\Big( V(t_{j-k},t_{j-k+1})\cdots V(t_{\ell},1)(-P )\Big)_{1,1}(-1)\label{tricky}
\end{eqnarray}
which is manifestly positive.
In general, we compute $T_{j,k}$ via the formula \eqref{soloneUV}. Again, if $j+k\leq \ell+1$, the solution is the same as
in the unrestricted case, and positivity follows. Otherwise, we have a reflection of the segment of initial values
on the right against the line $j=\ell+2$. 
More precisely, denoting by $M_\ell=U$ if $\ell$ is odd and $M_\ell=V$ if $\ell$ is even,
we get:
\begin{eqnarray}
\ \ \ \ \ \  T_{j,k}&=&\Big(V(t_{j-k},t_{j-k+1})\cdots M_\ell(t_{\ell},t_{\ell+1}) (-P) M_{\ell+1}(t_{\ell+2},t_{\ell+3})\cdots
U(t_{j+k-1},t_{j+k})\Big)_{1,1} t_{j+k}\nonumber \\
&=& \Big(V(t_{j-k},t_{j-k+1})\cdots U(t_{2(\ell+2)-j-k-1},t_{2(\ell+2)-j-k})  (-P) \Big)_{1,1} (-t_{2(\ell+2)-j-k})
\label{normal}
\end{eqnarray}
where we have used the symmetry $t_{2(\ell+2-j)}=-t_j$ and  \eqref{colla} to cancel out terms 
on both sides of the $(-P)$ factor. The two minus signs cancel, and we are left with 
a manifestly positive expression. 
Let us now turn to the first part of the theorem. By uniqueness of the solution in the left half-plane, we simply have
to show that $T_{\ell+1,k}=1$ for all $k\in \Z_{>0}$ such that $k+\ell$ is even.
For odd $\ell$, using \eqref{tricky} we first compute:
$T_{\ell+1,2}=\lim_{\epsilon\to 0}(V(t_\ell,1)U(1,\epsilon))_{1,1}\epsilon=1$. For all other cases, we use
\eqref{normal} to compute:
\begin{eqnarray*}T_{\ell+1,k}&=&  \Big(V(t_{\ell+1-k},t_{\ell+2-k})U(t_{\ell+2-k},t_{\ell+3-k})  P \Big)_{1,1} \, t_{\ell+3-k}\\
&=& \Big(V(t_{\ell+1-k},t_{\ell+2-k})U(t_{\ell+2-k},t_{\ell+3-k})  \Big)_{1,2} \, t_{\ell+3-k}
={1\over  t_{\ell+3-k}}\times  t_{\ell+3-k}=1
\end{eqnarray*}
This completes the proof of the theorem.
\end{proof}

\begin{figure}
\centering
\includegraphics[width=9.cm]{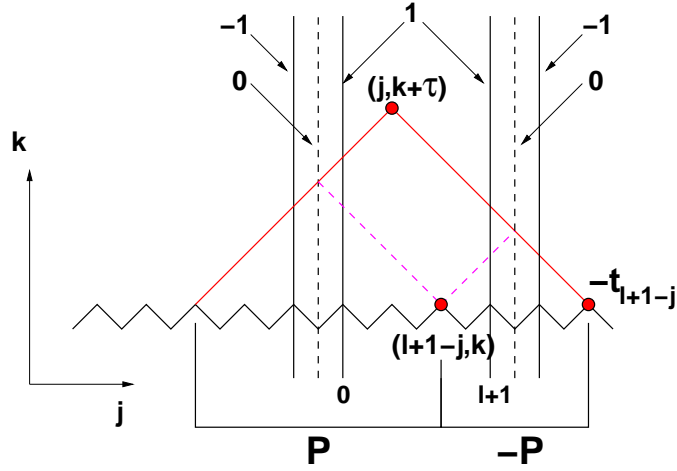}
\caption{For odd $k$, the segment of values necessary to express 
$T_{j,k+N/2}$ is reflected by both lines $j=-1$
and $j=\ell+2$ picking respectively a factor $P$ or $-P$ per reflection. After these two reflections,
the edges of the cone of projection of $(j,k+N/2)$
meet again at the point $(\ell+1-j,k)$. We have indicated the contributions from
the formula \eqref{soloneUV}.}
\label{fig:multirefone}
\end{figure}

\subsection{$\ell$-restricted boundaries and periodicity: proof of Theorem \ref{periodone}}

Combining Theorems \ref{leftbounone} and \ref{rightbounone}, we immediately deduce the following:

\begin{thm}\label{halfperio}
The solution of the $\ell$-restricted $A_1$ $T$-system (iv)
satisfies the following ``twisted half-periodicity"
relation:
$$ T_{j,k+\frac{N}{2}} =T_{\ell+1-j,k}$$
where $N=2(\ell+3)$.
\end{thm}
\begin{proof}
Let $\bk=2m+\bk_0$ denote the unique even translate of 
$\bk_0$ (with $k_j=k_j^{(0)}+2m$) which contains the point $(\ell+1-j,k)$. If $k$ is odd,
then $k_j= k-(j\, {\rm mod}\, 2)$, otherwise, $k_j=k+1-(j\, {\rm mod}\, 2)$ for all $j\in\Z$. 
The point $(\ell+1-j,k)$ is a local maximum if $k$ is odd, minimum otherwise. 
Let us compute $T_{j,k+{N\over 2}}$
via the formula \eqref{soloneUV}. The cone of projection of $(j,k+\frac{N}{2})$
onto the initial data segment $[j_0,j_1]$ 
is reflected once against each of the two
lines $j=-1$ and $j=\ell+2$, and the edges of the cone intersect in the point $(\ell-j,k)$ as
shown in Fig. \ref{fig:multirefone} for odd $k$. For odd $k$, we find that:
$$T_{j,k+{N\over 2}}= \Big( P (-P) \Big)_{1,1} (-t_{\ell+1-j})=T_{\ell+1-j,k} $$
If $k$ is even, the reflected cone edges meet the path $\bk$ respectively at points
$(\ell-j,k+1)$ and $(\ell+2-j,k+1)$, thus leading to:
$$T_{j,k+{N\over 2}}= \Big( P V(t_{\ell-j},t_{\ell+1-j})U(t_{\ell+1-j},t_{\ell+2-j}) 
(-P) \Big)_{1,1} (-t_{\ell+2-j})=t_{\ell+1-j}=T_{\ell+1-j,k} $$
as well. The theorem follows.
\end{proof}

We conclude that $T_{j,k+N}=T_{j,k}$ and the Theorem \ref{periodone} follows.

\subsection{Positivity for $\bk_0$: proof of Theorem \ref{ones}}

\begin{figure}
\centering
\includegraphics[width=10.cm]{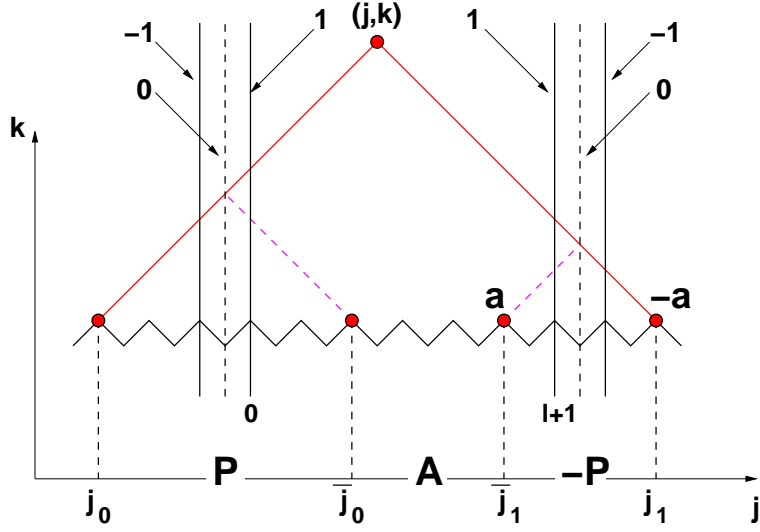}
\caption{The cone of the projection of $(j,k)$ is reflected by both lines $j=-1$
and $j=\ell+2$ picking respectively a factor $P$ or $-P$ per reflection. After these two reflections,
the edges of the cone meet the initial data path at $({\bar j}_0,1)$ and $({\bar j}_1,1)$
respectively. We have indicated the three contributions $P,A,-P$ from
the formula \eqref{soloneUV}, and the reflected boundary value $t_{{\bar j}_1}=a$.}
\label{fig:boxone}
\end{figure}

The first part of Theorem \ref{ones} follows by imposing simultaneously the symmetries of the initial values $t_j$ from 
both Theorems \ref{leftbounone} and \ref{rightbounone}. As these guarantee the $\ell$-restricted boundary conditions,
the result follows from uniqueness of the solution of the $\ell$-restricted system.

Let us now show that the solution $T_{j,k}$ of the $\ell$-restricted $A_1$ $T$-system (iv) is
a positive Laurent polynomial of the initial data along the path $\bk_0$.
Thanks to the half-periodicity property of Theorem \ref{halfperio}, we may restrict ourselves to 
values of $k$ such that $0\leq k\leq \frac{N}{2}$. In that case, the cone of projection of $(j,k)$
is reflected at most once against each line $j=-1$ and $j=\ell+2$.
If no reflection occurs, the positivity is clear, as the solution is identical to that of the 
unrestricted $A_1$ $T$-system. 
If only one reflection occurs, we are in the half-plane situation of Theorems
\ref{leftbounone} or \ref{rightbounone}, and positivity follows.
We are left with the case of two reflections, as illustrated in Fig. \ref{fig:boxone} (case $\ell$ odd). 
As usual we denote by $j_0,j_1$ the minimum and maximum 
of the projection of $(j,k)$ onto the initial data path, and by 
${\bar j}_0= -j_0-2$ and ${\bar j}_1=2(\ell+2)-j_1$ the reflected minimum and maximum of
the projection, such that $0\leq {\bar j}_0\leq {\bar j}_1\leq \ell+1$. 
Applying \eqref{soloneUV} and eliminating 
the left and right products involving $P$ and $-P$ leads to:
$$
T_{j,k}=(P A (-P))_{1,1} (-a)
$$
where 
$$ A=V(t_{{\bar j}_0},t_{{\bar j}_0+1})U(t_{{\bar j}_0+1},t_{{\bar j}_0+2}) \cdots U(t_{{\bar j}_1-1},t_{{\bar j}_1})
\quad {\rm and}\quad a=t_{{\bar j}_1}=-t_{j_1} $$
As usual, the two signs cancel and leave us with a manifestly positive answer, and
the second part of Theorem \ref{ones} follows.

\subsection{Positivity for $\bk$: proof of Theorem \ref{twos}}

In the case of an arbitrary path $\bk$ with associated initial conditions \eqref{initgenone}, we may repeat the
same arguments as in the case $\bk_0$. 
We first need to generalize the first part of Theorem \ref{ones} to the case of an arbitrary
path $\bk$. To this effect, Theorems \ref{leftbounone} and \ref{rightbounone} have 
the following counterparts for arbitrary $\bk$:

\begin{figure}
\centering
\includegraphics[width=9.cm]{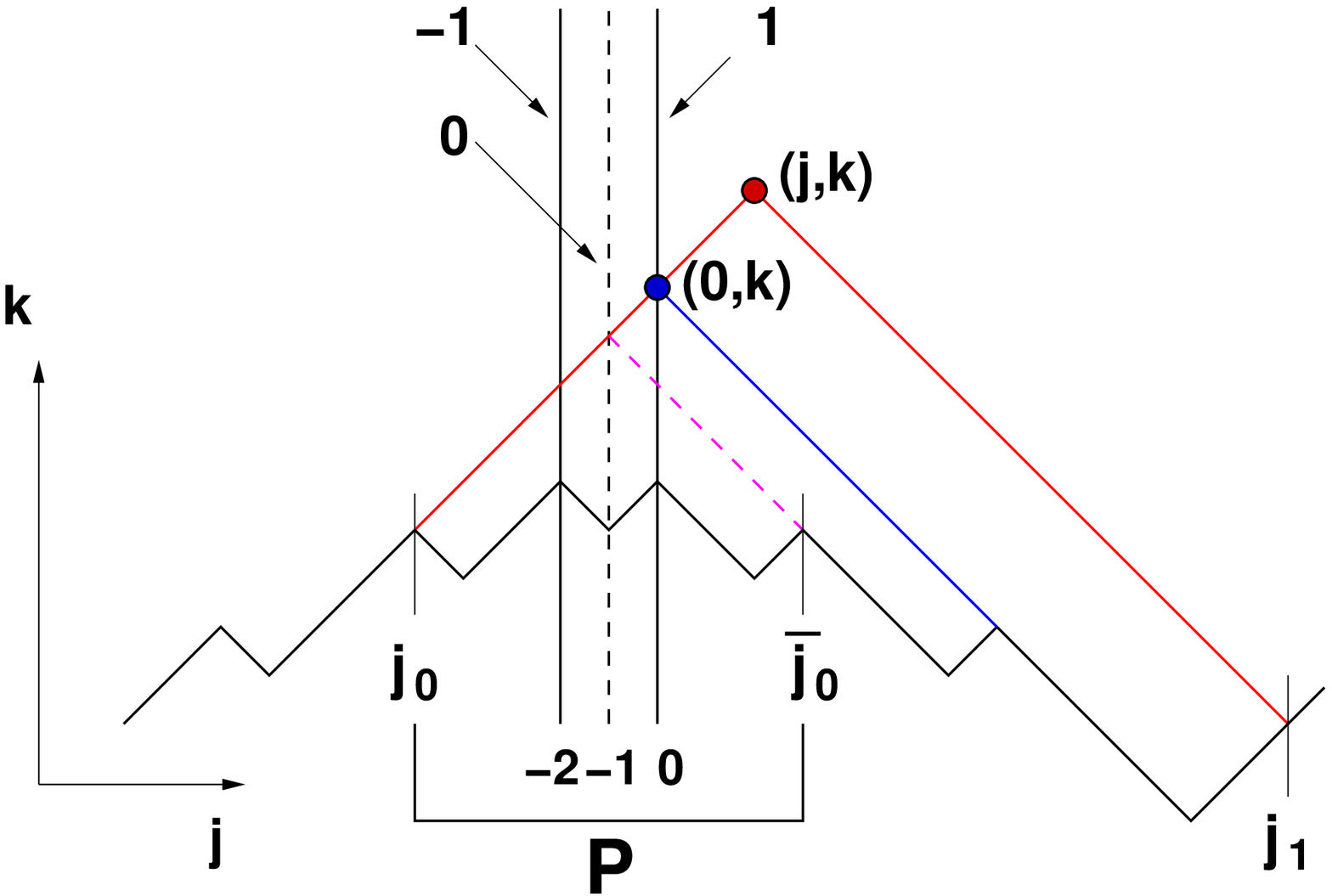}
\caption{The cone of projection of $(j,k)$ is reflected by the line $j=-1$.
As a result, only the values between ${\bar j}_0=-j_0-2$ and $j_1$ matter. We have also indicated the particular
case of the point $(j=0,k)$ (right projection line is blue):
the portion of boundary path for $j\geq {\bar j}_0$ within the projection
is made of $m$ down steps followed by one up step.}
\label{fig:genrefone}
\end{figure}

\begin{thm} \label{arbihalf}
The solution of the unrestricted $A_1$ $T$-system (i) with boundary conditions
$X_{\bk^{+}}(\bt^{+})$ \eqref{otsymii} (resp. $X_{\bk^{-}}(\bt^{-})$ \eqref{otsymiii} )
restricts to that 
of the right (resp. left) half-plane $A_1$ $T$-system (ii) (resp. (iii))
with boundary conditions $X_\bm^{+}(\bu)$ (resp. $X_\bm^-(\bu)$), where $\bu,\bm$ 
are the restrictions of $\bt^{+},\bk^{+}$ (resp. $\bt^{-},\bk^{-}$) to the range $j\geq 1$ (resp. $j\leq \ell$).
As such, the half-plane solutions are positive Laurent polynomials of their initial values.
\end{thm}
\begin{proof}
The theorem is proved by showing that the reflection conditions on $\bk$ and $t_j$
imply that $T_{0,k}=1$ (resp. $T_{\ell+1,k}=1$).
As before, this is proved by use of the formula \eqref{soloneUV}, 
upon noting that $\lim_{\epsilon\to 0}V(-1,\epsilon)U(\epsilon,1)=P$ 
and $\lim_{\epsilon\to 0}V(1,\epsilon)U(\epsilon,-1)=-P$. 
We note also that, with the definition \eqref{defM} and the symmetry properties
of $\bk$ and $t_j$, we have the following collapse relations:
\begin{eqnarray*}
M_{-j-3,-j-2}(t_{-j-3},t_{-j-2}) \! \! \! \! \! \! \! \! &P& \! \! \! \! \! \! \! \! M_{j,j+1}(t_j,t_{j+1})\\
&=&M_{-j-3,-j-2}(-t_{j+1},-t_j)P M_{j,j+1}(t_j,t_{j+1})=P\\
M_{j,j+1}(t_{j},t_{j+1}) \! \! \! \! \!  &(-P)& \! \! \! \! \!  M_{2(\ell+2)-j,2(\ell+2)-j+1}(t_{2(\ell+2)-j},t_{2(\ell+2)-j+1})\\
&=&M_{j,j+1}(t_{j},t_{j+1}) 
(-P) M_{2(\ell+2)-j,2(\ell+2)-j+1}(-t_{j+1},-t_{j})=-P
\end{eqnarray*}
For the right half-plane case, we have for $k\geq k_0$, $k$ odd:
\begin{equation*}
T_{0,k}=\Big( \prod_{j=j_0}^{-3} M_{j,j+1}(t_j,t_{j+1}) P \prod_{j=0}^{j_1-1} M_{j,j+1}(t_j,t_{j+1})\Big)_{1,1}t_{j_1} 
=\Big( P \prod_{j={\bar j}_0}^{j_1-1} M_{j,j+1}(t_j,t_{j+1})\Big)_{1,1}t_k ,
\end{equation*}
where $j_0,j_1$ denote the minimum and maximum of the projection of $(0,k)$ onto $\bk$.
Note that for $j$ between ${\bar j}_0=-j_0-2$ and $j_1$ 
the path $\bk$ must be made of a number $m=j_1+j_0+1$ of down steps,
followed by one up step (see Fig. \ref{fig:genrefone} for an example). This leads to:
$$T_{0,k}=\left[\Big(\prod_{j={\bar j}_0}^{j_1-2} V(t_j,t_{j+1})\Big) \, U(t_{j_1-1},t_{j_1})\right]_{2,1}t_{j_1}
={1\over t_{j_1}}\times t_{j_1}=1$$
The argument is similar for the left half-plane solution. 

To prove positivity, let us consider a point $(j,k)$  above the 
path $\bk$, namely with $k\geq k_j$. Then if $j_0\geq 0$ (resp. $j_1\leq \ell+1$), 
the solution $T_{j,k}$ is identical to that of
the full plane, and positivity is granted. Otherwise, note that the remark \ref{reflerem} extends to the present cases: 
the collapse relations above have the effect of reflecting the cone of projection against the line $j=-1$ 
(resp. $j=\ell+2$), as indicated in Fig. \ref{fig:genrefone}. 
This gives the following expressions for ${\bar j}_0=-j_0-2$ and ${\bar j}_1=2(\ell+2)-j_1$:
\begin{eqnarray*}
&& \! \!\!\!\! {\rm right}\, {\rm half-plane}: \ \ \ \ T_{j,k}=\Big( P\, N({\bar j}_0,j_1)\big)_{1,1} t_{j_1}\\
&& \! \!\!\!\! {\rm left}\, {\rm half-plane}: \ \ \ \ \ \ T_{j,k}=\Big( N(j_0,{\bar j}_1)(-P)\big)_{1,1} t_{j_1}=
\Big( N(j_0,{\bar j}_1)\, P\big)_{1,1} t_{{\bar j}_1}
\end{eqnarray*}
which are both manifestly positive Laurent polynomials of the initial data.
\end{proof}

To prove Theorem \ref{twos}, we now superimpose the symmetry conditions for the two half-plane cases as 
described in Theorem \ref{arbihalf}. Let us show that the solution $T_{j,k}$ with arbitrary path
initial data for $j\in [0,\ell+1]$ has the positive Laurent property.

The half-periodicity holds in general, so we may restrict ourselves to the case of a general path $\bk$
and a point $(j,k)$ above it such that $0\leq k-k_j \leq \frac{N}{2}$. In this case there is at most one reflection
of the cone of projection of $(j,k)$ against each of the lines $j=-1$ and $j=\ell+2$.
For no reflection at all, the solution is the same as that of the unrestricted $A_1$ $T$-system case, 
which is manifestly positive.
For one reflection against one of the lines, the solution is the same solution as that of the half-plane 
$A_1$ $T$-system case,
where positivity was established above.
Finally for two reflections, we have:
$$T_{j,k}=\Big( P N({\bar j}_0,{\bar j}_1)(-P) \Big)_{1,1} t_{j_1}=\Big( N({\bar j}_0,{\bar j}_1)\Big)_{2,2} t_{{\bar j}_1} $$
which is manifestly positive. This completes the proof of Theorem \ref{twos}.

\section{$\ell$-restricted T-system: the $A_r$ case}\label{ellsec}

Throughout this section, we study solutions of the $A_r$ $T$-system (\ref{tsys},\ref{arboundary}) 
with several types of boundary conditions,
and with  initial conditions of the form $X(\bt):=X_{\bk_0}(\bt)$  \eqref{initcondar} or appropriate subsets thereof.

Let $S$ be a subset of $\Z$. We have the $T$-system
$$
T_{i,j,k+1}T_{i,j,k-1}=T_{i,j+1,k}T_{i,j-1,k}+T_{i+1,j,k}T_{i-1,j,k} \qquad (i\in[1,r],k\in \Z, j\in S)
$$ 
with boundary conditions 
\begin{equation}\label{artsysi}
T_{0,j,k}=T_{r+1,j,k}=1\qquad (j\in S;k\in\Z) 
\end{equation}
with possibly additional boundary conditions depending on $S$, and an initial condition $X^{S}(\bt)$, which is an assignment of  formal variables $\bt=(t_{i,j})_{i\in[1,r],j\in S}$ to the points on the surface
$$
\bk_0(S)= \{(i,j,k^{(0)}(i,j)): i\in [1,r], j\in S\}.
$$
We consider the following four cases:

\begin{enumerate}[(i)]
\item
Unrestricted $A_r$ $T$-system: $S=\Z$, there are no additional boundary conditions, 
and the initial condition $X(\bt)$ is an assignment of values to the variables on points of $\bk_0$:
\begin{equation}\label{initcondi}
X(\bt): \left\{ T_{i,j,k_{i,j}^{(0)}}=t_{i,j} \quad (i\in[1,r];j\in \Z) \right\} .
\end{equation}
\item
Right half-space $A_r$ $T$-system: $S=\N$,
the additional
boundary conditions are
\begin{equation}
 T_{i,0,k}=1 \qquad  (i\in[1,r];k\in\Z) \label{artsysii}
\end{equation}
and initial condition $X^+(\bt)$ is the assignment
\begin{equation}\label{initcondii}
X^+(\bt): \left\{ T_{i,j,k_{i,j}^{(0)}}=t_{i,j} \quad (i\in[1,r];j \in S) \right\} .
\end{equation}
\item
Left half-space $A_r$ $T$-system: $S=(-\infty,\ell]$,
the additional boundary conditions are
\begin{equation}
T_{i,\ell+1,k}=1 \qquad  (i\in[1,r];k\in\Z)\label{artsysiii}
\end{equation}
and the initial conditions $X^-(\bt)$ are
\begin{equation}\label{initcondiii}
X^-(\bt): \left\{ T_{i,j,k_{i,j}^{(0)}}=t_{i,j} \quad (i\in[1,r];j\in S )\right\} .
\end{equation}
\item
$\ell$-restricted $A_r$ $T$-system: $S=[1,\ell]$, the additional boundary conditions are
\begin{equation}
T_{i,0,k}=T_{i,\ell+1,k}=1 \qquad  (i\in[1,r];k\in\Z)\label{artsysiv}
\end{equation}
and the initial conditions are $X^{[1,\ell]}(\bt)$
\begin{equation}\label{initcondiv}
X^{[1,\ell]}(\bt): \left\{ T_{i,j,k_{i,j}^{(0)}}=t_{i,j} \quad (i\in[1,r];j\in S) \right\} .
\end{equation}
\end{enumerate}

\begin{remark}\label{uniqueness}
In all the above cases, due to the form of the $T$-system as a three-term recursion,
the solution of the system is uniquely determined by its initial conditions.
\end{remark}

We will also consider the unrestricted $A_r$ $T$-system (case (i))
with initial conditions $X(\bt)$ \eqref{initcondi}, 
where we impose certain relations on the variables $\bt=(t_{i,j})_{i\in [1,r], j\in \Z}$:
\begin{itemize}
\item{$\bt^{+}$} is $\bt$ modulo the relations
\begin{eqnarray}
&t_{r+1-i,-r-1-j}=(-1)^{ri} \ t_{i,j},\qquad &(i\in [1,r];j\geq 0),\label{tsymii}\\
&t_{i,0}=1, \
t_{i,-j}=0, \qquad& (i,j\in[1,r]). \label{tvanii}
\end{eqnarray}
\item{$\bt^{-}$} is $\bt$ modulo the relations
\begin{eqnarray}& t_{r+1-i,-r-\ell-2+j}=(-1)^{ri} \ t_{i,\ell+1-j},\qquad &(i\in [1,r];j\geq 0)\label{tsymiii}\\
 &t_{i,\ell+1}=1,\ 
t_{i,\ell+1+j},\qquad &(i,j\in [1,r]).\label{tvaniii}
\end{eqnarray}
\item{$\bt^{[1,\ell]}$} is $\bt$ modulo the relations
\begin{eqnarray} 
&t_{r+1-i,-r-1-j}=(-1)^{ri} \ t_{i,j},\qquad &(i\in [1,r];j\geq0)\label{symone} ,\\
&t_{i,2(r+\ell+2)+j}=t_{i,j},\qquad &(i\in [1,r];j\in \Z),\label{symtwo} \\
&t_{i,0}=t_{i,\ell+1}=1,\qquad&(i\in [1,r]), \label{bcone} \\
&t_{i,-j}=0, \qquad &(i,j\in [1,r]).\label{vani}
\end{eqnarray}
\end{itemize}
\begin{remark}\label{simultarem}
The relations satisfied by $\bt^{[1,\ell]}$ correspond to simultaneously imposing  the relations of $\bt^{+}$
and $\bt^{-}$.
\end{remark}
\begin{example}\label{a3exsimple}
Initial data of type $\bt^+$ for the case $r=3$ has the form 
(with the $i$ direction is from bottom to top,
and $j$ direction is from left to right):
$$\small \begin{array}{cccccccccccccccccccccc}
\cdots &-t_{1,5} &-t_{1,4}& -t_{1,3}&-t_{1,2}& -t_{1,1}&-1 & 0 & 0 &0 & 1 &  t_{3,1} & t_{3,2} & t_{3,3} & t_{3,4} & t_{3,5} & \cdots\\
\cdots &\ \ t_{2,5} &\ \  t_{2,4}&\ \ t_{2,3}&\ \ t_{2,2}&\ \ t_{2,1}&\ \  1 & 0 & 0  &0 & 1 &  t_{2,1} & t_{2,2} & t_{2,3} & t_{2,4} & t_{2,5} & \cdots\\
\cdots &-t_{3,5} &-t_{3,4}& -t_{3,3}&-t_{3,2}&-t_{3,1}& -1 & 0 & 0 &0 & 1 &  t_{1,1} & t_{1,2} & t_{1,3} & t_{1,4} & t_{1,5} & \cdots
\end{array}$$
\end{example}

\begin{example}\label{a3ex}
Initial data of the type $\bt^{[1,\ell]}$
for the case $r=3$, $\ell=3$ has the form
$$\small \begin{array}{cccccc|ccccc|ccccccccccc}
& -t_{1,1}&-1 & 0 & 0 &0 & 1 &  t_{3,1} & t_{3,2} & t_{3,3} & 1 & 0 & 0 &0 & -1 & - t_{1,3} & -t_{1,2} & -t_{1,1} & -1& 0 & 0 &\\
\cdots & \ \ t_{2,1}& \ \ 1 & 0 & 0  &0 & 1 &  t_{2,1} & t_{2,2} & t_{2,3} & 1 & 0 & 0 &0 & \ \ 1  & \ \  t_{2,3}  
&\ \  t_{2,2}   &\ \  t_{2,1}  &\ \  1& 0 &  0&
\cdots   \\
&-t_{3,1}& -1 & 0 & 0 &0 & 1 &  t_{1,1} & t_{1,2} & t_{1,3} & 1 & 0 & 0 &0 & -1 &  -t_{3,3} & -t_{3,2} & -t_{3,1}& -1 & 0 & 0&
\end{array}$$
This array has period $2(\ell+r+2)=16$ along the horizontal ($j$-)direction. 
The vertical bars indicate the domain corresponding
to the $\ell$-restricted $A_r$ $T$-system (iv) initial data.
\end{example}

As in the $A_1$ case, the aim of this section is to use the known network solution for the unrestricted
system (i)  to obtain that for the other boundary conditions (ii,iii,iv). 

\subsection{Equivalent initial data and main theorems}\label{arone}

Here, we give the line of argument used to prove the periodicity
theorems \ref{priociT} and  \ref{positiT}. 

\begin{lemma}\label{suffiT}
The solutions of the $T$-system (i) with initial conditions $X(\bt^+)$ 
satisfy
\begin{equation}
T_{1,0,k}=1 \quad (k\in 2\Z+1), \qquad T_{1,-j,k}=0,\quad (j\in [1,r];k\in 2\Z+j+1).
\end{equation}
\end{lemma}
The determinant formula \eqref{wronsk} and the Lemma imply that $T_{i,0,k}=1$
for  all $i\in[1,r]$ and $k\in 2\Z+i$.
The proof of this Lemma is given in Section \ref{arthree}.

\begin{thm}\label{plusimple}
The solutions $T_{i,j,k}$ of the unrestricted $A_r$ $T$-system (i)
as a function of the initial conditions $X(\bt^{+})$ 
are equal, when $j>0$, to the solutions $T_{i,j,k}$ of the right half-space $A_r$ $T$-system (ii)
with initial conditions $X^+(\bt)$. 
\end{thm}
\begin{proof}
Given Lemma \ref{suffiT}, 
the theorem follows from the
uniqueness of the solutions (Remark \ref{uniqueness}) of the half-plane $T$-system with initial data $X^+(\bt)$.
\end{proof}

\begin{thm}\label{corolla}
The solutions $T_{i,j,k}$ of the unrestricted $A_r$ $T$-system (i) as a function of initial conditions $X(\bt^{-})$
are equal, when $j\leq \ell$, to the solutions $T_{i,j,k}$  of the left half-space $A_r$ $T$-system (iii) 
with initial conditions $X^-(\bt)$.
\end{thm}
\begin{proof}
Let $\sigma$ be the following endomorphism of $[1,r]\times \Z \times \Z $:
$$
\sigma(i,j,k) = \left\{ \begin{array}{ll} (i,\ell+1-j,k),& \hbox{$\ell$ odd};\\
(i,\ell+1-j,1-k),&  \hbox{$\ell$ even}.\end{array}\right.
$$
Then $\sigma(\bk_0)=\bk_0$ and $\sigma$ is also a symmetry of the unrestricted $A_r$ $T$-system (i).
It acts on $\bt$ in the natural way,  $\sigma(t_{i,j})=t_{i,\ell+1-j}$, and takes initial data of the form $\bt^+$ to data of the form $\bt^-$. The Theorem follows from application of $\sigma$ to the result of Theorem \ref{corolla}.
\end{proof}

Using the map $\sigma$ together with Lemma \ref{suffiT} we see that for all $k$ of appropriate parity,
\begin{equation}\label{suffiTb}
T_{i,\ell+1,k}=1 , \qquad T_{i,\ell+1+j,k}=0\quad (j\in [1,r]).
\end{equation}
Lemma \ref{suffiT} and its reflected version \eqref{suffiTb} imply the following result for $\ell$-restricted $A_r$ $T$-system solutions:
\begin{thm}\label{inithm}
The solutions $T_{i,j,k}$ of the unrestricted $A_r$ $T$-system (i) as a function of initial conditions $X(\bt^{[1,\ell]})$ 
are equal, when $j\in[1,\ell]$, to the solutions of the $\ell$-restricted $A_r$ $T$-system (iv)
with initial conditions $X^{[1,\ell]}(\bt)$. 
\end{thm}
\begin{proof}
This follows from Remark \ref{simultarem} and the uniqueness of the solutions.
\end{proof}

We will also prove certain positivity results for the solutions of the $A_r$
$T$-systems of types (ii)--(iv), using the explicit network solution of (i):
\begin{thm}\label{haposi}
The solutions $T_{i,j,k}$ of $A_r$ $T$-system of type (ii) and (iii) 
with initial conditions $X^+(\bt)$ or resp. $X^-(\bt)$, are Laurent polynomials of the
initial data $\bt$, with non-negative integer coefficients.
\end{thm}
From the network solution with the two half-plane boundaries superimposed, this implies
the positivity Theorem \ref{positiT}. The proof appears in Section \ref{arfive}.

\subsection{A regularized network matrix}\label{artwo}
Initial data of the form $\bt^+$ contains zeros. In order to define network matrices depending on this initial data, the matrices $U$ and $V$ cannot be used directly.
To define the matrices $P_j:=N(-j,0)$ (with $j\in[0,r+1]$) depending on $\bt^+$,
we use a limiting procedure as in the case of $A_1$.
First, define regularized initial data by replacing 
the vanishing conditions of Theorem \ref{plusimple} by non-zero values forming an array $(a_{i,j})$
compatible with the $T$-system. 
The regularized network matrices $P_j(\{a\})$ have a well-defined limit
when $a_{i,j}\to 0$.

\subsubsection{Regularized initial data}

\begin{defn}\label{defa}
We consider the array $(a_{i,-j})_{i,j\in[0,r+1]}$ such that:
\begin{eqnarray*}
&& a_{i,0}=a_{0,-j}=a_{r+1,-j}=1, \qquad i,j\in [0,r+1]\\
&& a_{i-1,-j}a_{i+1,-j}+a_{i,-j-1}a_{i,-j+1}=0,\qquad i,j\in [1,r]. 
\end{eqnarray*}
\end{defn}
The values of $a_{i,j}$ are determined recursively from the column with $j=-1$. Define $a_i=a_{i,-1}$
for $i\in [1,r]$. Then

\begin{equation}\label{avalue}
a_{i,j}=\epsilon_{i,j} \prod_{\ell=0}^{{\rm Min}(i,-j,r+1-i,r+1+j)-1} a_{|i+j|+1+2\ell} \end{equation}
where $\epsilon_{i,j}\in \{ -1,1\}$ is the solution to the recursion relations
$\epsilon_{i,j-1}=-\epsilon_{i-1,j}\epsilon_{i+1,j}/\epsilon_{i,j+1}$,
while $\epsilon_{i,0}=\epsilon_{i,-1}=\epsilon_{0,j}=\epsilon_{r+1,j}=1$ for all $i\in[0,r+1]$ and $j\in [-r-1,0]$.
In particular,  $a_{i,-r-1}=\epsilon_{i,-r-1}=(-1)^{ri}$. 

\begin{example}
For the case $r=3$, we have the following array (represented with index $i$ from bottom to top and $j$ from left to right):
\begin{equation}\label{initex}(a_{i,-j})_{0\leq i,j\leq 4}=\begin{pmatrix}
1  & 1 & 1 & 1 & 1 \\
-1  &a_1 &-a_2 & a_3 & 1\\
1  &-a_2 &-a_1a_3 & a_2 & 1\\
-1    &a_3 &-a_2 & a_1 & 1\\
1  & 1 & 1 & 1 & 1 
\end{pmatrix} ,\end{equation}
and for the case $r=4$, 
\begin{equation}\label{initextwo}(a_{i,-j})_{0\leq i,j\leq 5}=\begin{pmatrix}
1  & 1 & 1 & 1 & 1 & 1 \\
1 & -a_1 & a_2 & -a_3 & a_4 & 1\\
1  &-a_2 & -a_1a_3 &  -a_2a_4& a_3 & 1\\
1  &-a_3 & -a_2a_4&  -a_1a_3 & a_2 & 1\\
1  &-a_4 & a_3 & -a_2 & a_1 & 1\\
1  & 1 & 1 & 1 & 1  & 1
\end{pmatrix} \end{equation}
\end{example}

\begin{remark}\label{remone}
As apparent from the formula \eqref{avalue},
the expression for $a_{i,j}$ involves only $a_k$'s with a fixed parity of $k$, namely
$k=i+j+1$ mod 2.
\end{remark}

We define the regularlized initial data $\bt^+(\ba)$ as follows. We keep the symmetry
requirements \eqref{tsymii} but replace the zeros in \eqref{tvanii} with arrays satisfying
Definition \ref{defa}:
\begin{equation}\label{eani}
t_{i,j}=a_{i,j}  \qquad \qquad (i\in [1,r];j\in [-r-1,0])
\end{equation}

\subsubsection{Regularized network matrices}\label{primwork}

For each $j\in[0,r+1]$, define
$P_j(\{a\})=N(-j,0)(\{a\})$ corresponding to the network with initial values $\bt^+(\ba)$ as follows.
Let 
\begin{equation}\label{defm}
N_{i,j}(\{a\})=\left\{ \begin{array}{ll}U_i(a_{i,j-1},a_{i,j},a_{i+1,j-1}) & {\rm if}\ i+j=1\ {\rm mod} \ 2; \\
V_{i}(a_{i-1,j},a_{i,j-1},a_{i,j}) & {\rm otherwise},
\end{array}\right. \qquad (i,-j\in[1,r]).
\end{equation}
The regularized network matrix is the product of matrices
\begin{equation}\label{defpjofa}
P_j(\{a\})=\prod_{k=-j+1}^{0} \prod_{i=1}^r N_{i,k}(\{a\}) ,\qquad j\in[0,r+1]
\end{equation}
taken with the indicated order. With this definition, the matrix corresponding to the lower right corner of the network
is $U_{1}(a_1,1,a_2)$, as it corresponds to $i=1$ and $j=0$ in \eqref{defm}.

\begin{lemma}\label{squarelem}
Within the domain $j\in [-r-1,0]$ of the regularized network, each ``diamond" of the form $U_i(a,b,v)V_i(u,b,c)$
or $V_i(u,a,b)U_i(b,c,v)$, with $ac+uv=0$, has elements in $\Z[b,u,v,c^{-1}]$. 
In particular, only $c$ may occur as a denominator.
\end{lemma}
\begin{proof}
We compute the $UV$ diamond matrix:
\begin{equation}\label{UVm}
U(a,b,v)V(u,b,c)=\raisebox{-1.1cm}{\hbox{\epsfxsize=2.5cm \epsfbox{UVm.eps}}}=
\begin{pmatrix}
{b\over c} & {u\over c} \\
{v\over c} & 0 
\end{pmatrix}\end{equation}
where the $(2,2)$ matrix element vanishes, due to $ac+uv=0$. Analogously,
\begin{equation}\label{VUm}V(u,a,b)U(b,c,v)=\raisebox{-1.1cm}{\hbox{\epsfxsize=2.5cm \epsfbox{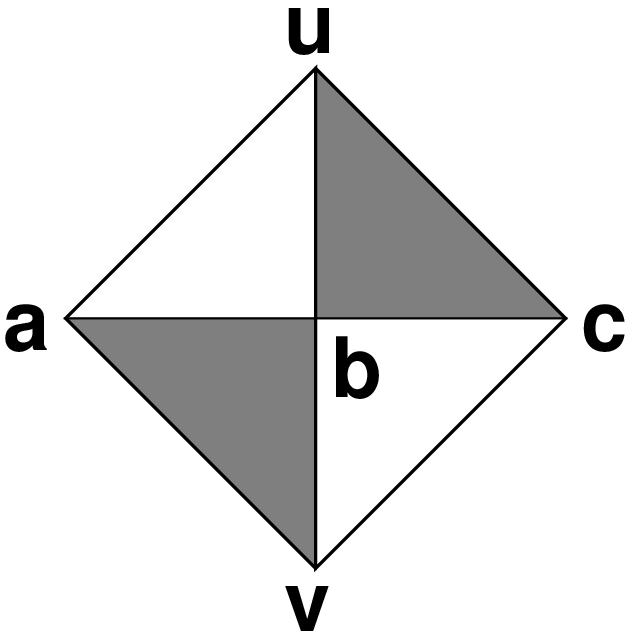}}}=
\begin{pmatrix}
0 & {u\over c} \\
{v\over c} & {b\over c}
\end{pmatrix}\end{equation}
where the $(1,1)$ matrix element vanishes, due to $ac+uv=0$.
\end{proof}

\begin{thm}\label{polth}
The entries of the matrices $P_j(\{a\})$, $j\in[0,r+1]$, are polynomials of the $a_i$'s, $i\in[1,r]$. Therefore,
the matrices
$$P_j:=\lim_{a_1,a_2,...,a_r\to 0} \ P_j(\{a\})$$
are well defined.
\end{thm}
\begin{proof}
We concentrate on the portion $[-j,0]$ of the regularized network. It may be decomposed 
into two types ($UV$ or $VU$) of diamonds as follows:
\begin{equation}\label{squarePj}
P_j(\{ a\})= \raisebox{-2.8cm}{\hbox{\epsfxsize=4.5cm \epsfbox{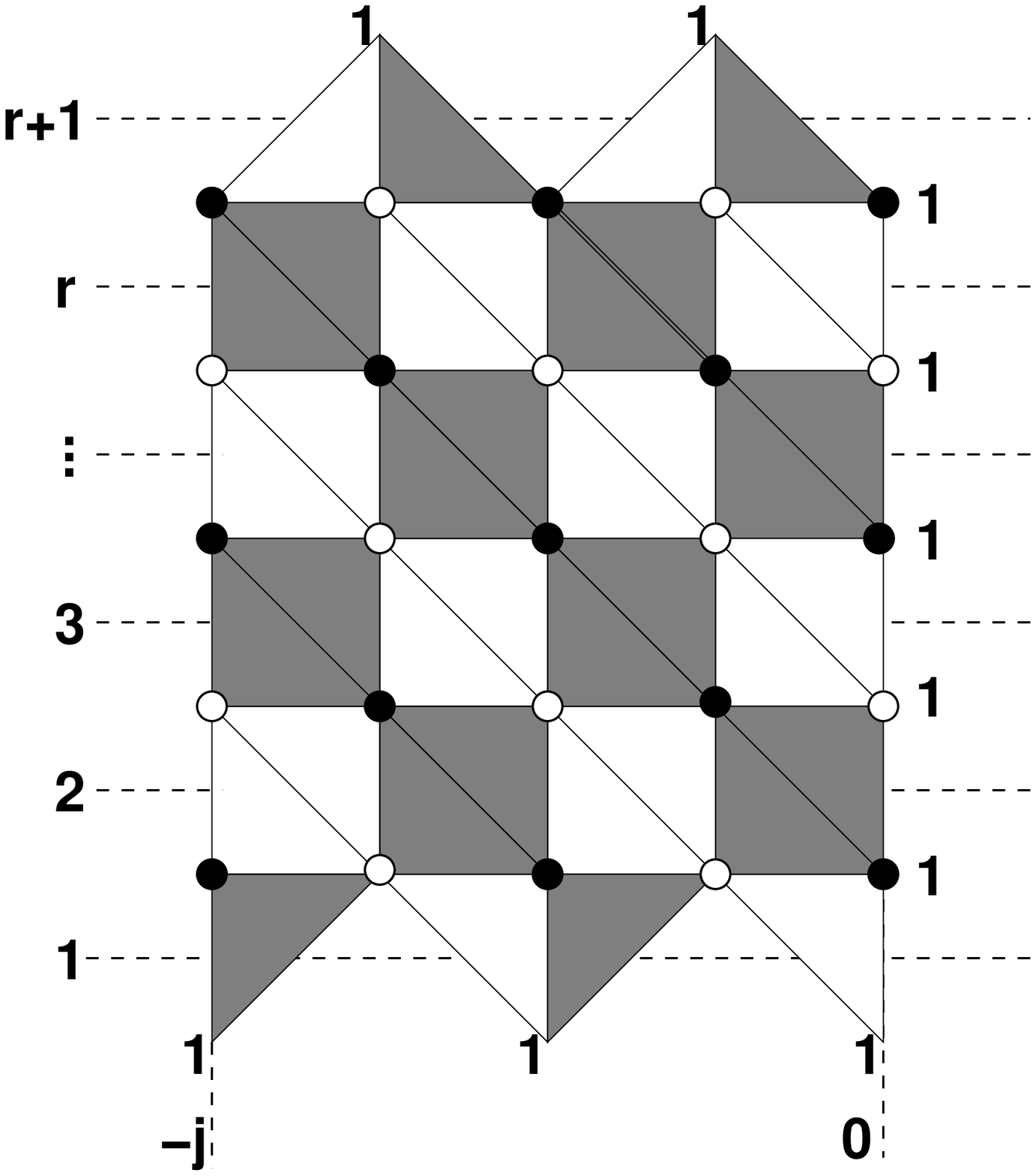}}}
=\raisebox{-2.8cm}{\hbox{\epsfxsize=4.3cm \epsfbox{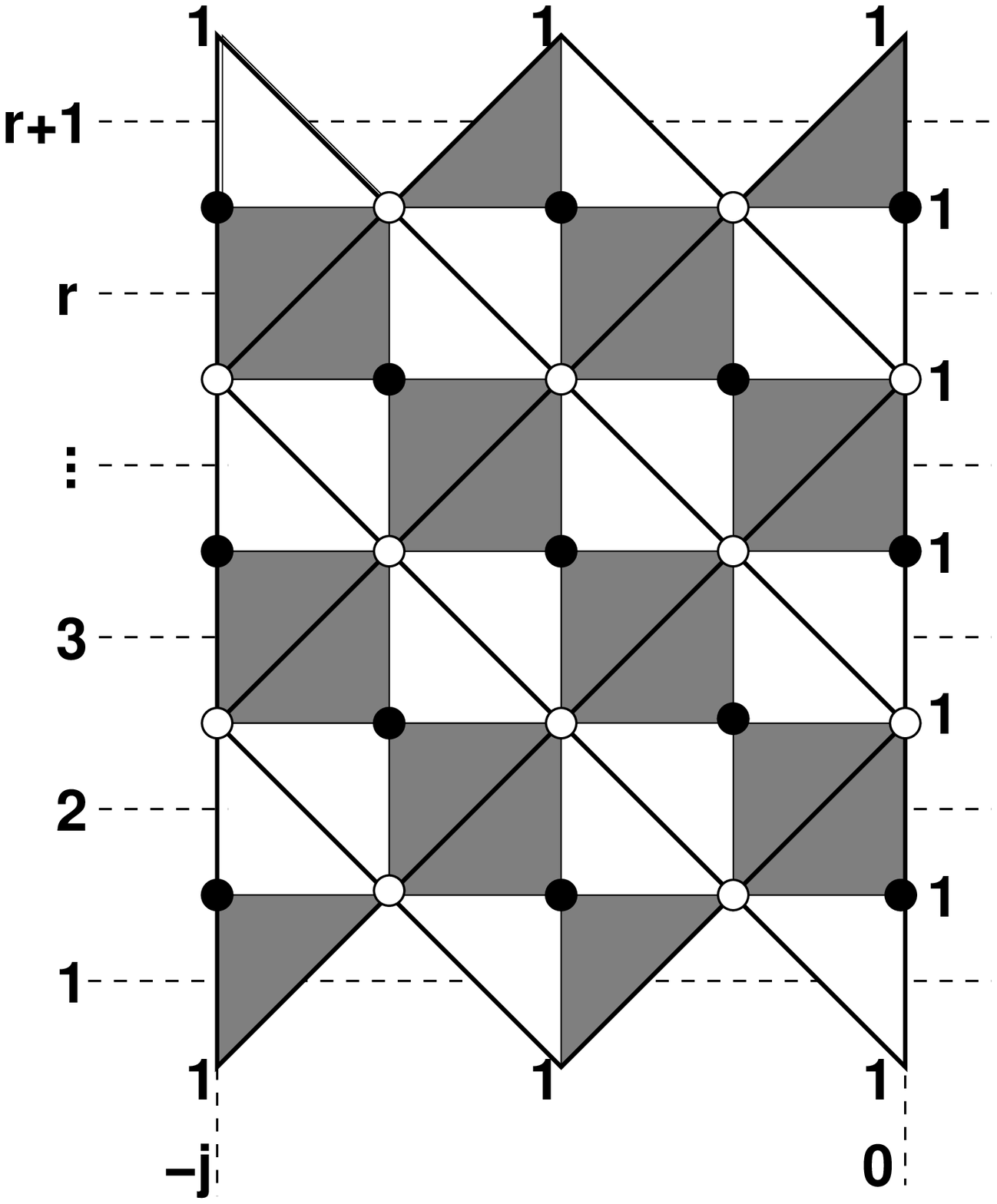}}}
=\raisebox{-2.8cm}{\hbox{\epsfxsize=4.3cm \epsfbox{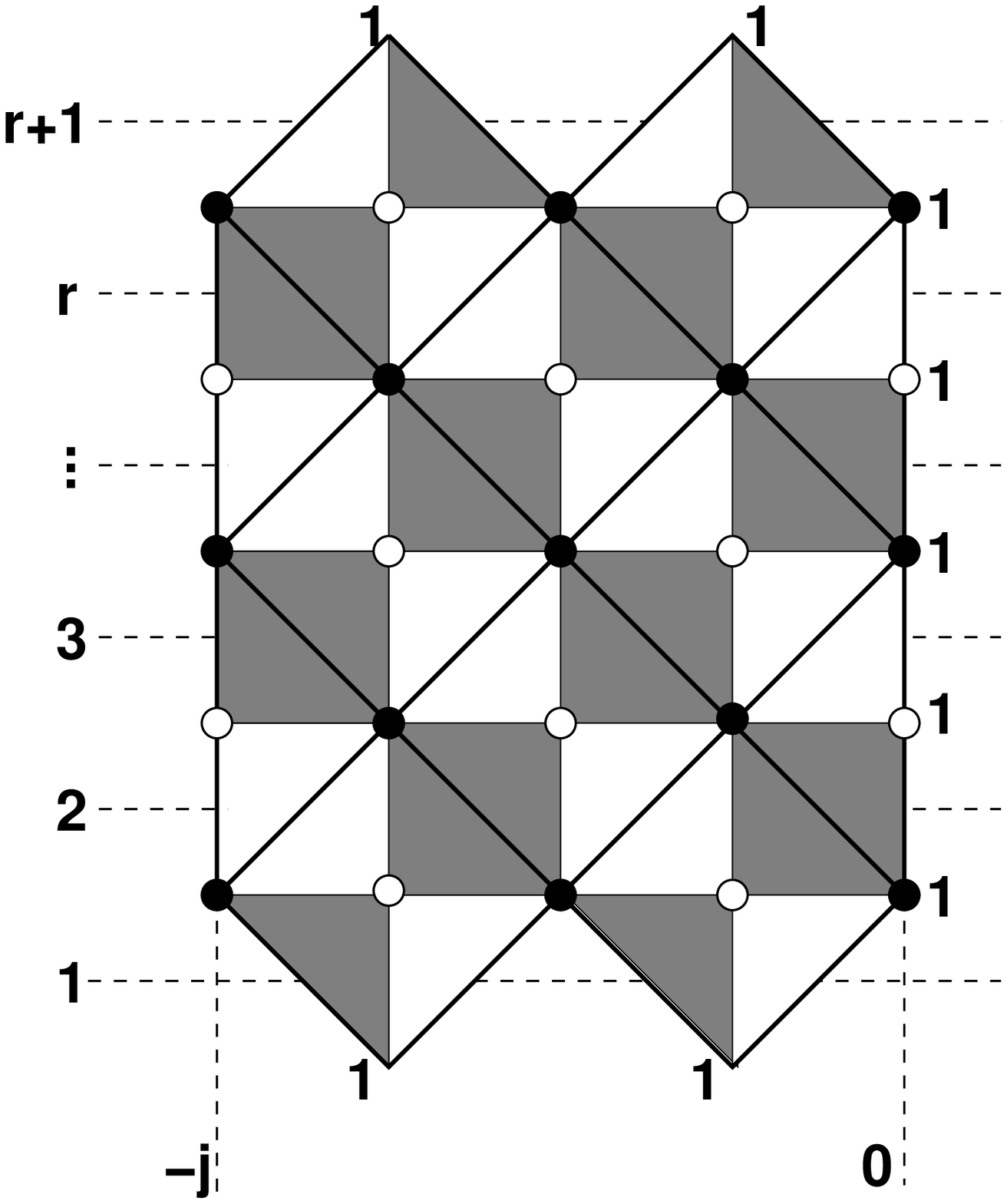}}} 
\end{equation}
Due to  Lemma \ref{squarelem}, 
the first decomposition gives rise to matrix elements with denominators corresponding
to values of $a_{i,j}$ at white circles (with $i+j=1$ mod 2), while in the second the only possible denominators
correspond to values of $a_{i,j}$ at black circles (with $i+j=0$ mod 2). 
The unpaired column of $U$ and $V$ matrices on the right has only $1$ as possible denominator,
due to the boundary condition along the $j=0$ column.

From \eqref{avalue} and
Remark \ref{remone}, the matrix elements for the first expression for $P(\{a\})$ may only have denominators 
that are monomials
of the $a_{2i}$'s, while the second expression may only  have denominators that are monomials of the $a_{2i+1}$'s.  
We conclude that none of these denominators may occur in $P_j(\{ a\})$, which is therefore a polynomial of the $a$'s,
and the theorem follows.
\end{proof}

Let $P$ denote the  $(r+1)\times (r+1)$ matrix with entries:
\begin{equation}\label{Pmat}
[P]_{i,j}= (-1)^{(r-1)(i-1)} \delta_{i+j,r+2}.
\end{equation}
Clearly, $P^2={\mathbb I}$, and when $r=1$ the above definition agrees with \eqref{limP}.

\begin{lemma}\label{deterP}
$$\det(P_{r+1}(\{a\})= \det(P)=(-1)^{r(r+1)(r+2)\over 2}$$
\end{lemma}
\begin{proof}
The determinants of  the $U$ and $V$ matrices are equal to the product of weights of the horizontal edges.
Therefore
$\det(P_{r+1}(\{a\})$ is a product over the weights of all horizontal edges of the regularized network, each of which is equal to $1$ or
$a_{i,j-1}/a_{i,j}$ where $i\in [1,r]$ and $j\in [-r,0]$. Therefore, 
$$\det(P_{r+1}(\{a\})=\prod_{i-1}^r\prod_{j=-r}^0 {a_{i,j-1}\over a_{i,j}}=\prod_{i=1}^r {a_{i,-r-1}\over a_{i,0}}=
\prod_{i=1}^r (-1)^{ri},$$
and the lemma follows.
\end{proof}

We also note the following useful properties of $U,V$ 
matrices:
\begin{equation}\label{proj} U(a,b,c)=U(\lambda a,\lambda b,\lambda c)\qquad V(a,b,c)=V(\lambda a,\lambda b,\lambda c)
\end{equation}
\begin{equation}\label{invUV} U(a,b,c)^{-1}= U(b,a,-c) \qquad V(a,b,c)^{-1}= V(-a,c,b) 
\end{equation}
We give below a pictorial proof of the following formula for $P_{r+1}(\{a\})$: 
\begin{thm}\label{Prplus}
Given $r=2s+\epsilon$ with $\epsilon\in\{0,1\}$, the matrix $P_{r+1}(\{a\})$ is
$$ P_{r+1}(\{a\})=\left( \prod_{i=1-\epsilon}^s U_{2i}(1,1,(-1)^{\epsilon+1}a_{2(s-i)+1}) \right) P \left(  \prod_{i=1}^s U_{2i}(1,1,a_{2i}) \right)$$
\end{thm}
\begin{proof}
In the case $r=2s$, using the formula for the inverse of $U$ \eqref{invUV}, 
the statement of the theorem will follow if we prove that
$$\Pi_{r+1}(\{a\})=\left( \prod_{i=1}^s U_{2i}(1,1,a_{2(s-i)+1}) \right)P_{r+1}(\{a\})\left(  \prod_{i=1}^s U_{2i}(1,1,-a_{2i}) \right)$$
is equal to $P$, independently of the $a$'s. Analogously, when $r=2s+1$, using also the projectivity property \eqref{proj}
with $\lambda=-1$, the theorem boils down to proving that
$$\Pi_{r+1}(\{a\})=\left(\prod_{i=0}^s U_{2i+1}(-1,-1,a_{2(s-i)+1})\right) P_{r+1}(\{a\})\left(  \prod_{i=1}^s U_{2i}(1,1,-a_{2i}) \right)$$
is equal to $P$, independently of the $a$'s. 
The matrix $\Pi_{r+1}(\{a\})$ corresponds in both cases to an augmented network matrix. We illustrate the network
below for the cases $r=7, 8$:
$$
 \raisebox{-3.cm}{\hbox{\epsfxsize=12.cm \epsfbox{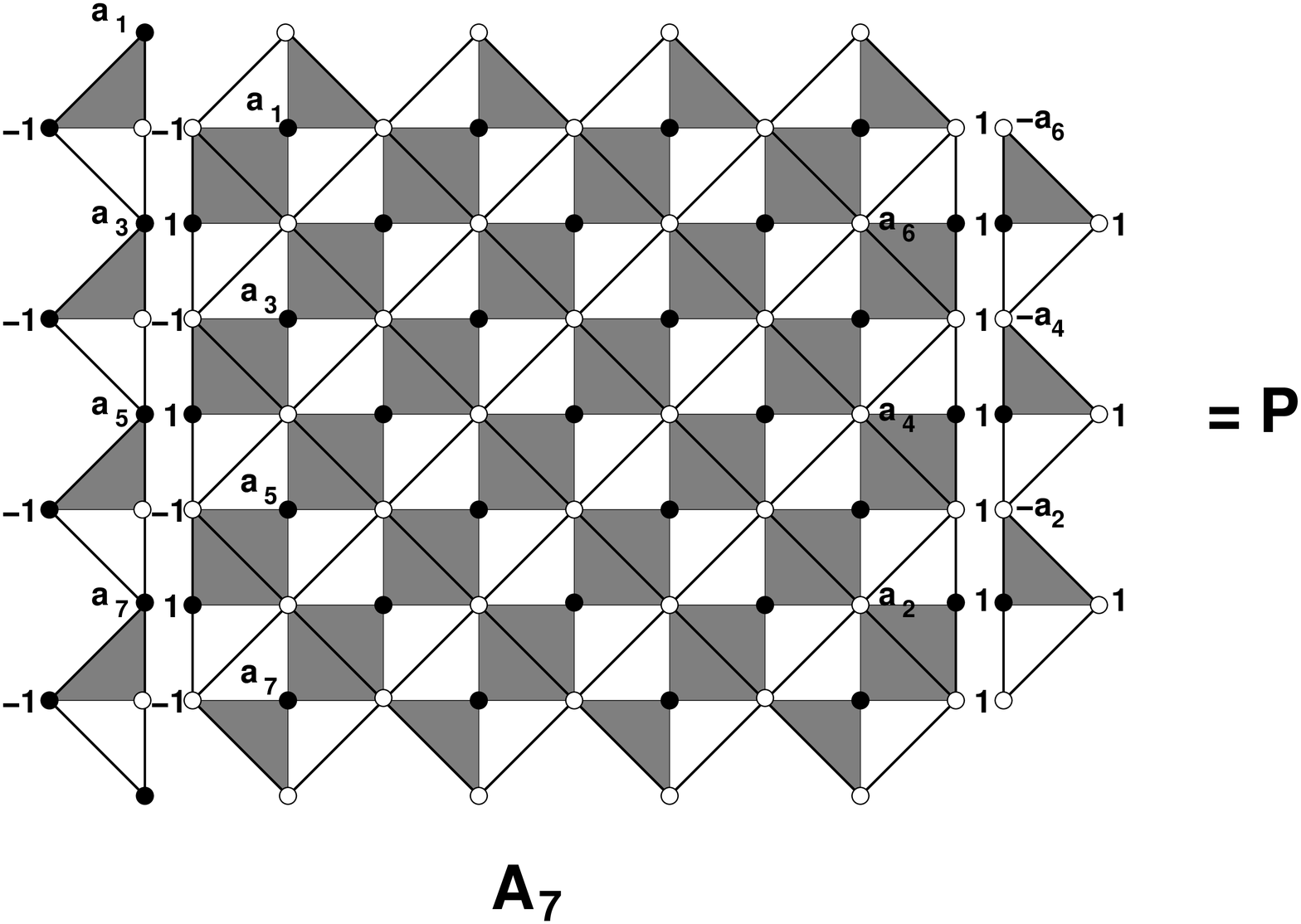}}}
$$
$$
\raisebox{-3.cm}{\hbox{\epsfxsize=12.cm \epsfbox{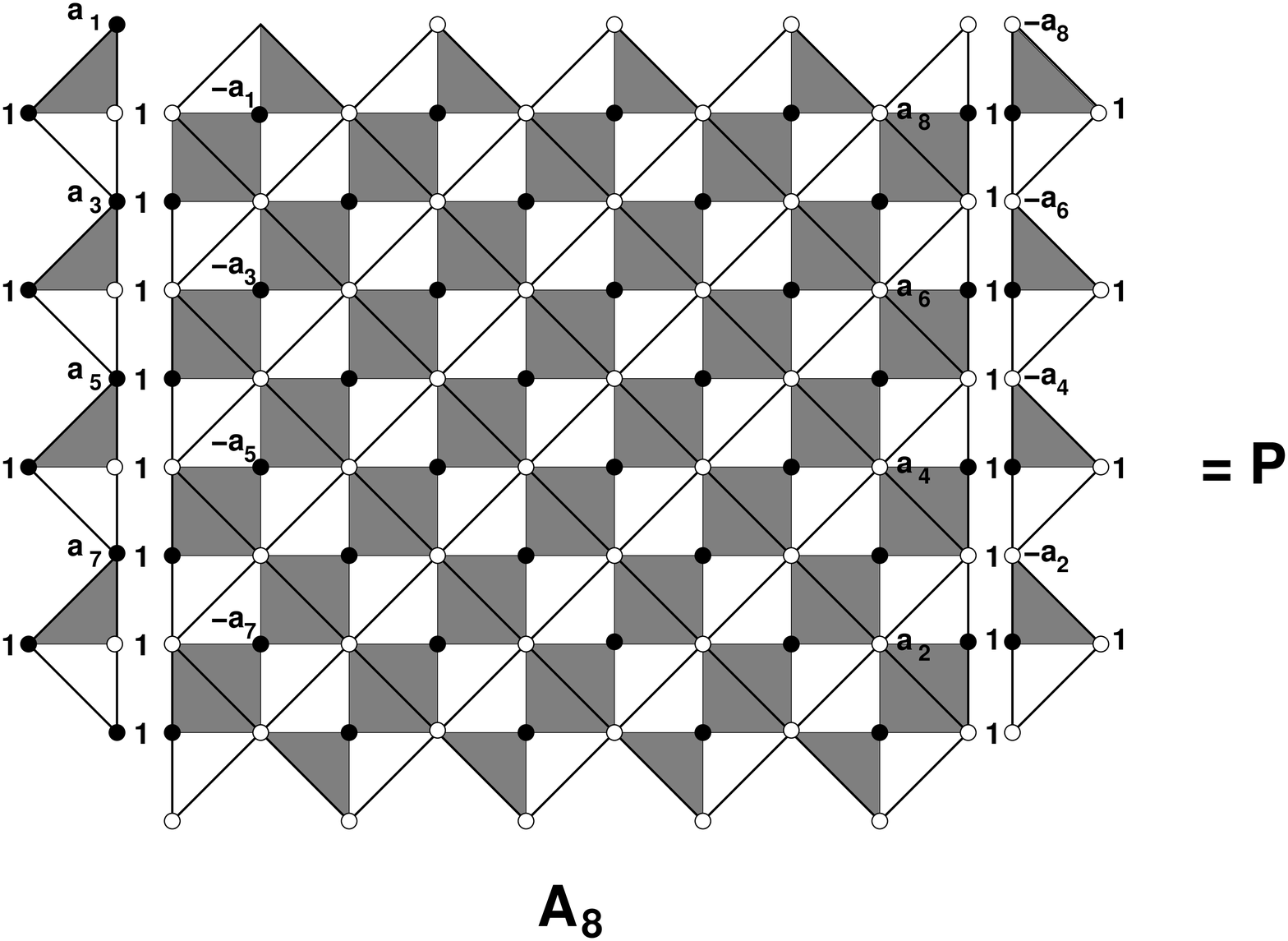}}}
$$
(Note that the actual values of the added pieces are compatible with both $UV$ and $VU$ diamond 
decompositions of Theorem \ref{polth}.).
The matrix elements of $\Pi_{r+1}(\{a\})$ are all polynomials of the $a$'s. This is readily seen from Theorem \ref{polth} 
for $k=r+1$,
together with the explicit form of 
$U(1,1,x)=U(-1,-1,-x)=\begin{pmatrix}1 & 0\\ x & 1\end{pmatrix}$ which has only polynomial entries of $x$.
This also implies that 
\begin{equation}\label{detpi}
\det(\Pi_{r+1}(\{a\}))=\det(P_{r+1}(\{a\})=(-1)^{r(r+1)(r+2)\over 2}\end{equation}
by Lemma \ref{deterP}.

To compute $\Pi_{r+1}(\{a\})$, we use the pictorial representation II
\eqref{rep2UV} for the the non-zero matrix elements of the $U,V$ matrices, 
and we note that the network chips for the $UV$ and $VU$ diamonds \eqref{UVm} and \eqref{VUm} may be represented as:
\begin{eqnarray}
U(a,b,v)V(u,b,c)&=&\raisebox{-1.1cm}{\hbox{\epsfxsize=2.5cm \epsfbox{UVm.eps}}}
=\raisebox{-1.1cm}{\hbox{\epsfxsize=2.5cm \epsfbox{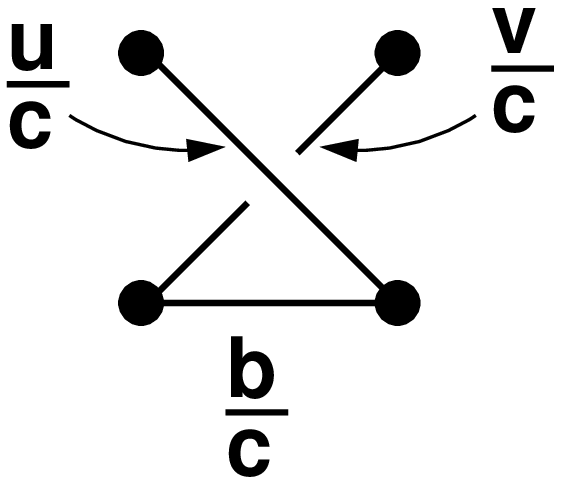}}} \label{UVnetrep} \\
V(u,a,b)U(b,c,v)&=&\raisebox{-1.1cm}{\hbox{\epsfxsize=2.5cm \epsfbox{VUmp.eps}}}
=\raisebox{-1.1cm}{\hbox{\epsfxsize=2.5cm \epsfbox{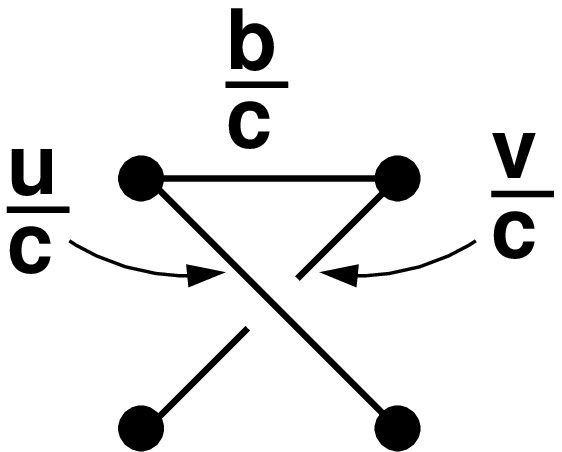}}} \label{VUnetrep} 
\end{eqnarray}
where the missing horizontal edge on the regularized network is due to vanishing condition $uv+ac=0$.
The two different ($UV$ or $VU$) diamond decompositions of $\Pi_{r+1}(\{a\})$ in pictorial representation II, in the case
$r=8$ are:
$$ \raisebox{-2.cm}{\hbox{\epsfxsize=8.cm \epsfbox{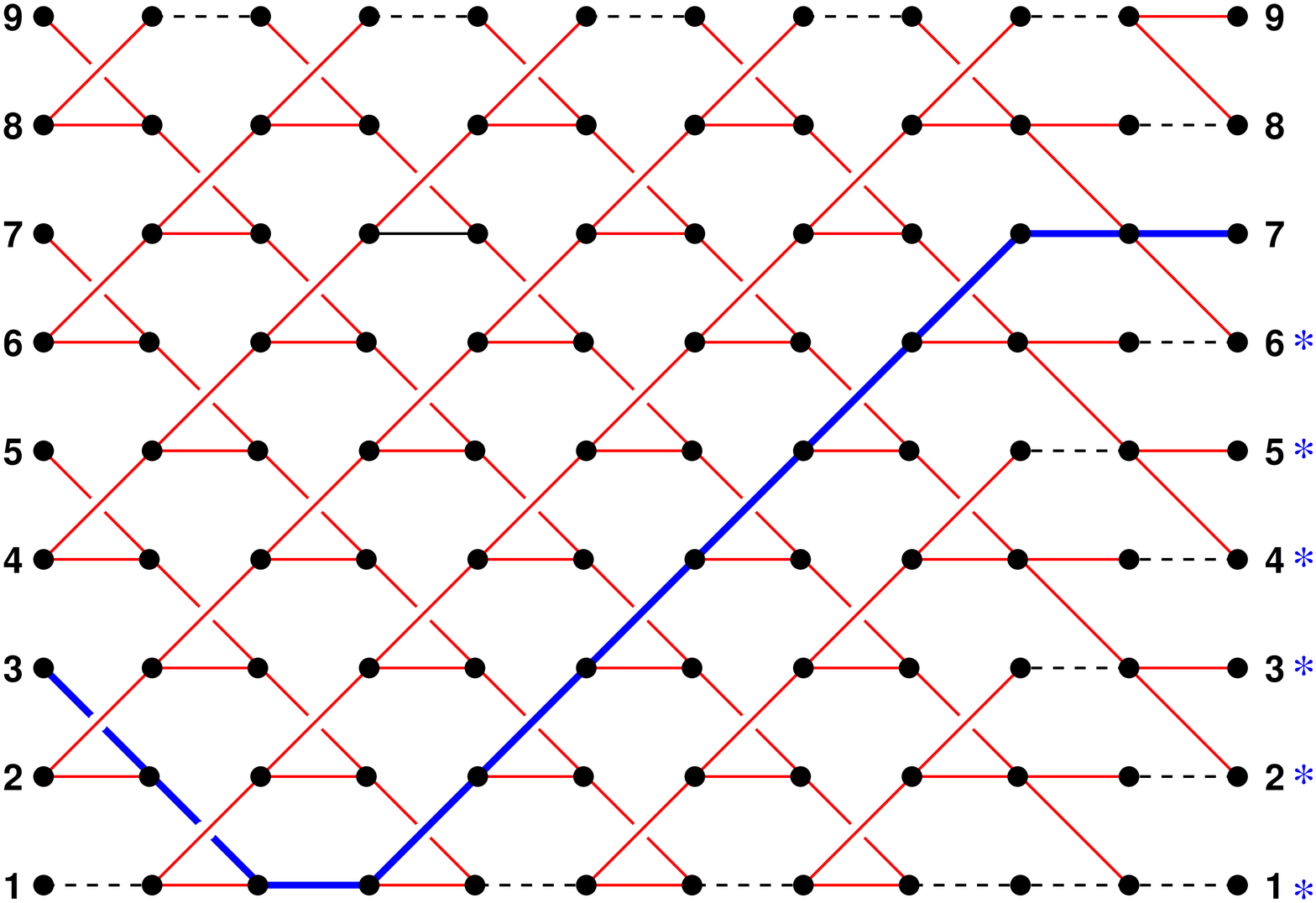}}}=\raisebox{-1.95cm}{\hbox{\epsfxsize=8.cm \epsfbox{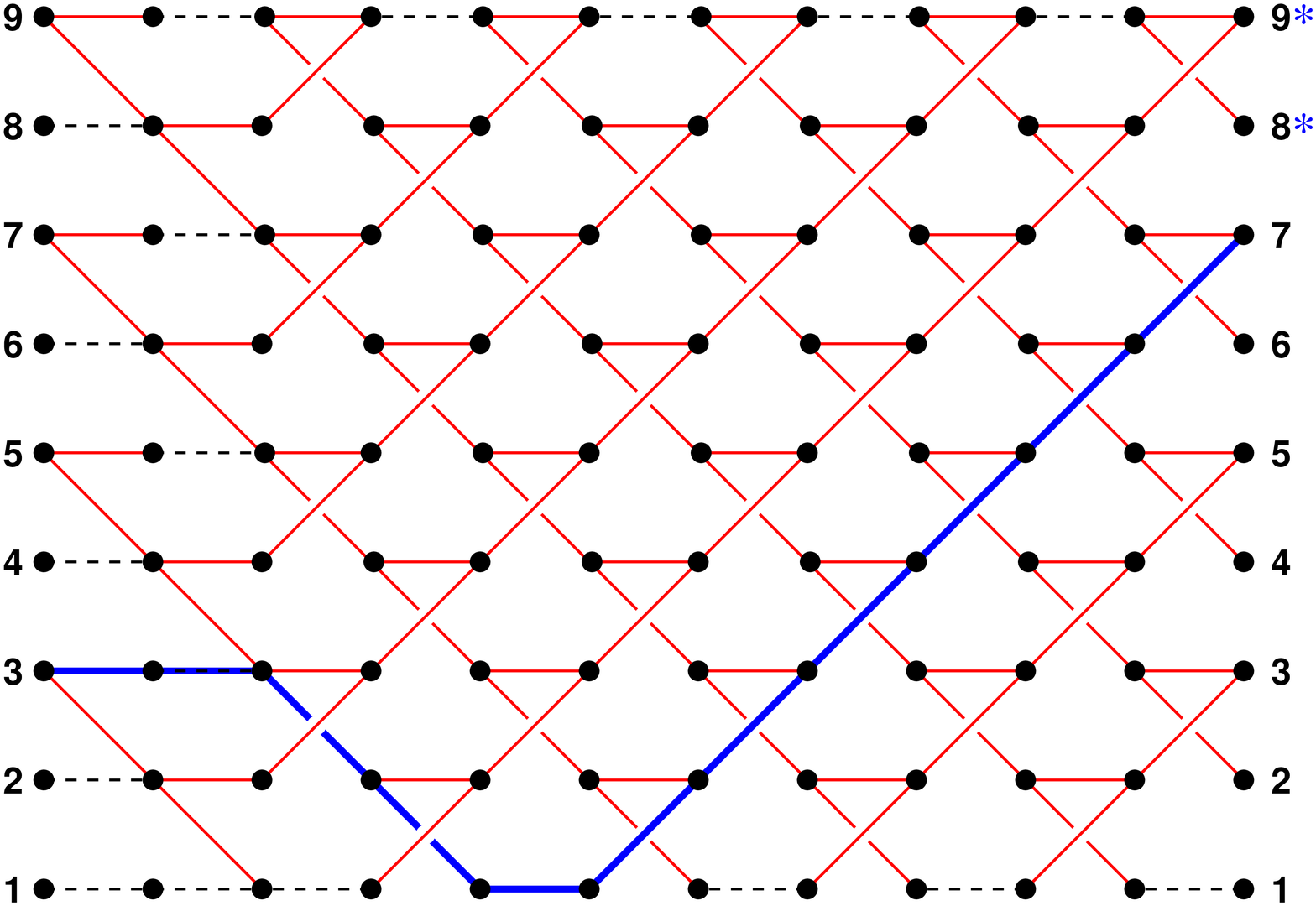}}}$$
where each edge is weighted with a Laurent monomial of the $a$'s. 
The left diagram shows that there are paths from vertex $i$ on the left to vertex $i'$ on the right only if 
$i'\leq r+2-i$, whereas the right diagram shows that there are paths from $i$ to $i'$ only if $i'\geq r+2-i$, for each $i$.
Therefore there are non-zero weighted paths only from each vertex $i$ to $r+2-i$. The corresponding path is unique.  Equivalently, 
$\Pi_{r+1}(\{a\})_{i,j}=0$ unless $j=r+2-i$.
This is illustrated in the above picture by highlighting in thick solid blue line the unique contributing path $3\to 7$,
while the other attainable points via paths starting at vertex $3$ are indicated by blue asterisks. 

Moreover, the total weight of the single contributing
path $i\to r+2-i$, equal to the matrix element $\big[\Pi_{r+1}(\{a\})\big]_{i,r+2-i}$,
is a monomial of the $a$'s (with only non-negative powers, as the entries of $\Pi_{r+1}(\{a\})$ are all polynomials). 
The determinant of $\Pi_{r+1}(\{a\})$ reads:
$$ \det(\Pi_{r+1}(\{a\})) =(-1)^{r(r+1)/2}\, \prod_{i=1}^{r+1} \Pi_{r+1}(\{a\})_{i,r+2-i} $$
Comparing this with \eqref{detpi}, we see that none of the matrix elements $\Pi_{r+1}(\{a\})_{i,r+2-i}$
vanish, and each of them has value $\pm 1$. 
To conclude, we note that the face weights cancel out along the paths as the product over step weights is telescopic,
leaving us with only the ratio: (leftmost face variable)$/$(rightmost face variable).
Inspecting the signs from the boundary faces carefully, we finally conclude that $\Pi_{r+1}(\{a\})=P$.
\end{proof}

\begin{cor}\label{corP}
We have
$$\lim_{a_1,a_2,...,a_r\to 0} \ P_{r+1}(\{a\})=P$$
\end{cor}
\begin{proof}
We use the expressions of Theorem \ref{Prplus}, and note that $U_i(1,1,0)=U(-1,-1,0)=\mathbb I$ for all $i\in [1,r]$.
\end{proof}

\begin{lemma}\label{pj}
For each $j\in[1,r]$, 
$$(P_j)_{1,i}=\lim_{a_1,a_2,...,a_r\to 0} \left(P_j(\{ a\})\right)_{1,i}=\delta_{i,2\lfloor {j\over 2}\rfloor+1 }.$$ 
\end{lemma}
\begin{proof} We give a pictorial proof.
Use the diamond decomposition of the networks \eqref{squarePj}
in pictorial representation II with chips \eqref{VUnetrep}. For even and odd $j$'s, we get respectively (here $j=4,5$):
$$ \raisebox{-3.cm}{\hbox{\epsfxsize=6.5cm \epsfbox{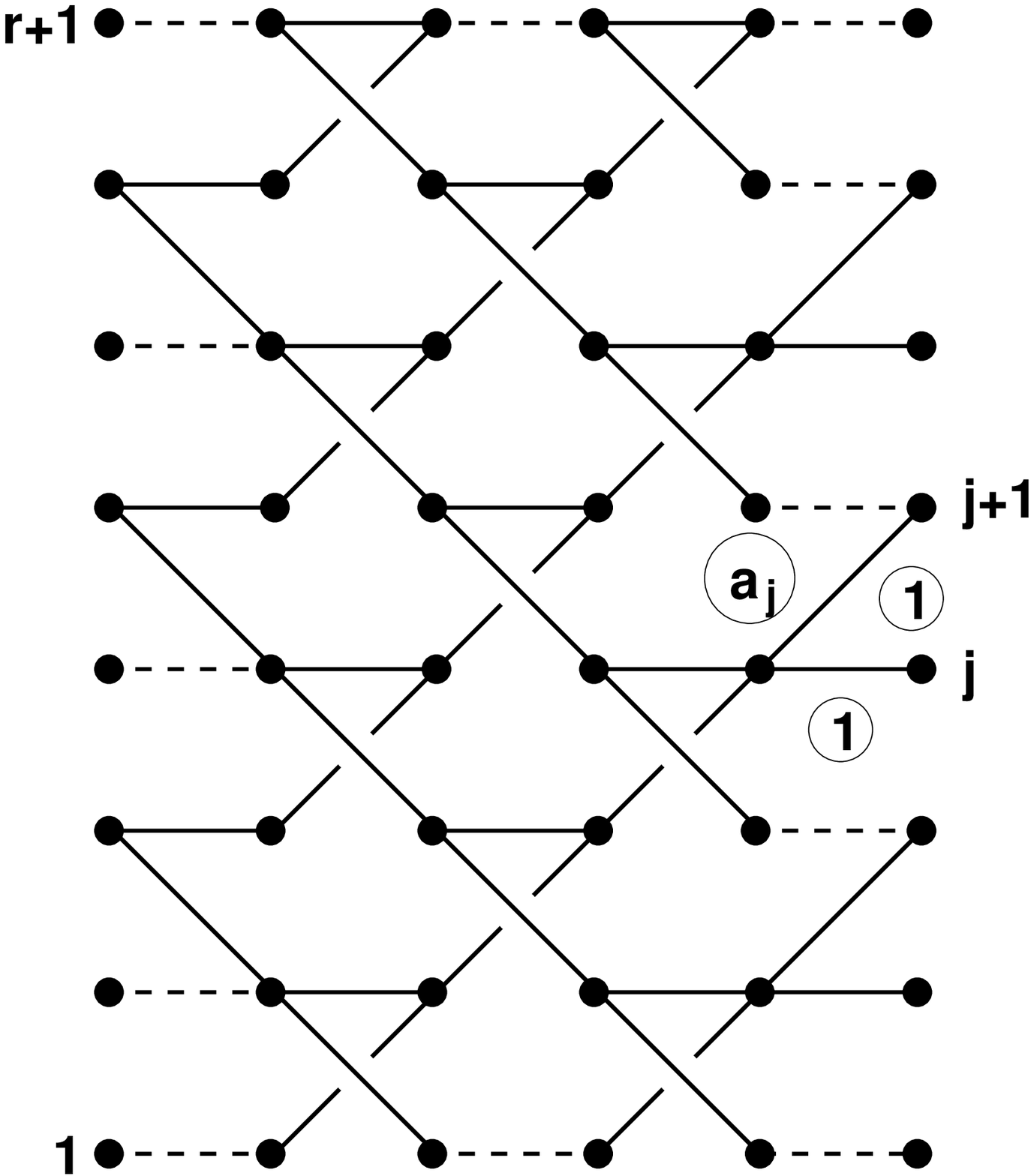}}}
\qquad \raisebox{-3.cm}{\hbox{\epsfxsize=7.8cm \epsfbox{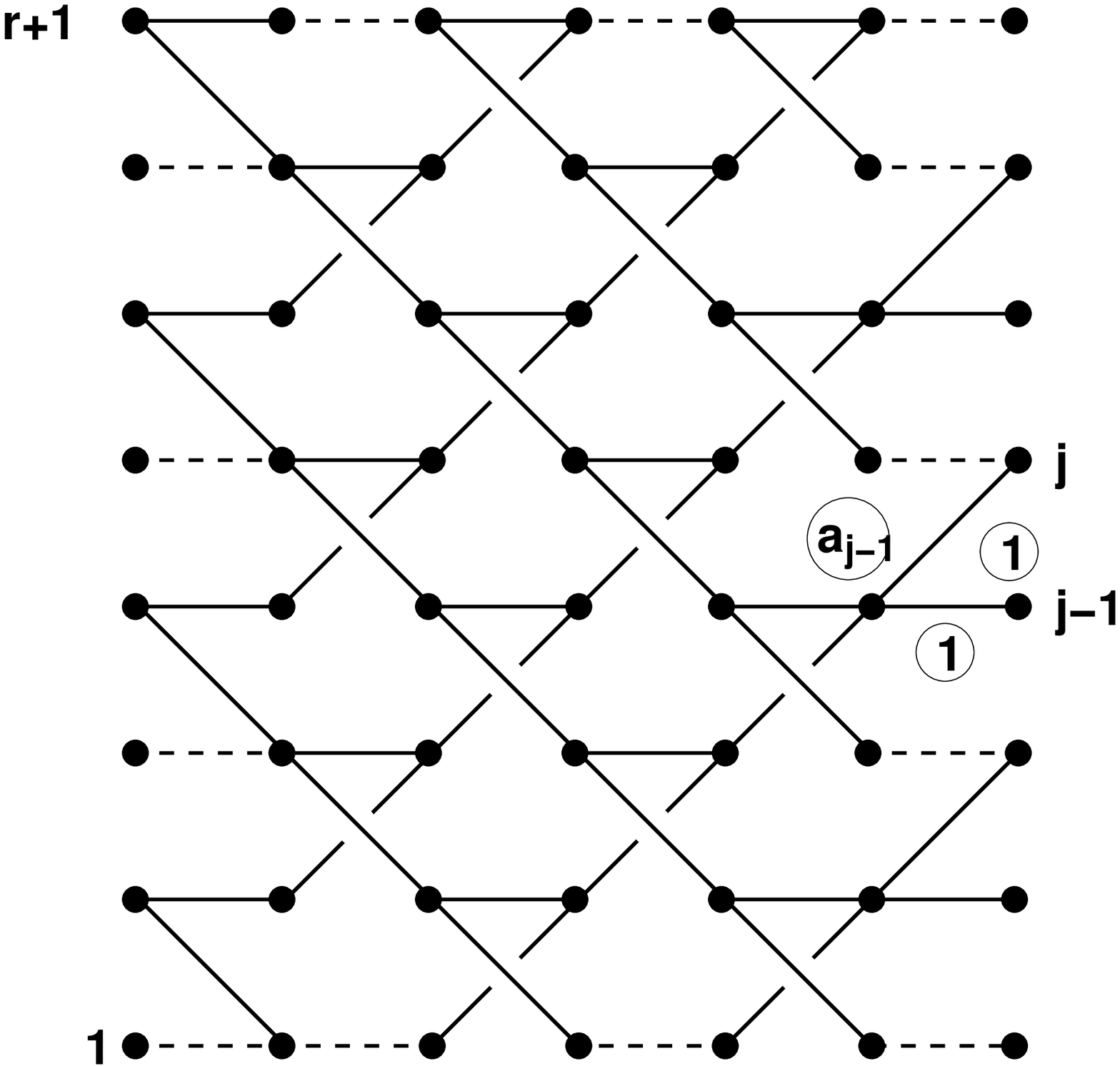}}}   $$
There are exactly two paths from $1\to k$, with $k=j,j+1$ if $j$ is even, and $k=j-1,j$ is $j$ is odd.
The face labels of the last chip are represented inside circles. The weights of the two
paths in the case of even (resp. odd) $j$
are identical except for their last step, weighted respectively by: $1$ if the last step is diagonal
and $a_{j}$ (resp. $a_{j-1}$) if it is horizontal. Therefore only the path ending with a diagonal step
contributes in the limit $a_i\to 0$. Moreover, the weights along this remaining
path, due to \eqref{VUnetrep}, are of the form $v/c$ and therefore their product is telescopic and leaves
us with (leftmost face label)$/$ (rightmost face label)$=1$. This proves the lemma.
\end{proof}

\subsection{The reflected network matrix}

We can give a similar definition of the
regularized network matrix for $N(\ell+1,\ell+r+2)$ of the form $N(\ell+1,\ell+r+2)(\{b\})$ for a compatible
array $(b_{i,j})$. 
In order to satisfy the boundary conditions in the initial data $\bt^{-}$, choose
$N(\ell+1,\ell+r+2)(\{b\})$ to be the regularized network matrix with $b_{i,j}$ an array defined as follows:
$$
b_{i,j} = (-1)^{r i} a_{i,j-(\ell+r+2)}, \qquad i\in [0,r+1], j\in[\ell+1,\ell+r+2].
$$
where $a_{i,j}$ is given by Definition \ref{defa}. This ensures that
$b_{i,\ell+1}=1$ and $b_{i,\ell+r+2}=(-1)^i$ when $r$ is odd.

Let $S$ be the matrix with entries:
\begin{equation}\label{Smat}
[S]_{i,j}=(-1)^{i-1}\delta_{i,j}  \qquad (i,j\in [1,r+1]).
\end{equation}
Clearly, 
\begin{equation}\label{Srela} S^2={\mathbb I} \qquad S\, P=(-1)^r P\, S.
\end{equation}
Moreover, 
\begin{eqnarray} S \, U_i(a,b,c) \, S &=&U_i(a,b,-c)=U_i(-a,-b,c) \label{SUS}, \\
 S \, V_i(a,b,c) \, S &=&V_i(a,-b,-c)=V_i(-a,b,c).\label{SVS}
\end{eqnarray}

\begin{lemma}\label{otherP}
The regularized network matrix ${\tilde P}_{r+1}(\{b\})=N(\ell+1,\ell+r+2)(\{b\})$ defined above satisfies:
\begin{equation} \lim_{b_1,...,b_r\to 0} {\tilde P}_{r+1}(\{b\}) =(-1)^r\, P.
\end{equation}
\end{lemma}
\begin{proof}
The lemma is clear for even $r$, as $N(\ell+1,\ell+r+2)(\{b\})=P_{r+1}(\{a\})$ with $b_{i,j}=a_{i,j-(\ell+r+2)}$.
For odd $r$, we have:
$$ S P_{r+1}(\{a\}) S = N(\ell+1,\ell+r+2)(\{b\}) \qquad {\rm where} \qquad  b_{i,j}=(-1)^i a_{i,j-(\ell+r+2)},$$
with $S$ as in \eqref{Pmat}.  Indeed, 
eqns. \eqref{SUS} and \eqref{SVS} allow us to interpret the conjugate action of $S$ as flipping the sign
of all array entries along every other row, say $i=1,3,...r$. Taking the $a\to 0$ limit in both cases
leads respectively to $P_{r+1}(\{0\})=P$ for even $r$ and $SPS=-P$ for odd $r$ by \eqref{Srela}, and the lemma follows.
\end{proof}

It will also be useful to have the corresponding version of Lemma \ref{pj}. Define the 
family of regularized network matrices 
${\tilde P}_j(\{b\})= N(\ell+1,\ell+1+j)(\{b\})$, $j=0,1,2,...,r$, with ${\tilde P}_0(\{b\})={\mathbb I}$.
Each $b_{i,j}$ is a signed monomial of the variables $\{b_k:=b_{k,\ell+2}\}$. In particular,
$b_{1,\ell+1+j}=(-1)^{j-1} b_j$.

\begin{lemma} \label{otherPlim}
The limit $b_j\to 0$ of the regularized network matrices is
\begin{equation}
\lim_{b_1,...,b_r\to 0} \Big[{\tilde P}_j(\{b\}) \Big]_{i,1} b_{1,\ell+1+j}=\delta_{i,a_\ell(j)} 
\qquad j\in[1,r].
\end{equation}
where
\begin{equation}\label{defaell}
a_\ell(x)=\left\{ \begin{matrix} & 2\lfloor \frac{x+1}{2} \rfloor & {\rm if}\, \ell \, {\rm is}\, {\rm even} \\
& 2\lfloor \frac{x}{2} \rfloor+1 & {\rm if}\, \ell \, {\rm is}\, {\rm odd} \end{matrix} \right. 
\end{equation}
\end{lemma}
\begin{proof}
The proof is very similar to that of Lemma \ref{pj}
The difference is that
one must distinguish between odd $\ell$ (the actual reflection of the case of Lemma \ref{pj}) and
even $\ell$, in which $U$ and $V$ matrices are interchanged. The telescopic products of weights
for the remaining path ending at $1$ is $1/b_{1,\ell+1+j}$, where and the denominator is cancelled by the
prefactor above.
\end{proof}

\subsection{Collapse relations}

The following relations may be verified by direct calculation.
\begin{eqnarray} P \, U_i(a,b,c) \, P &=&V_{r+1-i}((-1)^{r-1} c,a,b) \label{PUPa}, \\
 P \, V_i(a,b,c) \, P &=&U_{r+1-i}(b,c,(-1)^{r-1}a) \label{PUPb}.
\end{eqnarray}

\begin{lemma}\label{collapse}
Let $i\in [1,r]$ and $j\geq 1$. Given initial data of the form $\bt^+$,
\begin{eqnarray*}
V_{r+1-i}(t_{r-i,-r-j},t_{r+1-i,-r-1-j},t_{r+1-i,-r-j})PU_i(t_{i,j-1},t_{i,j},t_{i+1,j-1})&=& P,\\
U_{r+1-i}(t_{r+1-i,-r-1-j},t_{r+1-i,-r-j},t_{r+2-i,-r-1-j})PV_i(t_{i-1,j},t_{i,j-1},t_{i,j})&=& P.
\end{eqnarray*}
\end{lemma}
\begin{proof}
Multiplying the relation \eqref{PUPa} from the left by the inverse of $V_{r+1-i}((-1)^{r-1}c,b,a)$ 
and similarly \eqref{PUPb} by the inverse of $U_{r+1-i}(b, c, (-1)^{r-1}a)$ using 
\eqref{invUV} gives:
$$V_{r+1-i}((-1)^rc,b,a)P U_i(a,b,c)=P,\qquad U_{r+1-i}(c,b,(-1)^ra)P V_i(a,b,c)=P.$$
The reflection symmetry on $\bt^+$ \eqref{tsymii} means that
\begin{equation}\begin{matrix}
t_{r-i,-r-j}=(-1)^{r(i+1)}\, t_{i+1,j-1}, & t_{r+1-i,-r-1-j}=(-1)^{ri}\, t_{i,j}, \\
 t_{r+1-i,-r-j}=(-1)^{ri}\, t_{i,j-1}, &
t_{r+2-i,-r-1-j}=(-1)^{r(i-1)}\, t_{i-1,j}.\end{matrix}\end{equation}
The Lemma follows from the projective property \eqref{proj} with $\lambda =(-1)^{ri}$.
\end{proof}

\subsection{Proof of Lemma \ref{suffiT}}\label{arthree}
We prove the two statements in the Lemma.
\begin{lemma} The solutions of the unrestricted $A_r$ $T$-system of type (i) 
with initial conditions $X(\bt^+)$ have the property that $T_{1,0,k}=1$ for all odd $k$.
\end{lemma}
\begin{proof}
By reflection symmetry, it is only necessary to consider $k>0$.
The projection of the point $(1,0,k)$ onto $\bk_0$ 
is the portion with $j\in [-k+1,k-1]$: 
$$T_{1,0,k}=\left[N(-k+1,k-1)\right]_{1,1}t_{1,k-1}.$$ There are two cases to consider.

\noindent{\em  case 1:}
$k-1>r$. In this case,
\begin{eqnarray*}
N(-k+1,k-1)&=&N(-k+1,-r-1)N(-r-1,0)N(0,k-r-2)N(k-r-2,k-1)\\
&=&N(-k+1,-r-1)PN(0,k-r-2)N(k-r-2,k-1)
\end{eqnarray*}
using Corollary \ref{corP}. Lemma \ref{collapse} implies $N(-k+1,-r-1)PN(0,k-r-2)=P$ for the initial data $\bt^+$.
We deduce that $T_{1,0,k}=\left[N(k-r-2,k-1)\right]_{r+1,1}t_{1,k-1}$. Let us examine the network
corresponding to $N(k-r-2,k-1)$. As before, let us decompose it into diamonds of the form
$$U(a,b,v)V(u,b,c)=\raisebox{-1.1cm}{\hbox{\epsfxsize=2.5cm \epsfbox{UVm.eps}}}=
\begin{pmatrix}
{b\over c} & {u\over c} \\
{v\over c} & {uv+ac\over bc}
\end{pmatrix} =\raisebox{-1.1cm}{\hbox{\epsfxsize=2.5cm \epsfbox{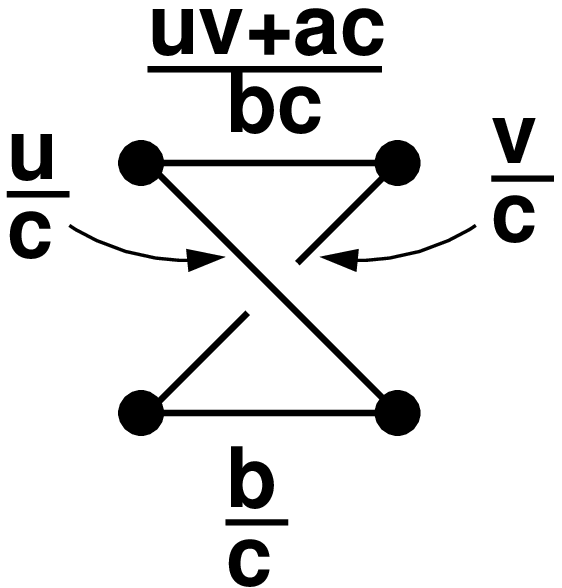}}}  $$
Note that as the arguments are generic, the quantity $uv+ac$ does not vanish like in the
$UV$ diamond of \eqref{UVnetrep}.
As the network for $N(k-r-2,k-1)$ is a square, we have a decomposition of the form (say for $r$ even):
$$\raisebox{-1.5cm}{\hbox{\epsfxsize=4.cm \epsfbox{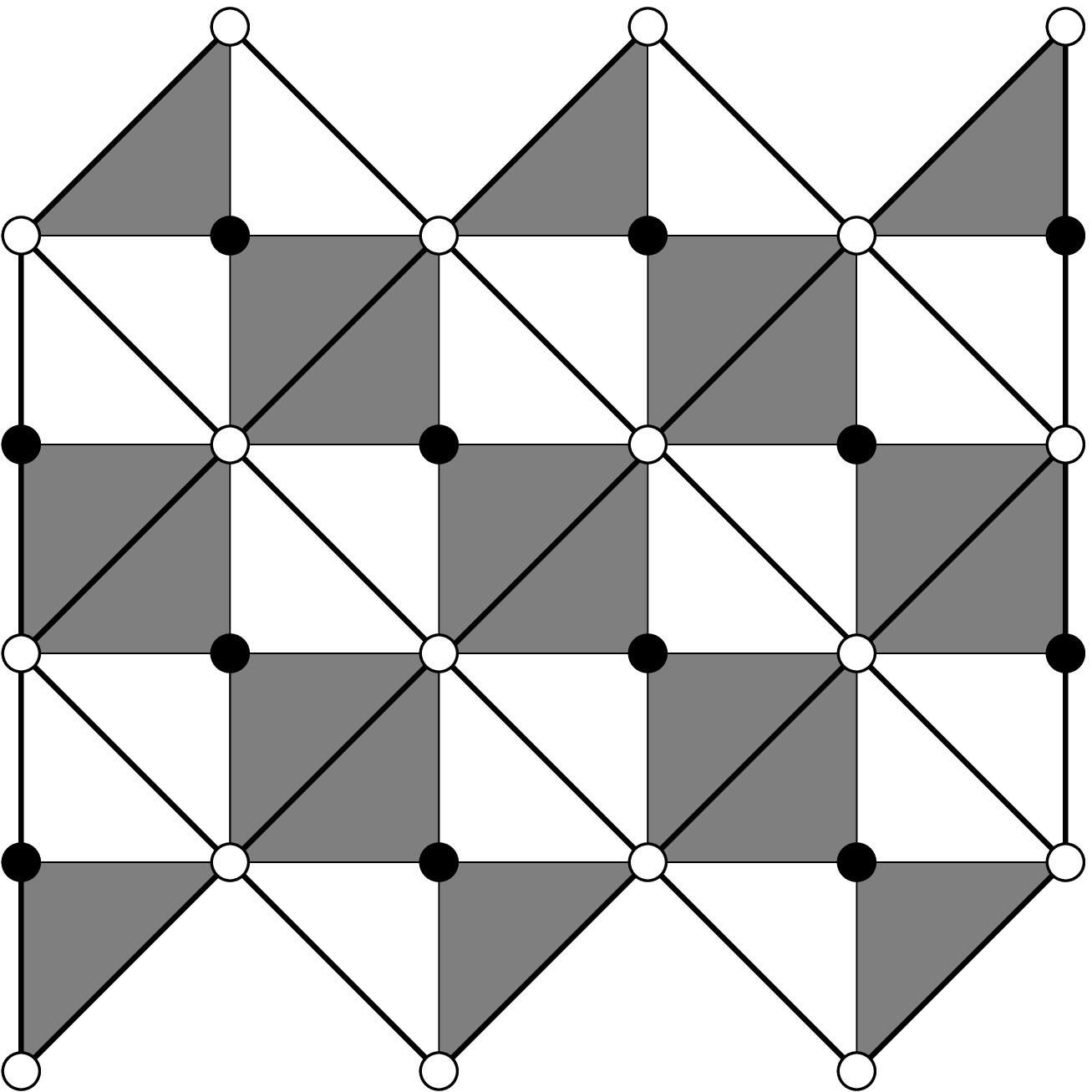}}} \quad 
=\quad \raisebox{-1.1cm}{\hbox{\epsfxsize=5.5cm \epsfbox{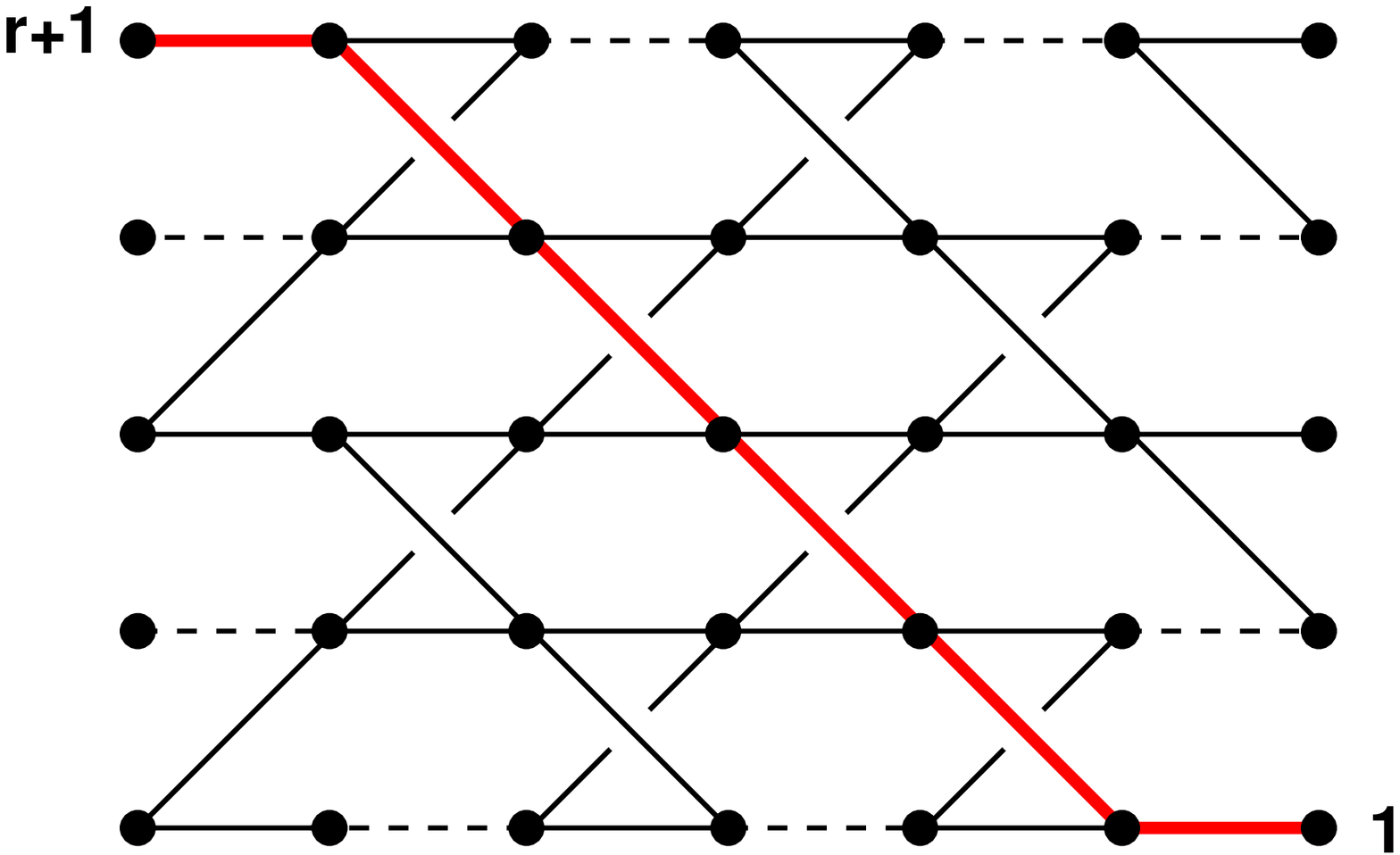}}}$$
The matrix element $(r+1,1)$ corresponds to the unique path from $r+1\to 1$, highlighted in red.
The product of weights along the path is telescopic, and leaves us only with (leftmost face label)$/$(rightmost face label)
$=t_{r+1,k-r-2}/t_{1,k-1}=1/t_{1,k-1}$, as we have $t_{r+1,j}=1$ for all $j$. We conclude that $T_{1,0,k}=1$.

\noindent{\em Case 2:} $0<k-1< r+1$. In this case,
$N(-k+1,k-1)=N(-k+1,0)N(0,k-1)=P_{k-1}N(0,k-1)$. Then
$$T_{1,0,k}=\left[N(-k+1,k-1)\right]_{1,1}t_{1,k-1}
=\left[N(0,k-1)\right]_{2\lfloor{k-1\over 2}\rfloor+1,1}t_{1,k-1}$$
by Lemma \ref{pj}. Noting that $2\lfloor{k-1\over 2}\rfloor+1=k$, it is easy to see that, again, a unique path contributes to this,
as the paths $k\to 1$ only ``see"  the lower triangle part of the network, with vertices $(i,j)=(1,0),(1,k-1),(k,0)$,
represented below:
$$ \raisebox{-1.5cm}{\hbox{\epsfxsize=5.cm \epsfbox{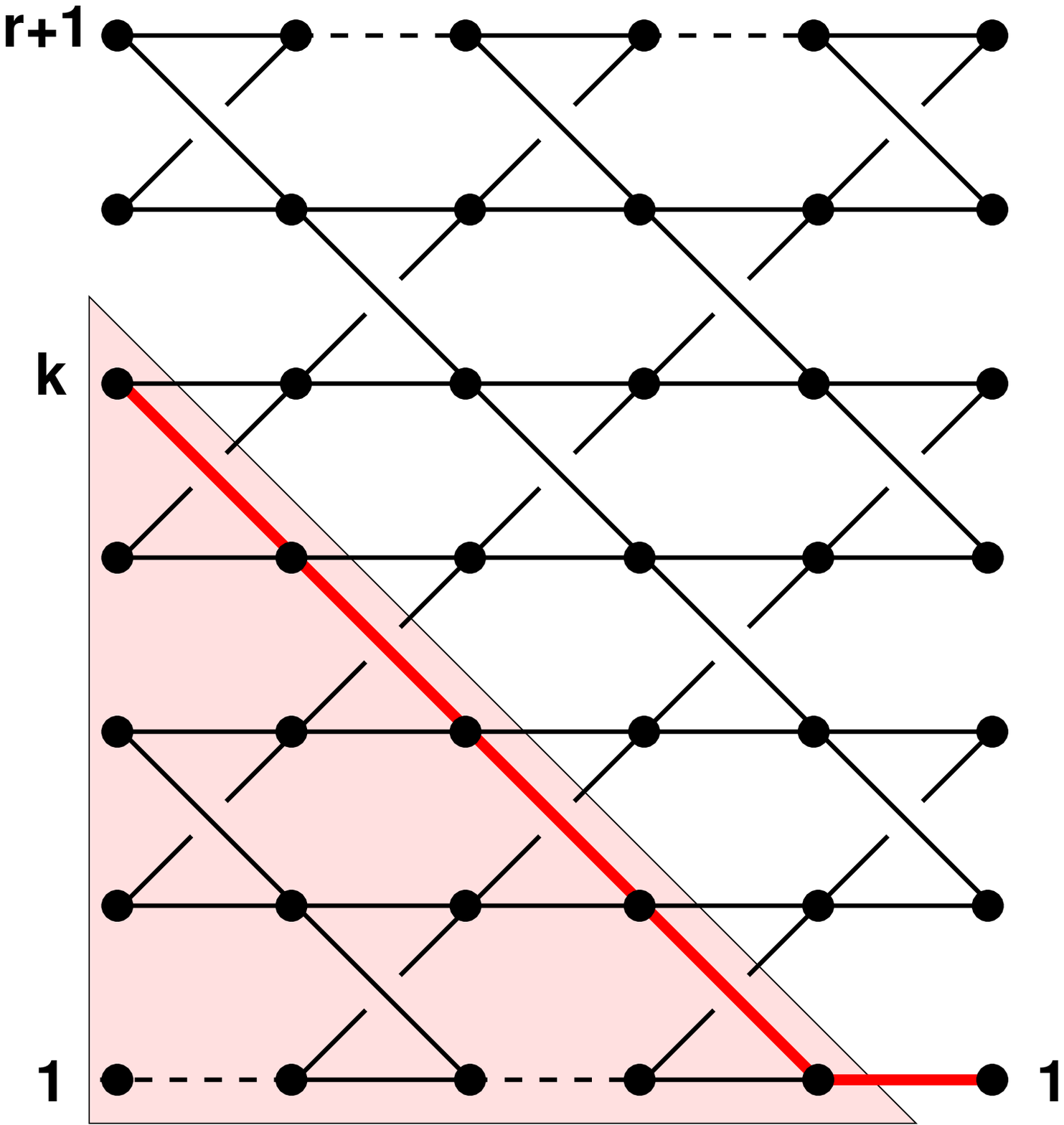}}}  $$
The total weight of this path is equal to (leftmost face label)$/$(rightmost face label)$= t_{0,k}/t_{1,k-1}=1/t_{1,k-1}$,
which implies $T_{1,0,k}=1$.
\end{proof}

\begin{lemma} The solutions of the unrestricted $A_r$ $T$-system of type (i) with initial conditions $X(\bt^+)$ have the property that $T_{1,-j,k}=0$ for all $j\in[1,r]$
\end{lemma}
\begin{proof}
Writing
$$T_{1,-j,k}=\left[N(-j-k+1,k-j-1)\right]_{1,1}t_{1,k-j-1},$$
there are five regions
for the point $(-j,k)$, which are depicted in Fig.\ref{fig:regions}.

\begin{figure}
\centering
\includegraphics[width=7.cm]{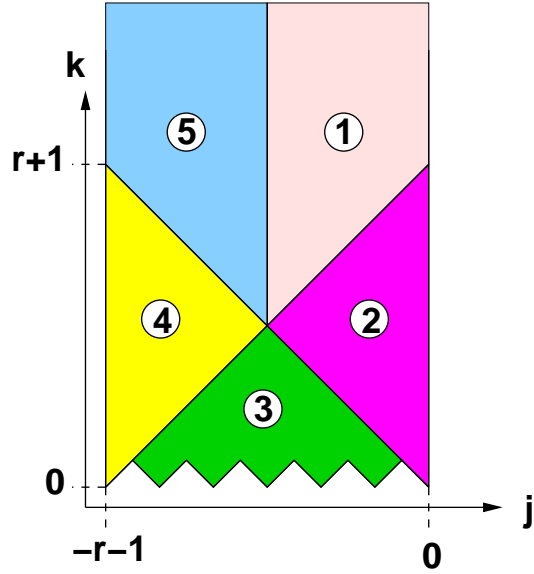}
\caption{The five regions in $(j,k)$ plane for the proof of $T_{1,-j,k}=0$.}\label{fig:regions}
\end{figure}

\noindent{\em Region 1:} $1\leq j\leq {r+1\over 2}$
and  $k\geq r+1-j$. In this case,
$$N(-j-k+1,k-j-1)=N(-j-k+1,-r-1)N(-r-1,0)N(0,k-j-1)=PN(k+j-r-2,k-j-1),$$
where we have used Lemma \ref{collapse}.
This yields $$T_{1,-j,k}=\left[N(k+j-r-2,k-j-1)\right]_{r+1,1}t_{1,k-j-1}.$$ 
The network for $N(k+j-r-2,k-j-1)$,
once decomposed into $VU$ diamonds as above, looks like:
$$  \raisebox{-1.5cm}{\hbox{\epsfxsize=9.cm \epsfbox{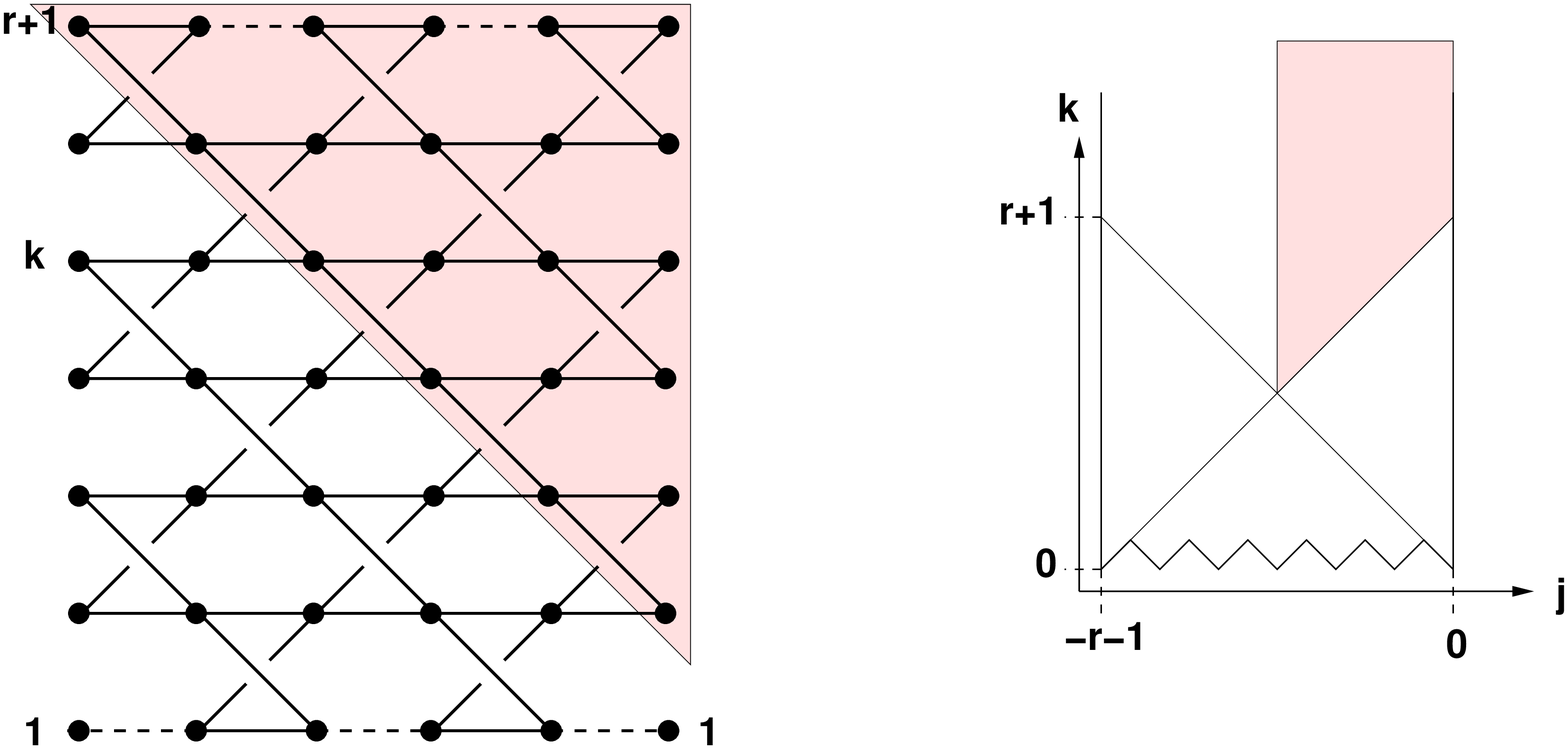}}}    $$
where we have represented a typical network in the diamond chip representation, and the corresponding region $1$
in $(j,k)$ space.
We have shaded the range of the paths from $r+1$.
The network  is in  a rectangle of width strictly smaller than its height (size $(r+1-2j)\times (r+1)$), hence there are
no path joins $r+1\to 1$. Therefore $T_{1,-j,k}=0$.

\noindent{\em Region 2:} $r+1-j>k\geq j$. Since
$$N(-j-k+1,k-j-1)=N(-j-k+1,0)N(0,k-j-1)=P_{j+k-1}N(0,k-j-1),$$ we have 
$$T_{1,-j,k}=\left[N(0,k-j-1)\right]_{2\lfloor{j+k-1\over 2}\rfloor+1,1}t_{1,k-j-1}.$$ 
The width of the network is
$k-j-1<2\lfloor{j+k+1\over 2}\rfloor+1=k+j$,
hence no path contributes, and $T_{1,-j,k}=0$.

\noindent{\em Region 3:} $0\leq k<j$
and  $k<r+1-j$. As $-r-1<-j-k+1<k-j-1<0$, we may use the matrix $N(-k-j+1,k-j-1)(\{a\})$
of the regularized network, with labels \eqref{eani}. The paths $1\to 1$ only see the triangle shaded
in the typical configuration below:
$$ \raisebox{-1.5cm}{\hbox{\epsfxsize=10.cm \epsfbox{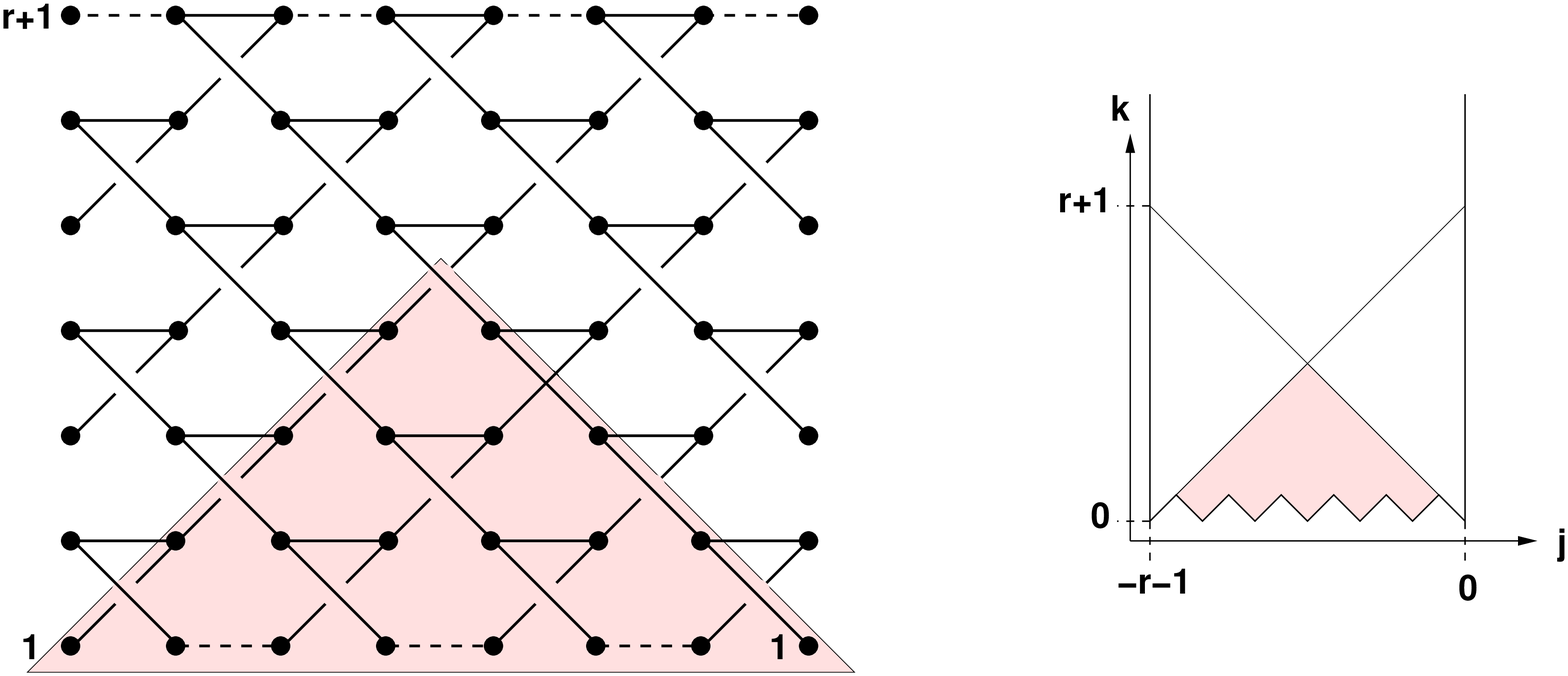}}}    $$
and it is clear that no path can go from $1\to 1$, hence 
$$T_{1,-j,k}=\left[N(-k-j+1,k-j-1)\right]_{1,1} t_{0,k-j-1}=0.$$

\noindent{\em Region 4:} $r+1-j\leq k\leq j$. We have the decomposition $N(-k-j+1,k-j-1)=N(-k-j+1,-r-1)N(-r-1,k-j-1)$.
As before, we may use the matrix $N(-r-1,k-j-1)(\{a\})$
of the regularized network, with labels \eqref{eani}. In the square decomposition, the complete network 
looks like:
$$  \raisebox{-1.5cm}{\hbox{\epsfxsize=12.cm \epsfbox{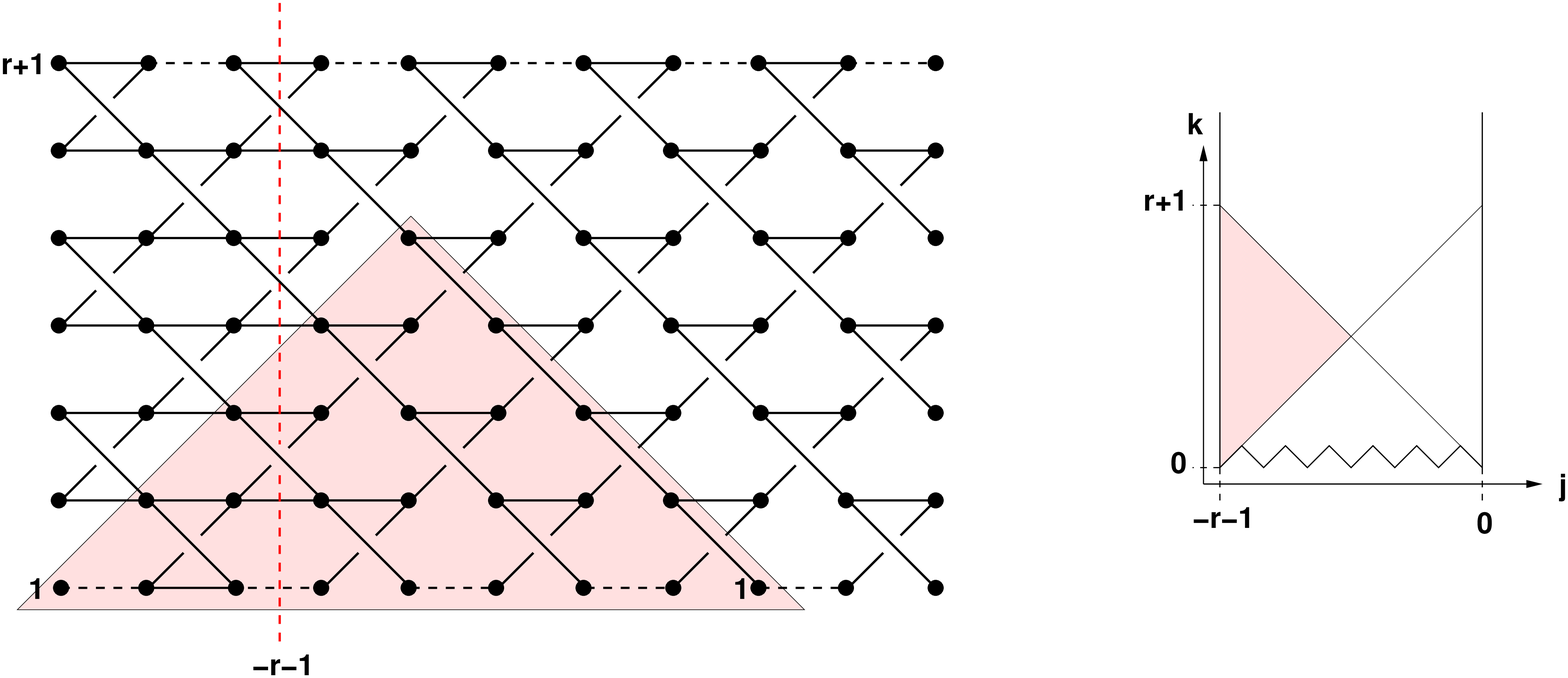}}}    $$
and there are no paths $1\to 1$, as in region 3.

\noindent{\em Region 5:} $r+1>j>{r+1\over 2}$ and $k>j$.
We have
$N(-j-k+1,k-j-1)=N(-j-k+1,-r-1)PN(0,k-j-1)=N(-j-k+1,j-k-r)P$, so that $T_{1,-j,k}=\left[N(-j-k+1,j-k-r)\right]_{1,r+1}$. Again,
the width of the network is $2j-r-1\leq r-1$, hence there is no path from $1\to r+1$, and
$T_{1,-j,k}=0$. 
\end{proof}

\subsection{Proof of the periodicity Theorem \ref{priociT}}\label{arfour}

Let $T_{i,j,k}$ be the solution of the $\ell$-restricted $A_r$ $T$-system (iv).
Using Theorem \ref{inithm}, it is equal to the solution $T_{i,j,k}$ of
the unrestricted $A_r$ $T$-system (i) subject to initial conditions $X(\bt^{[1,\ell]})$ on $\bk_0$.
We can also use the initial conditions on 
any integer translate $\bk_0+2m$, $m\in \Z$, of this surface.
Let $N=2(\ell+r+2)$.
Due to the determinant formula \eqref{wronsk} it is sufficient to consider $i=1$.
As in the $A_1$ case, we prove the more general half-periodicity theorem.

\begin{figure}
\centering
\includegraphics[width=16.cm]{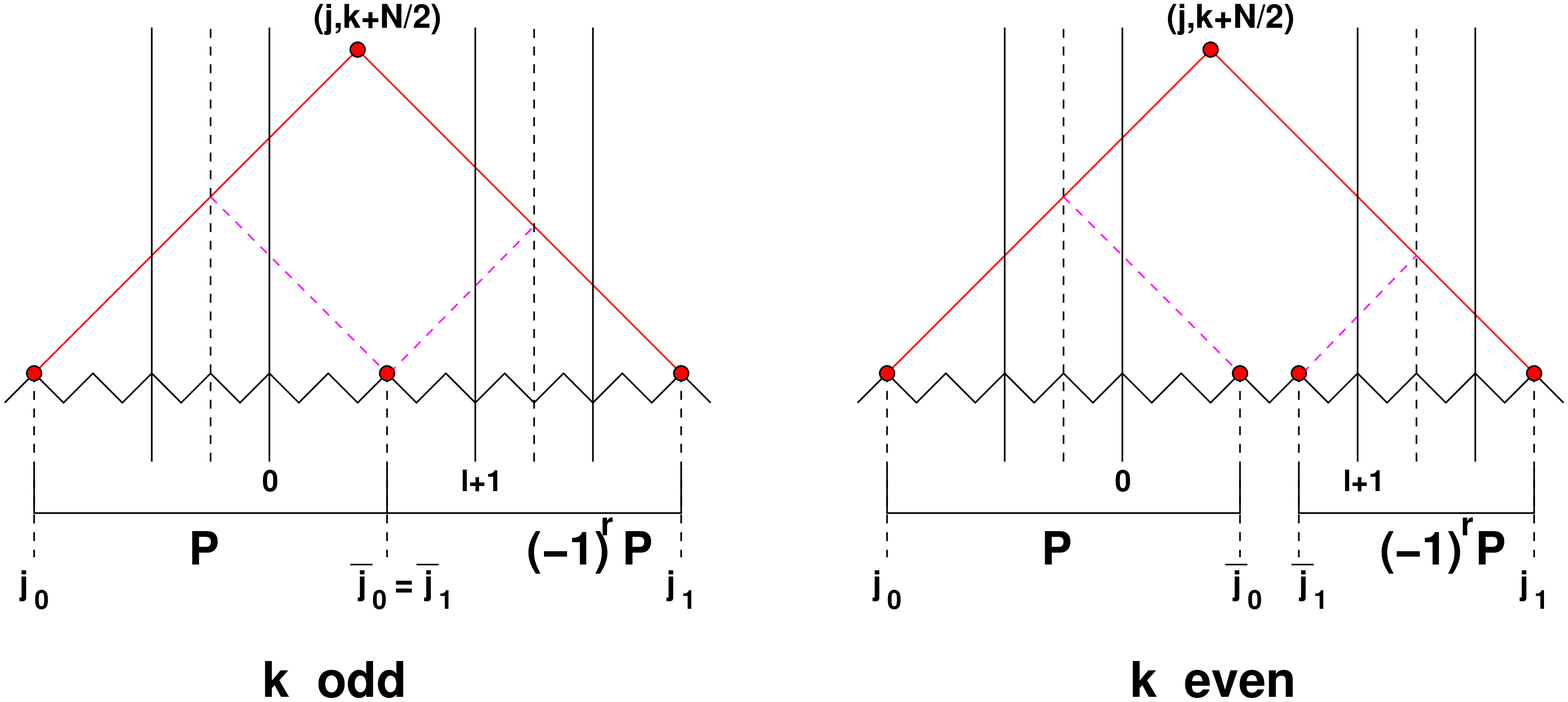}
\caption{The projection $[j_0,j_1]$ of $(1,j,k+N/2)$ onto $\bk$
is shown for $k$ odd and even respectively. For $k$ odd, $j_0=k-N/2$, $j_1=k+N/2$
and the reflections ${\bar j}_0={\bar j}_1=\ell+1-k$ coincide. For $k$ even, we have
$j_0=k-N/2+1$, $j_1=k+N/2-1$
and the reflections are ${\bar j}_0=\ell-k$ and ${\bar j}_1=\ell+2-k$. In both cases, we have indicated
the network matrices corresponding to the various segments.}\label{fig:fourposs}
\end{figure}

\begin{thm}\label{haper}
The solution of the $\ell$-restricted $T$-system satisfies $T_{1,j,k+\frac{N}{2}}=T_{r,\ell+1-j,k}$ 
for all $k\in \Z$, $j\in [1,\ell]$ such that $j+k+\ell+r$ is odd.
\end{thm}
\begin{proof}
We use the network solution of Theorem \ref{soluT}.
Choose the surface $\bk=2m+\bk_0$ to be the unique translation of $\bk_0$
passing through the point $(r,\ell+1-j,k)$. The integer $m$ is fixed by requiring $k-2m=k^{(0)}_{r,\ell+1-j}$, and
$$
k_{x,y} = \left\{ \begin{array}{ll}k-(x+y+1\, {\rm mod}\, 2),& \hbox{$k$ odd;}\\
k+1-(x+y+1\, {\rm mod}\, 2), &\hbox{$k$ even}.\end{array}\right.
$$
The corresponding initial conditions are $T_{x,y,k_{x,y}}=t_{x,y}$,
where $t_{x,y}$ are initial conditions of type $\bt^{[1,\ell]}$.

Figure \ref{fig:fourposs} shows the projection of the point $(1,j,k+\frac{N}{2})$ onto $\bk$ in the cases when $k$ is even and odd. Let $\epsilon=1-k \mod 2$. Then
\begin{eqnarray*}
T_{1,j,k+\frac{N}{2}}=
 N\left(j-\frac{N}{2}+\epsilon,j+\frac{N}{2}-\epsilon\right)_{1,1} \, t_{1,j+\frac{N}{2}-\epsilon} .
\end{eqnarray*}

When $k$ is odd, we write
\begin{eqnarray*}
N\left(j-\frac{N}{2},j+\frac{N}{2}\right) &=&N\left(j-\frac{N}{2},\ell+1-j\right)N\left(\ell+1-j,j+\frac{N}{2}\right)\\
&=& N(j-(\ell+r+2),-r-1)N(-r-1,0)N(0,\ell+1-j) \\
&& \ \ \ \times N(\ell+1-j,\ell+1)N(\ell+1,\ell+r+2)N(\ell+r+2,j+\ell+r+2)\\
&=& P \times (-1)^r P=(-1)^r\,  {\mathbb I}
\end{eqnarray*}
where we have used Lemma \ref{collapse}.
Using 
$t_{1,j+\frac{N}{2}}=(-1)^r t_{r,\ell+1-j}$, 
$$T_{1,j,k+\frac{N}{2}} =(-1)^r\,  {\mathbb I}_{1,1} (-1)^r t_{r,\ell+1-j}=T_{r,\ell+1-j,k}.$$

When $k$ even, the splitting yields analogously:
\begin{eqnarray*}
&&\!\!\!\!\!\!\!\!\!\!\!\!\!\!\!\!\!\!\!\!\!  N\Big(j-\frac{N}{2}+1,j+\frac{N}{2}-1\Big)\\ 
&=&N\Big(j-\frac{N}{2}+1,\ell-j\Big)N(\ell-j,\ell+2-j)N\Big(\ell+2-j,j+\frac{N}{2}-1\Big)\\
&=& N(j-(\ell+r+1),-r-1)N(-r-1,0)N(0,\ell-j)N(\ell-j,\ell+2-j)  \\
&&\times N(\ell+2-j,\ell+1)N(\ell+1,\ell+r+2)N(\ell+r+2,j+\ell+r+1)\\
&=& P \, N(\ell-j,\ell+2-j) \, (-1)^r P
\end{eqnarray*}
and we get
$$T_{1,j,k+\frac{N}{2}} =(-1)^r\, N(\ell-j,\ell+2-j)_{r+1,r+1} \, t_{1,j+r+\ell+1}=N(\ell-j,\ell+2-j)_{r+1,r+1} \, t_{r,\ell+2-j}. $$
As $k$ is even, we have $\ell-j=r+1$ mod 2, and  the partition function 
$N(\ell-j,\ell+2-j)_{r+1,r+1}$ only depends on the top part of the network matrix, namely the $VU$ diamond:
$$ V_{r}(v,a,b)U_{r}(b,c,u)=\, \raisebox{-1.5cm}{\hbox{\epsfxsize=3.cm \epsfbox{VUmp.eps}}}\, =\, 
\raisebox{-1.5cm}{\hbox{\epsfxsize=3.cm \epsfbox{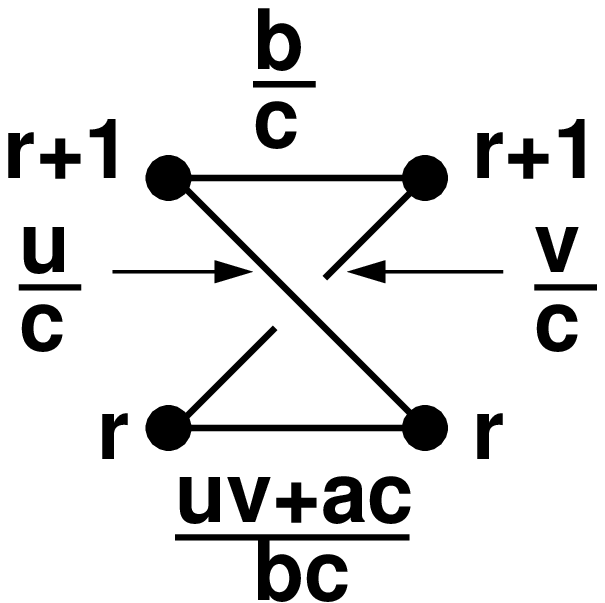}}} $$ 
with $b=t_{r,\ell+1-j}$ and $c=t_{r,\ell+2-j}$. 
Therefore,
$N(\ell-j,\ell+2-j)_{r+1,r+1}=t_{r,\ell+1-j}/t_{r,\ell+2-j}$, and
$$T_{1,j,k+\frac{N}{2}} = t_{r,\ell+1-j}/t_{r,\ell+2-j} \times t_{r,\ell+2-j} =T_{1,r,\ell+1-j}$$
\end{proof}

\begin{cor}
The solution of the $\ell$-restricted $A_r$ $T$-system satisfies the following half-periodicity relation:
\begin{equation}\label{otherperiod}
T_{i,j,k+\frac{N}{2}}= T_{r+1-i,\ell+1-j,k} \qquad (i\in [1,r],j\in [1,\ell],k\in \Z) 
\end{equation}
with $N=2(\ell+r+2)$.
\end{cor}
\begin{proof}
 Lemma \ref{wronsk} gives
\begin{equation}\label{dejac}
T_{i,j,k+\frac{N}{2}}=\det_{1\leq a,b \leq i}(T_{1,j+a-b,k+\frac{N}{2}+a+b-i-1})
=\det_{1\leq a,b \leq i}(T_{r,\ell+1-j-a+b,k+a+b-i-1}).
\end{equation}
Using the Desnanot-Jacobi identity \eqref{desnajac} it is possible to write solutions $T_{i,j,k}$
 with $i<r$ in terms of $(r+1-i)\times (r+1-i)$ determinants of the $T_{r,j',k'}$'s,
$$ T_{i,j,k}=\det_{1\leq a,b \leq r+1-i}(T_{r,j+a-b,k+a+b-r-2+i}) .$$
Threfore,
$$T_{r+1-i,\ell+1-j,k}=\det_{1\leq a,b \leq i}(T_{r,\ell+1-j+a-b,k+a+b-i-1}) ,$$
Comparing this with \eqref{dejac}, and noting that the transposed matrix has the same determinant
yields \eqref{otherperiod}.
\end{proof}

In particular, we have that $T_{r,j,k+\frac{N}{2}}=T_{1,\ell+1-j,k}$. Combining this with Theorem \ref{haper},
we deduce that $T_{1,j,k+N}=T_{1,j,k}$, and therefore $T_{i,j,k+N}=T_{i,j,k}$.
This completes the proof of the periodicity.

\subsection{Positivity: Proof of Theorems \ref{haposi} and \ref{positiT}}\label{arfive}

Theorem \ref{haposi} is the claim that solutions of the $A_r$ $T$-system of type (ii) are Laurent polynomials with non-negative integer coefficients of the initial data $\bt$.

\begin{lemma}\label{posibar} 
The solutions $T_{1,j,k}$ of the half-space $A_r$ $T$-system of type (ii) with initial conditions $X^+(\bt)$ 
\eqref{initcondii}
on the surface $\bk_0=\{(i,j,i+j\mod 2):i\in[1,r], j>0\}$ are Laurent polynomials of $\{t_{i,j}: i\in[1,r],j>0\}$ with non-negative integer coefficients.
\end{lemma}
\begin{proof} First, use Theorem \ref{plusimple} to identify $T_{1,j,k}$ as the solution of the $A_r$
$T$-system of type (i) with initial data $\bt^+$. Consider the projection of $(1,j,k)$ onto $\bk_0$,
with minimum $j_0$ and maximum $j_1$. We will show that the partition function for paths starting at $(1,j_0)$
and ending at $(1,j_1)$ on the network $N(j_0,j_1)$, and
with weights $\bt^+$, is equal to the partition function from $(i_0,\bar j_0)$ to $(1,j_1)$ on the network $N(\bar j_0,j_1)$,
for some $i_0\in [1,r+1]$ and $\bar j_0 \in [0,j_1]$. 
The latter portion $N(\bar j_0,j_1)$ of the network has only positive weights from the set $\bt^+$, hence positivity follows. 

The formula for $(i_0,\bar j_0)$ depends on the value of $j_0$. Three cases may occur:

\begin{itemize}
\item $j_0\geq 0$: $(i_0,\bar j_0)=(1,j_0)$. The solution is identical to that of the unrestricted $A_r$
$T$-system of type (i),  and positivity follows from Theorem \ref{unposiT}.

\item $j_0\leq -r-1$: $(i_0,\bar j_0)=(r+1,-r-1-j_0)$. This is a consequence of the collapse relations 
on the network solution,
$$ T_{1,j,k}=\big[ N(j_0,j_1)\big]_{1,1}\, t_{1,j_1} 
=\big[ PN({\bar j}_0,j_1)\big]_{1,1}\, t_{1,j_1}
=\big[ N({\bar j}_0,j_1)\big]_{r+1,1}\, t_{1,j_1},$$
where ${\bar j}_0=-r-1-j_0$ (see Figure \ref{fig:bouquetfinal} for an illustration). 

\begin{figure}
\centering
\includegraphics[width=15.cm]{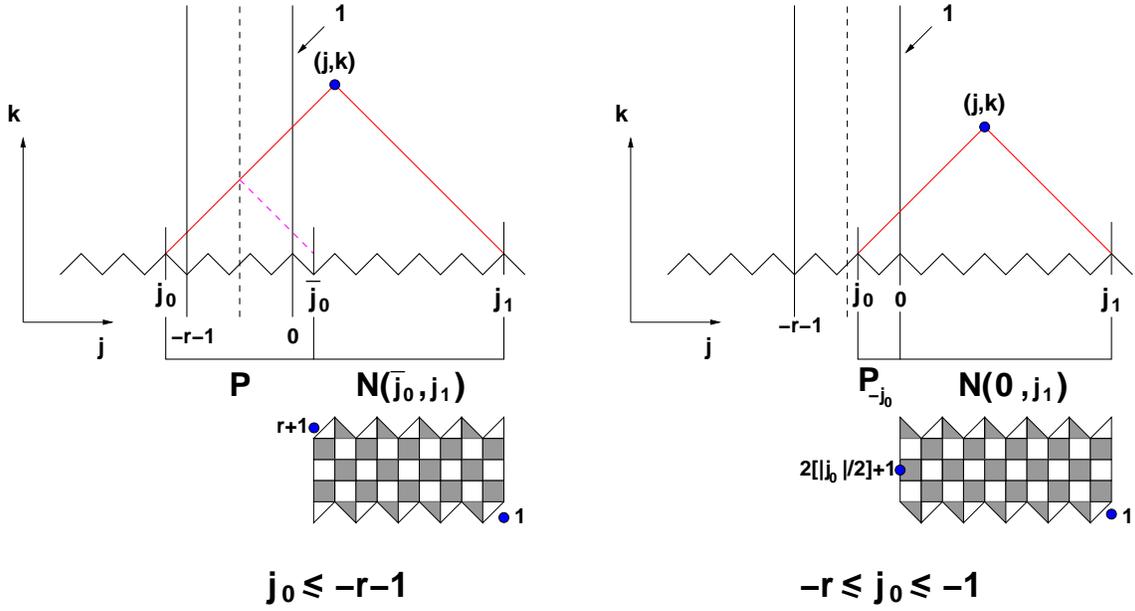}
\caption{The cases (ii) $j_0\leq -r-1$ and (iii) $-r\leq j_0 \leq -1$ for the position of the minimum
of the projection of $(1,j,k)$ onto $\bk_0$. We have represented in both cases the resulting network
and the position of the starting and ending point of the paths whose partition function
produces $T_{1,j,k}/T_{1,j_1,k_{j_1}^{(0)}}$.}\label{fig:bouquetfinal}
\end{figure}

\item $-r\leq j_0 \leq -1$: $(i_0,\bar j_0)=(2\lfloor\frac{|j_0|}{2}\rfloor+1,0)$. This follows by applying Lemma \ref{pj}
to the network solution:
\begin{equation*} 
T_{1,j,k}=\big[ N(j_0,j_1)\big]_{1,1}\, t_{1,j_1} =\big[ P_{-j_0}N(0,j_1)\big]_{1,1}\, t_{1,j_1}
=\big[N(0,j_1)\big]_{2\lfloor\frac{|j_0|}{2}\rfloor+1,1}\, t_{1,j_1}.\end{equation*}
(see the right of Figure \ref{fig:bouquetfinal} for an illustration).
\end{itemize}
\end{proof}

\begin{cor}\label{barposi}
 The solutions $T_{i,j,k}$ of the half-space $A_r$ $T$-system of type (ii) are non-negative Laurent polynomials of the initial data $\{t_{i,j}: i\in[1,r],j>0\}$, assigned at points on $\bk_0$ with $j>0$.
\end{cor}
\begin{proof}
We use the determinant formula \eqref{wronsk} for $T_{i,j,k}$.
Let $j_0(b)=j-k+i-2(b-1)$, $b=1,2,\ldots ,i$ and $j_1(a)=j+k-i+2(a-1)$,
$a=1,2,\ldots ,i$ be respectively the minima and the maxima of the projections
of the $i^2$ points $(1,j+a-b,k+a+b-i-1)$ involved in the formula.
From the proof of the previous lemma,
the quantity $Z_{a,b}=T_{1,j+a-b,k+a+b-i-1}/t_{1,j_1(a)}$
is the partition function for paths on $N(0,j_1)$, ending at  position $(1,j_1(a))$
and starting at position
$(i_0(b),\bar j_0(b))$, defined as the pair $(i_0,\bar j_0)$ of Lemma \ref{posibar} for $j_0=j_0(b)$.
The Lindstr\"om-Gessel-Viennot Theorem \cite{LGV1,LGV2} gives an interpretation of the 
determinant 
$T_{i,j,k}/\prod_{a=1}^i t_{1,j_1(a)}
=\det_{1\leq a,b\leq i} (Z_{a,b})$, as the partition
function of $i$ non-intersecting paths on the network of $N(0,j_1)$ 
with $i$ starting points $(i_0(b),\bar j_0(b))$ ($b\in[1,i]$),
and with $i$ ending points at positions $(1,j_1(a))$ ($a\in[1,i]$).
We deduce the positive Laurent property from the positivity of weights.
\end{proof}
We have the analogous result for the left half-space $T$-system:
\begin{lemma}\label{leftCor}
The solutions of the left half-space $A_r$ $T$-system (iii) are Laurent polynomials with non-negative integer 
coefficients of the initial data $\{t_{i,j}: i\in[1,r], j\leq \ell\}$ assigned at the points of the surface $\bk_0$ with $j\leq \ell$.
\end{lemma}
\begin{proof}
The proof is identical to that of Lemma \ref{posibar} and Corollary \ref{barposi}. We start from Theorem \ref{corolla}
to identify $T_{1,j,k}$ as the solution of the $A_r$
$T$-system of type (i) with initial data $\bt^-$. Consider the projection of $(1,j,k)$ onto $\bk_0$,
with minimum $j_0$ and maximum $j_1$. We will show that the product of $t_{1,j_1}$
with the partition function for paths starting at $(1,j_0)$
and ending at $(1,j_1)$ on the network $N(j_0,j_1)$, and
with weights $\bt^-$, is equal to the product of $t_{\lambda_1,\bar j_1}$ with the 
partition function from $(1,j_0)$ to $(i_1,\bar j_1)$ 
on the network $N(j_0,\bar j_1)$,
for some $\lambda_1\in [1,r]$, $i_1\in [1,r+1]$ and $\bar j_1 \in [j_0,\ell+1]$. 
The latter portion $N(j_0,\bar j_1)$ of the network has only positive weights from the set $\bt^-$, hence positivity follows. 

The formula for $(\lambda_1,i_1,\bar j_1)$ reads as follows. Three cases may occur:

\begin{itemize}
\item $j_1\leq \ell+1$: $(\lambda_1,i_1,\bar j_1)=(1,1,j_1)$. The solution is identical to that of the unrestricted 
$A_r$ $T$-system of type (i), 
and positivity follows from Theorem \ref{unposiT}.

\item $j_1\geq \ell+r+2$: $(\lambda_1,i_1,\bar j_1)=(r,r+1,2\ell+r+3-j_1)$. This is a consequence of the collapse relations 
on the network solution, and of the symmetries of $\bt^-$:
$$ T_{1,j,k}=\big[ N(j_0,j_1)\big]_{1,1}\, t_{1,j_1} 
=\big[ N(j_0,{\bar j}_1) (-1)^r P\big]_{1,1}\, (-1)^rt_{r,{\bar j}_1}
=\big[ N(j_0,{\bar j}_1)\big]_{r+1,1}\, t_{r,{\bar j}_1},$$
where we have used Lemma \ref{otherP} and ${\bar j}_1=2\ell+r+3-j_1$. 

\item $\ell+1<j_1<\ell+r+2$: $(\lambda_1,i_1,\bar j_1)=(1,a_\ell(j_1-\ell-1),\ell+1)$, with $a_\ell(x)$ is as in \eqref{defaell}.
This follows by applying Lemma \ref{otherPlim} to the (regularized) network solution:
\begin{equation*} 
T_{1,j,k}=\lim_{b_1,...,b_r\to 0}
\big[ N(j_0,\ell+1){\tilde P}_{j_1-\ell-1}(\{b\}) \big]_{1,1}\, b_{1,j_1}
=\big[N(j_0,\ell+1)\big]_{1,a_\ell(j_1-\ell-1)},\end{equation*}
and noting that $1=t_{1,\ell+1}$.
\end{itemize}
The equivalent of Corollary \ref{barposi} follows from interpreting \`a la Gessel-Viennot the quantity
$T_{i,j,k}/\prod_{a=1}^i t_{\lambda_1(a),\bar j_1(a)}$, as the partition function for $i$ non-intersecting paths on the network
that start at $(1,j_0(b))$ ($b\in [1,i]$) and end at $(i_1(a),\bar j_1(a))$ ($a\in [1,i]$). Positivity follows.
\end{proof}

Theorem \ref{positiT} follows from:
\begin{lemma}
  The solutions of the $A_r$ $T$-system with initial data of type
  $\bt^{[1,\ell]}$ are Laurent polynomials with non-negative integer
  coefficients of the variables $\{t_{i,j}: i\in[1,r],
  j\in[1,\ell]\}$, assigned along the points of $\bk_0$ with $0<j\leq
  \ell$.
\end{lemma}
\begin{proof}
We start by proving the property for $T_{1,j,k}$.
The half-periodicity property of Theorem \ref{haper} allows us to
restrict to points $(1,j,k)$ with $0\leq k\leq \frac{N}{2}$.  

Let $j_0$ and $j_1$ be the minimum and maximum of the projection of $(1,j,k)$ onto $\bk_0$. Using the definitions
of $i_0,\bar j_0,\lambda_1,i_1,\bar j_1$ given in the proofs of Lemmas \ref{posibar} and \ref{leftCor},
we will show that the network partition function for paths from $(1,j_0)$ to $(1,j_1)$, multiplied by $t_{1,j_1}$ is equal to
$t_{\lambda_1,\bar j_1}$ times the partition function of paths from $(i_0, \bar j_0)$ to $(i_1,\bar j_1)$
where $0\leq \bar j_0\leq \bar j_1\leq \ell+1$. The weights in this region of $\bt^{[1,\ell]}$ are all positive, hence so is the partition
function. Define the following subsets of $\Z$:
\begin{eqnarray*}
&A=[0,\ell+1],\quad  B=[-\ell-r-2,-r-1],\quad  C=[-r,-1], &\\ 
& D=[r+\ell+2,r+2\ell+3],\quad E=[\ell+2,r+\ell+1].&
\end{eqnarray*}
When $j_0$ or $j_1\in A$, the solution is identical to that of Theorem
\ref{unposiT}, Lemma
\ref{haposi} or Corollary \ref{leftCor}, in which positivity has been proven.
There are four remaining cases.

\begin{itemize}
\item $(j_0,j_1)\in B\times D$: We use collapse relations on both sides:
$$ T_{1,j,k}=\big[ P N({\bar j}_0,{\bar j}_1) (-1)^rP\big]_{1,1} t_{1,j_1}=
\big[N({\bar j}_0,{\bar j}_1)\big]_{r+1,r+1} t_{r,{\bar j}_1}, $$ 

\item $(j_0,j_1)\in B\times E$: We use collapse relations on the left, and Lemma \ref{otherPlim} on the right:
$$ T_{1,j,k}=\lim_{b_1,...,b_r\to 0}\big[ P N({\bar j}_0,\ell+1) {\tilde P}(\{b\})_{j_1-\ell-1} \big]_{1,1} b_{1,j_1}=
\big[ N({\bar j}_0,\ell+1)\big]_{r+1,a_\ell(j_1-\ell-1)}$$

\item $(j_0,j_1)\in C\times D$: We use collapse relations on the right, and Lemma \ref{pj} on the left:
$$ T_{1,j,k}=\big[ P_{-j_0} N(0,{\bar j}_1) (-1)^rP \big]_{1,1} t_{1,j_1}=
\big[ N(0,{\bar j}_1)\big]_{a_1(|j_0|),r+1} t_{r,{\bar j}_1}$$

\item $(j_0,j_1)\in C\times E$: We use Lemma \ref{pj} on the left and Lemma \ref{otherPlim} on the right:
$$ T_{1,j,k}=\lim_{b_1,...,b_r\to 0}\big[ P_{-j_0} N(0,\ell+1) {\tilde P}_{j_1-\ell-1}(\{b\}) \big]_{1,1} b_{1,j_1}=
\big[ N(0,\ell+1)\big]_{a_1(|j_0|),a_\ell(j_1-\ell-1)} $$
\end{itemize}

To summarize, in all cases $T_{1,j,k}$ is expressed in terms of the
partition function for paths on the {\it same} network but with
different starting and ending positions, depending on the values of
$j$ and $k$. We may now apply the
Lindstr\"om-Gessel-Viennot theorem \cite{LGV1,LGV2} to the determinant
expression of Lemma \ref{elimdet}. Let
$j_0(b)=j-k+i-2(b-1)$ and $j_1(a)=j+k-i+2(a-1)$ be the minima and maxima of the projections
of the points $(1,j',k')$ involved in the determinant.  We interpret the quantity
$T_{i,j,k}/\prod_{a=1}^i t_{\lambda_1(a),\bar j_1(a)}$, as the partition function for $i$
non-intersecting paths on the network, starting at $(i_0(b),\bar j_0(b))$ ($b\in [1,i]$) and ending
at $(i_1(a),\bar j_1(a))$ ($ a\in [1,i]$).
This proves positivity, as all the path weights are
positive Laurent monomials of the initial data $t_{i,j}$ in the positive part of $\bt^{[1,\ell]}$.

\begin{figure}
\centering
\includegraphics[width=15.cm]{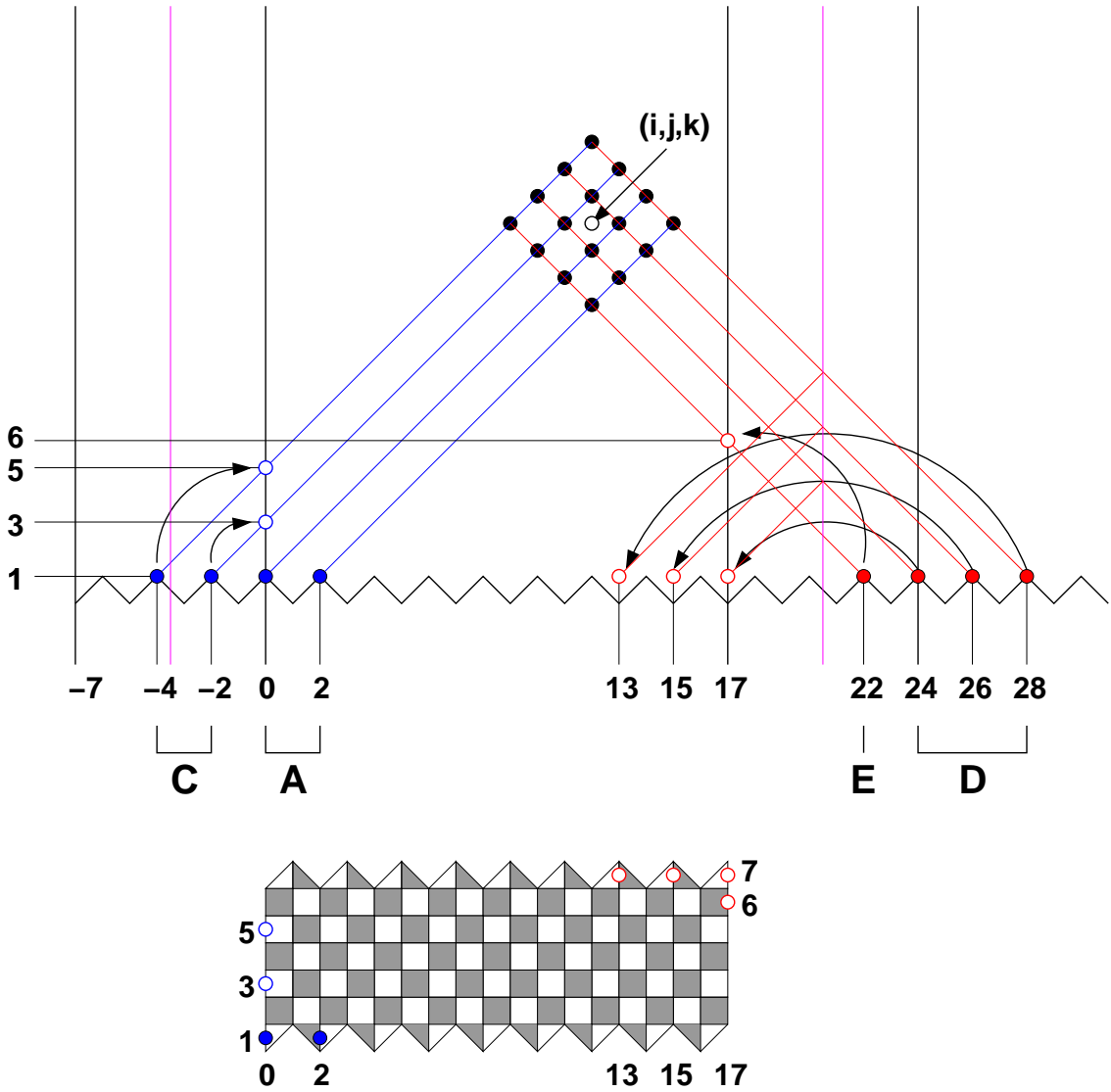}
\caption{A typical example of expression of $T_{i,j,k}$ as non-intersecting path partition function.
Here $(i,j,k)=(4,12,14)$, $r=6$ and $\ell=16$. We have indicated the $16$ 
terms involved in the determinant of Lemma \ref{elimdet} by black dots, and by blue (resp. red dots) their
projection minima (resp. maxima) onto $\bk_0$. The corresponding starting and endpoints are indicated
on the associated network picture in representation I.}\label{fig:parlutfin}
\end{figure}

Typically, depending on $i,j,k$, and as $i\leq r$, we may have at worst some of the 
minima in $A$ and the rest in $C$, 
or some in $B$ and the rest in $C$, and similarly for the maxima, either in $A\cup D$ or in $A\cup E$. 
We illustrate this in Fig. \ref{fig:parlutfin} for $r=6$ and $\ell=16$, and $(i,j,k)=(4,12,14)$. 
In this case, $A=[0,17]$, $B=[-24,-7]$, $C=[-6,-1]$, $D=[24,41]$, $E=[18,23]$.
The minima
$j_0(1)=-4,j_0(2)=-2$ are both in $C$ and give rise to respective starting points $(5,0),(3,0)$, while
$j_0(3)=0,j_0(4)=2$ are both in $A$ and give rise to starting points $(1,0),(1,2)$.
The maxima are $j_1(1)=22$ in $E$ giving rise to the endpoint $(a_\ell(j_1(1)-\ell-1),\ell+1)=(6,17)$,
and $j_1(2)=24,j_1(3)=26,j_1(4)=28$, all in $D$, giving rise to the endpoints 
$(7,17),(7,15),(7,13)$. The network partition function for these $4$ non-intersecting
paths is equal to $T_{4,12,14}/(t_{6,13}t_{6,15})$.

\begin{figure}
\centering
\includegraphics[width=15.cm]{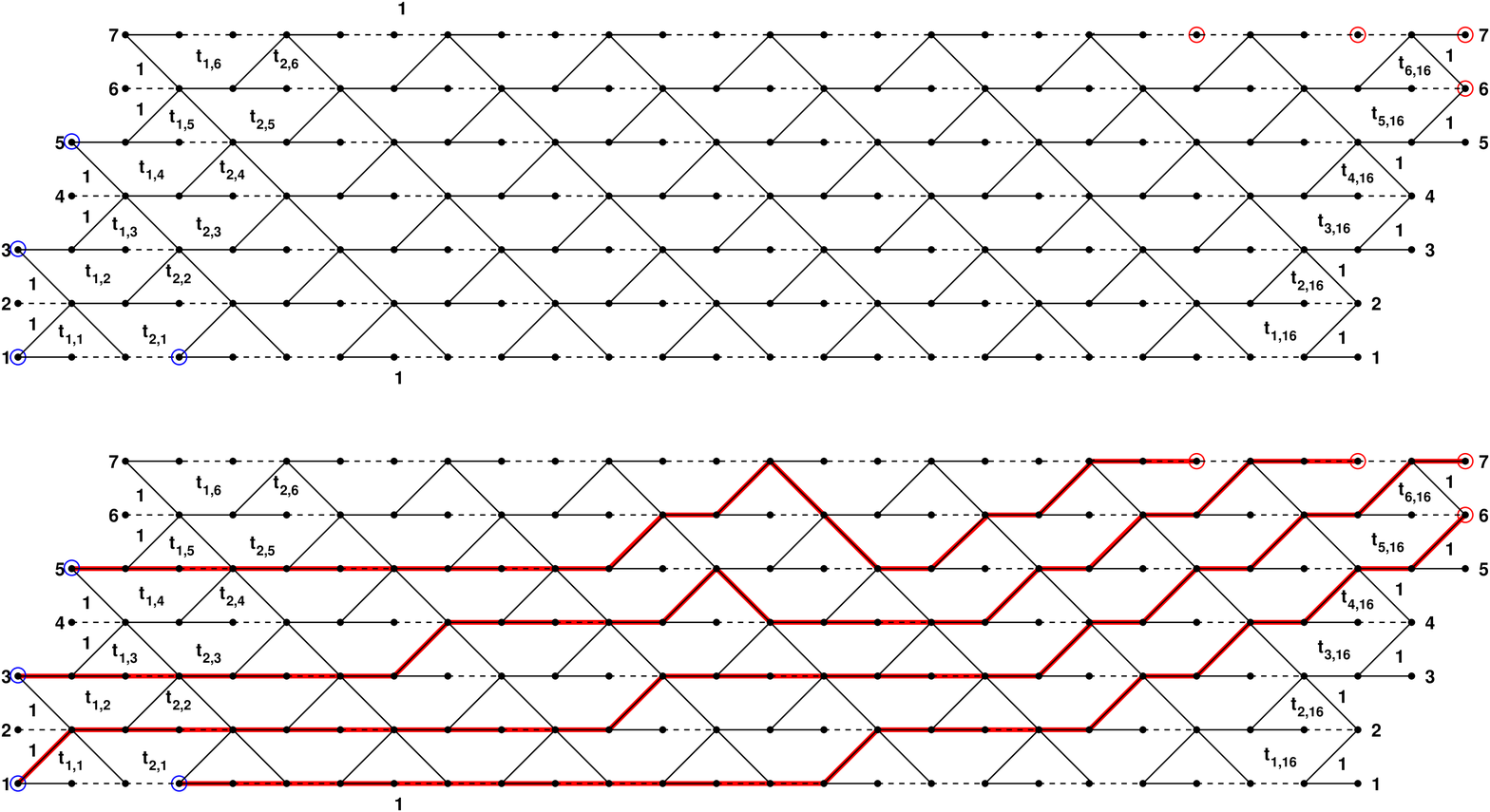}
\caption{The pictorial representation II of the network of Fig.\ref{fig:parlutfin}, with starting points circled in blue
and endpoints circled in red. We have also represented a typical configuration of non-intersecting
paths on the network that contribute to $T_{4,12,14}$.}\label{fig:luttefinale}
\end{figure}

We have represented the corresponding network in the pictorial representation II in Fig.\ref{fig:luttefinale},
together with a typical configuration of four non-intersecting paths that contributes to $T_{4,12,14}$.
\end{proof}

\section{Other boundary conditions}
\label{other}

So far, we have used the network solution of the $T$ systems of type $A$
to find expressions for their solutions for wall-type boundary conditions. 
In particular, in the $\ell$-restricted case, this gives a combinatorial 
proof of Zamolodchikov's periodicity conjecture.

There are other interesting types of boundary conditions on the $T$-system, and we comment on some of them below.

\subsection{Friezes}

The so-called $SL_2$ frieze patterns \cite{Cox,FRISES} are known to obey the $A_1$ $T$-system relation.
In \cite{FRISES}, special boundaries were considered,
coded by affine Dynkin diagrams. In particular,
the case of the affine Dynkin diagram $\tilde A_\ell$ with an acyclic orientation 
corresponds to the $A_1$ $T$-system
with
a periodic initial data path $\bk$, such that $k_{j+\ell}=k_j+m$ for some fixed integer $m$ and  for all $j\in \Z$, and 
periodic initial values $\bt$ along this path, with $t_{j+\ell}=t_j$ for all $j\in \Z$. The 
Laurent positivity for this case follows immediately from that of the unrestricted system.
Note that the case of the ordinary Dynkin diagram $A_\ell$ corresponds 
to the $\ell$-restricted boundaries. 

In the context of higher rank $T$-systems, boundary conditions coded by pairs $(G,G')$ of Dynkin diagrams
lead to the most general periodicity conjecture of Zamolodchikov, proved in \cite{Keller}. 
In that context, the first Dynkin diagram
codes the type of $T$-system ($A_r$ throughout this paper), while the second codes the particular
boundary conditions ($A_\ell$ for the $\ell$-restricted boundaries for instance). 

In a way similar to frieze patterns, we may consider the case $(A_r,{\tilde A}_{\ell-1})$, for even integers $\ell\geq 2$, 
where the $A_r$ $T$-system solutions are $\ell$-periodic in the $j$ direction, with $T_{i,j+\ell,k}=T_{i,j,k}$.
This is guaranteed by imposing that both the stepped surface $\bk$ 
and the attached initial data $t_{i,j}$ of \eqref{initcondar} be $\ell$-periodic in the $j$ direction, i.e. 
$k_{i,j+\ell}=k_{i,j}+m_i$
for some fixed $m_i$ compatible with the stepped surface conditions of Def.\ref{steppeddef} and
$t_{i,j+\ell}=t_{i,j}$ for all $i\in [1,r]$ and $j\in\Z$. The positivity of the corresponding $T$-system 
follows from Theorem \ref{unposiT}.

We may also consider the case $({\tilde A}_{r-1},{\tilde A}_{\ell-1})$ for even $r,\ell\geq 2$, 
in which the $T$-system is wrapped on a torus, by imposing that the $T$-system solutions be doubly periodic, with
$T_{i+r,j,k}=T_{i,j+\ell,k}=T_{i,j,k}$ for all $i,j,k\in \Z$. The solutions
of the corresponding system are obtained from those of the unrestricted $A_\infty$ $T$-system by imposing that 
both the stepped surface and the initial data of \eqref{initcond} be doubly periodic as well. Positivity then follows from
Theorem \ref{unposiT}.

\subsection{Higher pentagram maps as $T$-system tori}

The pentagram map has been shown to relate to cluster algebra, and its solution
was expressed in \cite{GLICK} in terms of some particular $T$-system solution. 
Higher versions of this map were considered by \cite{GEKH}. 
In all cases, we note that these correspond to quivers that are quotients of the
$T$-system quiver by a torus, defined as follows.

Let us consider
the solutions of the unrestricted $A_\infty$
$T$-system, with initial data $\bt$ along the stepped surface $\bk_0$ satisfying
a {\it toric periodicity} property. Let us
fix $\vec a=(a_1,a_2)$, $\vec b=(b_1,b_2)$ two non-collinear vectors in $\Z^2$
and such that $a_1+a_2$ and $b_1+b_2$ are even. 
We impose the double periodicity property:
$$\Theta_{\vec a,\vec b}:\ \  t_{i+a_1,j+a_2}=t_{i,j} \qquad t_{i+b_1,j+b_2}=t_{i,j} $$
This is a generalization of the rectangular torus case $(\tilde A_{r-1}, \tilde A_{\ell-1})$ 
described in the previous section, corresponding to ${\vec a}=(r,0)$ and ${\vec b}=(0,\ell)$. 

In the cluster algebra identification for the $T$-system, the seed of the cluster
algebra is made of a cluster and an exchange matrix, both of infinite size, as the rank is infinite.
The cluster is the set of initial values $\{t_{i,j}\}$ along the stepped surface $\bk_0$, and gets mutated into other 
initial data. 
The exchange matrix
is coded by the quiver $Q_0$ with vertices $(j,i)\in \Z^2$ and oriented edges $(j,i)\to (j\pm 1,i)$,
and $(j,i\pm 1)\to (j,i)$ for all $i,j\in \Z^2$ with $i+j$ even. Note that the edge configurations around
even vertices ($(j,i)$ with $i+j$ even) are opposite to those around odd ones ($(j,i)$ with $i+j$ odd).
The parity conditions ($a_1+a_2$ and $b_1+b_2$ even) guarantee that only vertices of the same parity
are identified. By taking a quotient of $\Z^2$ by the lattice $\Z\vec a+\Z\vec b$,
this allows to fold the corresponding infinite quiver $Q_0$ into a {\it finite} one ${\tilde Q}_0$. The example below illustrates
the case $\vec a=(0,2)$ and $\vec b=(3,1)$:
$$  \raisebox{-1.cm}{\hbox{\epsfxsize=9.cm \epsfbox{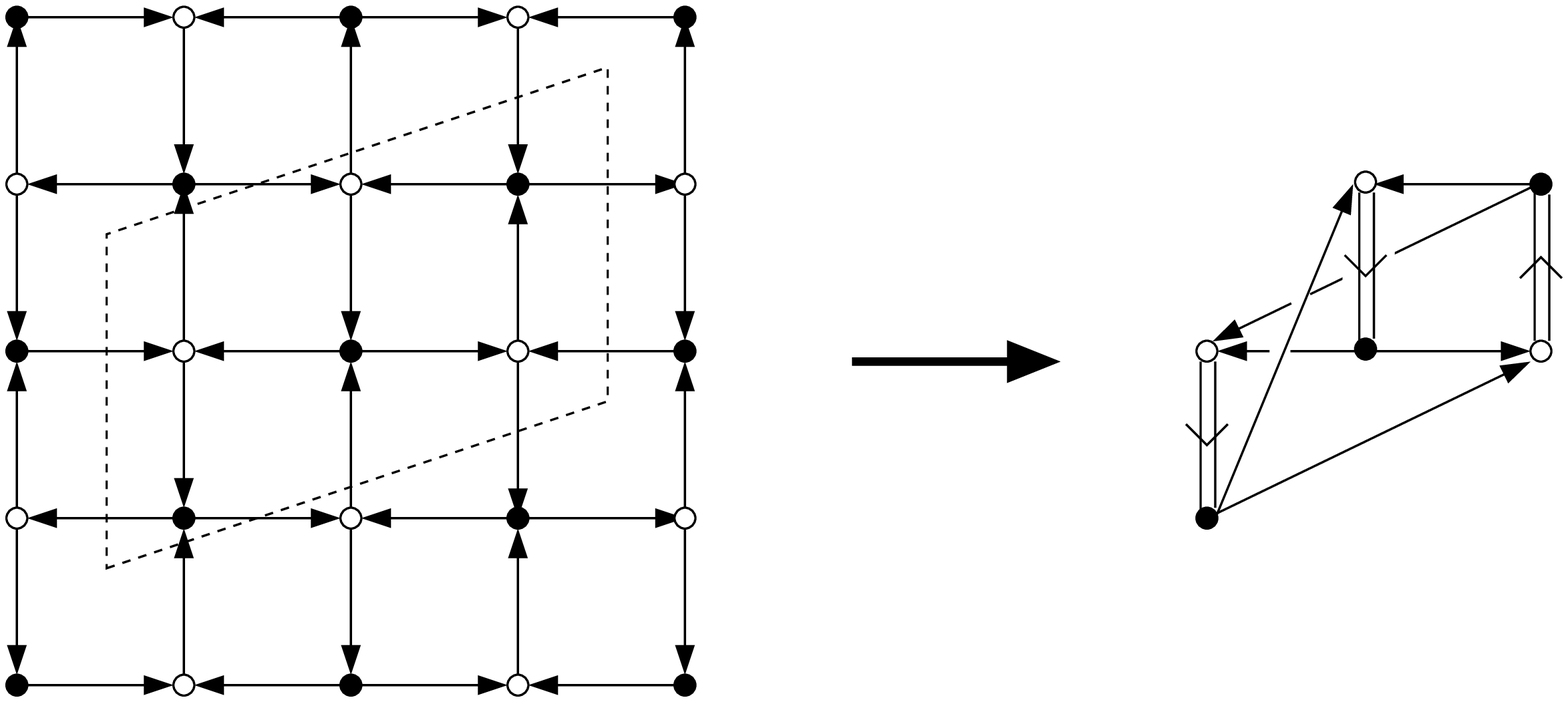}}} $$
The folded system obtained by considering torus-periodic initial values and performing a quotient by $\Z\vec a+\Z\vec b$
is also part of a cluster algebra. Its mutations $\mu_{i,j}$ correspond to considering infinite compound 
(mutually commuting) mutations $\prod_{m,n\in \Z} \mu_{i+na_1,mb_1,j+na_2+mb_2}$ and passing to the quotient.

The quiver of the initial seed of the cluster algebra underlying the pentagram map of \cite{GLICK} 
was generalized to a quiver $Q_{k,n}$ for higher pentagram maps in \cite{GEKH} 
(the former case corresponding to $k=3$). The quiver
$Q_{k,n}$ has two (even and odd) sets of vertices denoted by $p_i,q_i$, $i\in\Z$ with periodic
identifications $p_{i+n}=p_i$ and $q_{i+n}=q_i$ which makes the quiver finite, with $2n$ vertices.
It is easy to rewrite the (infinite) quiver before identifications as that, $Q_0$, of the unrestricted $A_\infty$ $T$-system  with
vertices $(j,i)\in\Z^2$ as follows:
$$ \raisebox{-1.cm}{\hbox{\epsfxsize=7.cm \epsfbox{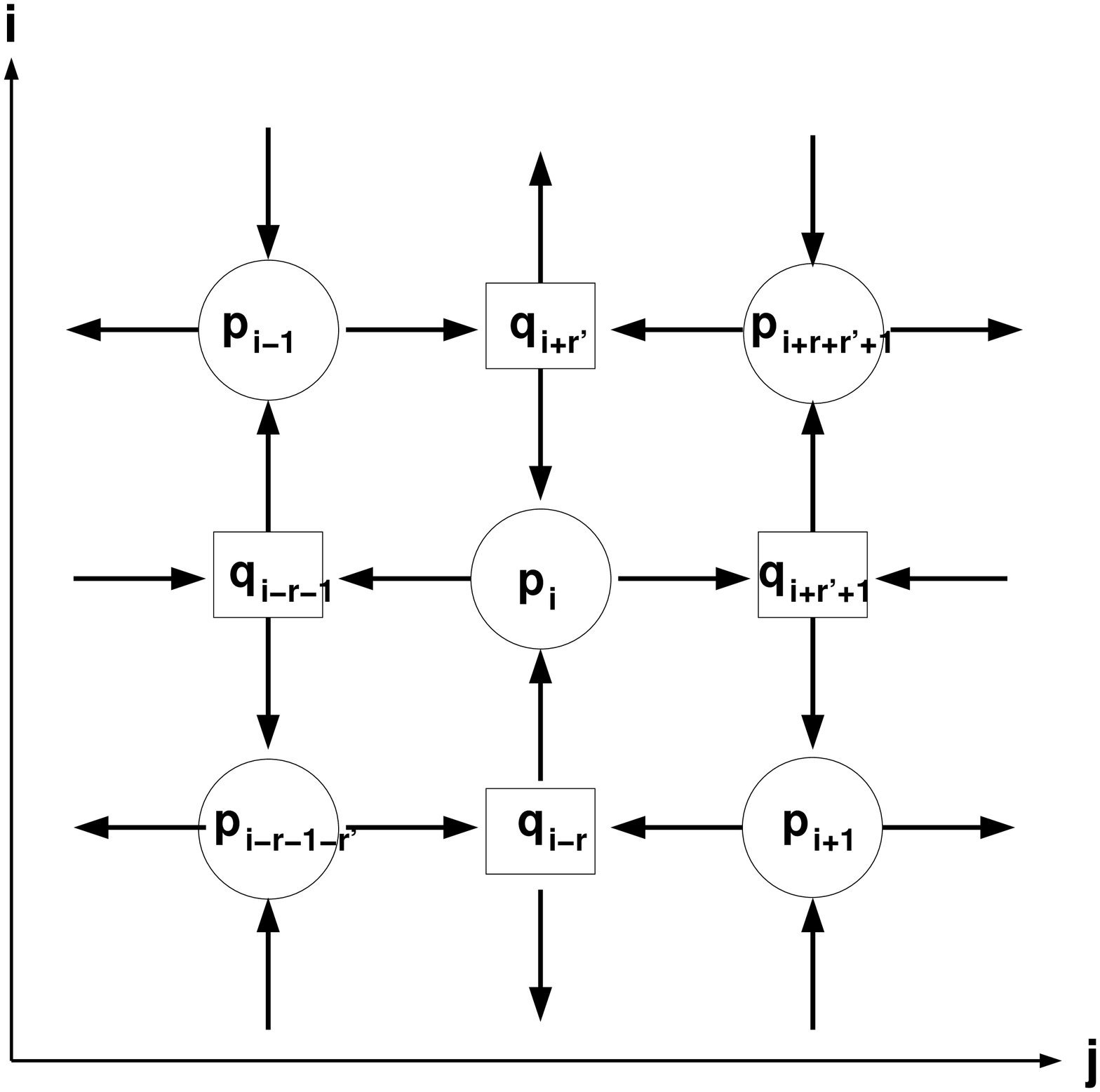}}} $$
where $r,r'$ are two non-negative integers such that $r+r'=k-2$ and $r'=r$ if $k$ is even, and $r'=r+1$ otherwise,
and where we have represented by circles (resp. squares) the vertices $(j,i)\in \Z^2$ 
with $i+j$ even (resp. odd). 
Note that horizontal arrows correspond to a shift $+r'+1$ (resp. $-r-1$)  in the indices when the arrow points
to the right (resp. left), while vertical arrows correspond to a shift $+r$  (resp. $-r'$)  in the indices when the arrow points
up (resp. down).
The identifications
$p_i\equiv p_{i+n}$ and $q_i=q_{i+n}$ of the two types of vertices are therefore equivalent to a double
translational invariance. A simple calculation shows that
it corresponds to the torus $\Z^2/\Z \vec{a}+\Z\vec{b}$, with $\vec a=(-(k-2),k)$ and
$\vec b=(-n,n)$ with a fundamental domain under these translations
containing $|\vec{a}\wedge \vec{b}|=2n$ vertices. 
In all cases, the finite quiver $Q_{k,n}$ with $2n$ vertices is the corresponding folding of $Q_0$.

\subsection{Cuts}

We may consider solutions of the $T$-system with {\it cuts} defined as follows. 
We consider the unrestricted $A_\infty$ $T$-system, 
and pick a set $S$ of tetrahedra (called singular)
along which the $T$-system relation is not imposed.  
A natural question for this system is: for which choices of $S$ and of  initial conditions does the
positive Laurent property hold?

In the $A_1$ case, $S$ is a set of
diamonds of the form $(W_i,S_i,E_i,N_i)=\big((j_i-1,k_i),(j_i,k_i-1),(j_i+1,k_i),(j_i,k_i+1)\big)$ for $i=1,...,|S|$.
The particular case of a single diamond $|S|=1$ (which we refer to as ``puncture") 
may be solved by the techniques of the present paper. Consider an initial data surface of the form:
$$ \raisebox{-1.cm}{\hbox{\epsfxsize=10.cm \epsfbox{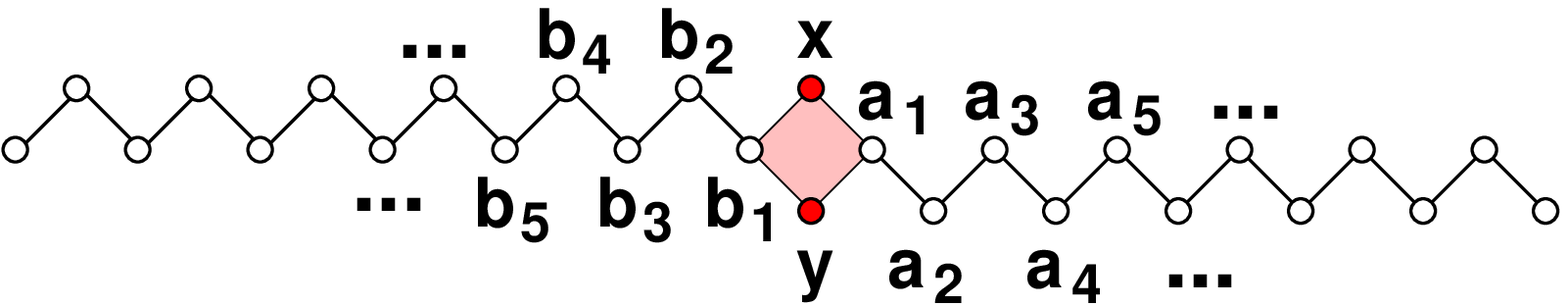}}}$$
with the following initial conditions. Let $\epsilon_j=0$ if $j$ is odd, $\epsilon_j=-1$ if $j$ is even positive
and $\epsilon_j=1$ if $j$ is even negative. Then the initial data assignment for the system with a puncture
at $((-1,0),(0,-1),(1,0),(0,1))$ is:
\begin{eqnarray*}
T_{j,\epsilon_j}&=&\left\{ \begin{matrix} a_j & {\rm if}\ \ j>0\\
b_j & {\rm if}\ \ j<0 \end{matrix} \right. \\
T_{0,1}&=& x \\
T_{0,-1}&=& y \end{eqnarray*}
for $x,y$ some formal invertible variables.
The Laurent property of the solutions is a direct consequence of Theorem \ref{unrestaone}. In fact,
for any point above the initial data paths, the solution is a Laurent polynomial of the $a$'s, $b$'s and
of $x$, while for any point below the initial data paths, the solution is a Laurent polynomial of the $a$'s, 
$b$'s and of $y$. More interestingly, we may consider different initial data paths containing the edges of
the puncture, such as the example below:
$$  \raisebox{-1.cm}{\hbox{\epsfxsize=9.cm \epsfbox{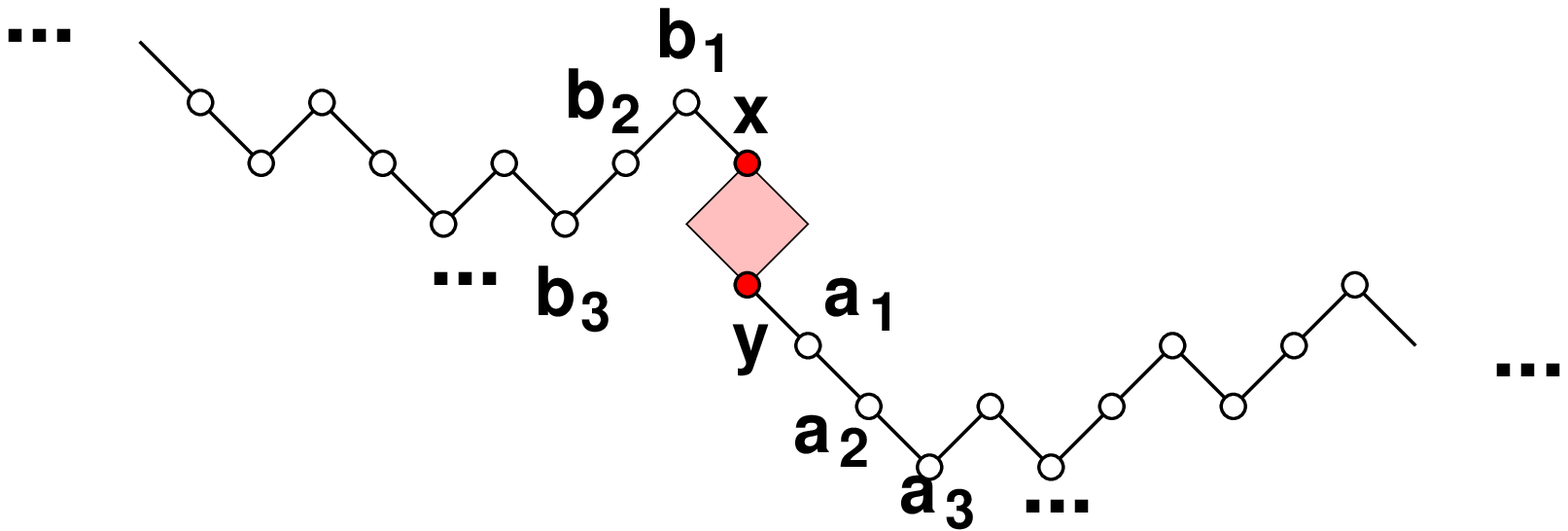}}} $$
where the two uncircled vertices of the puncture bear no assignment.
It is easy to show that a necessary condition for the Laurent property to still hold is that $xy=1$.
Conversely, we have the following:

\begin{thm}
The solution of the $A_1$ $T$-system with a puncture, an arbitrary path of initial conditions (passing by the puncture)
as above,  and with $xy=1$, is a positive Laurent polynomial of the initial data.
\end{thm}
\begin{proof}
We start by reinterpreting the fundamental relation \eqref{muta} as a flatness condition of the form
$V(a,b') U(b',c)V(b,c)^{-1}U(a,b)^{-1}=\mathbb I $, around each diamond where the $T$-system
relation holds. We see that the matrices $V,U,V^{-1},U^{-1}$ corresponding to the 4 possible oriented
edges along the vectors $(1,-1),(1,1),(-1,1),(-1,-1)$ respectively form a flat connection on any domain
made of diamonds on which the $T$-system relation holds.
As a consequence, the product of corresponding matrices along the boundary of
any connected domain made of diamonds where the $T$-system relation holds is the identity matrix.
Due to the presence of the puncture, we may only consider domains that do not contain it.
Let $(j,k)$ be a point above the initial data paths, and $j_0,j_1$ the minimum and maximum of its
projection as usual. We consider the domain $D$ shaded below:
$$  \raisebox{-1.cm}{\hbox{\epsfxsize=8.cm \epsfbox{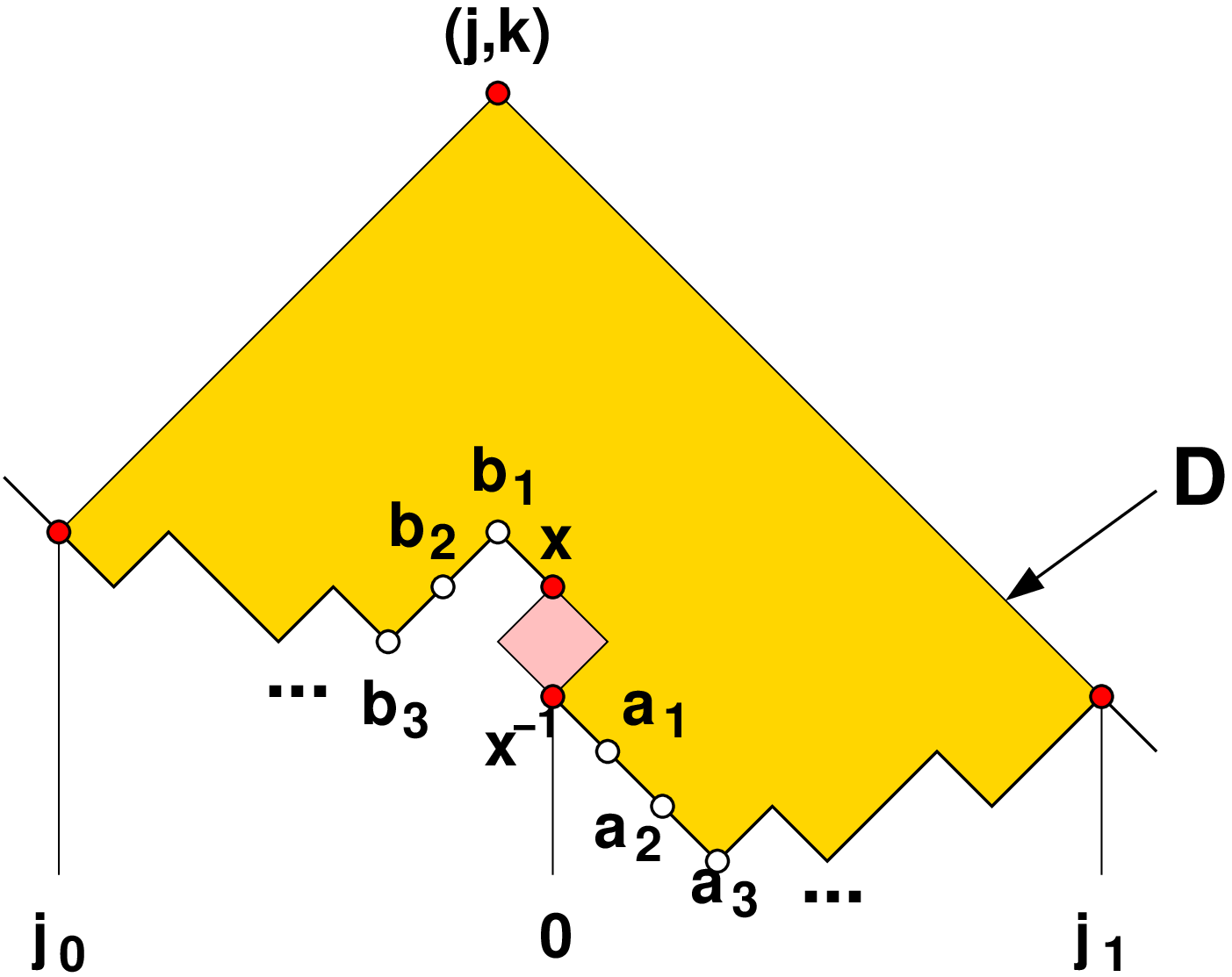}}}$$
In particular, the two (topmost and bottommost) paths joining the minimum and maximum contribute the
same matrix $M$, satisfying $M_{1,1} T_{j_1,k_{j_1}}=T_{j,k}$, clear from the topmost path expression
of the form $UU..UVV...V$. 
The contribution from
the bottommost path is the product along the initial data path (including edges of the puncture)
of the corresponding connection matrices. These have non-negative entries that are
Laurent monomials of the initial data,
except for edges of the puncture 
namely $V(x,w)U(x^{-1},w)^{-1}$ or $V(v,x^{-1})^{-1}U(v,x)$. The following local situations may occur in general:
$$  \raisebox{-1.cm}{\hbox{\epsfxsize=10.cm \epsfbox{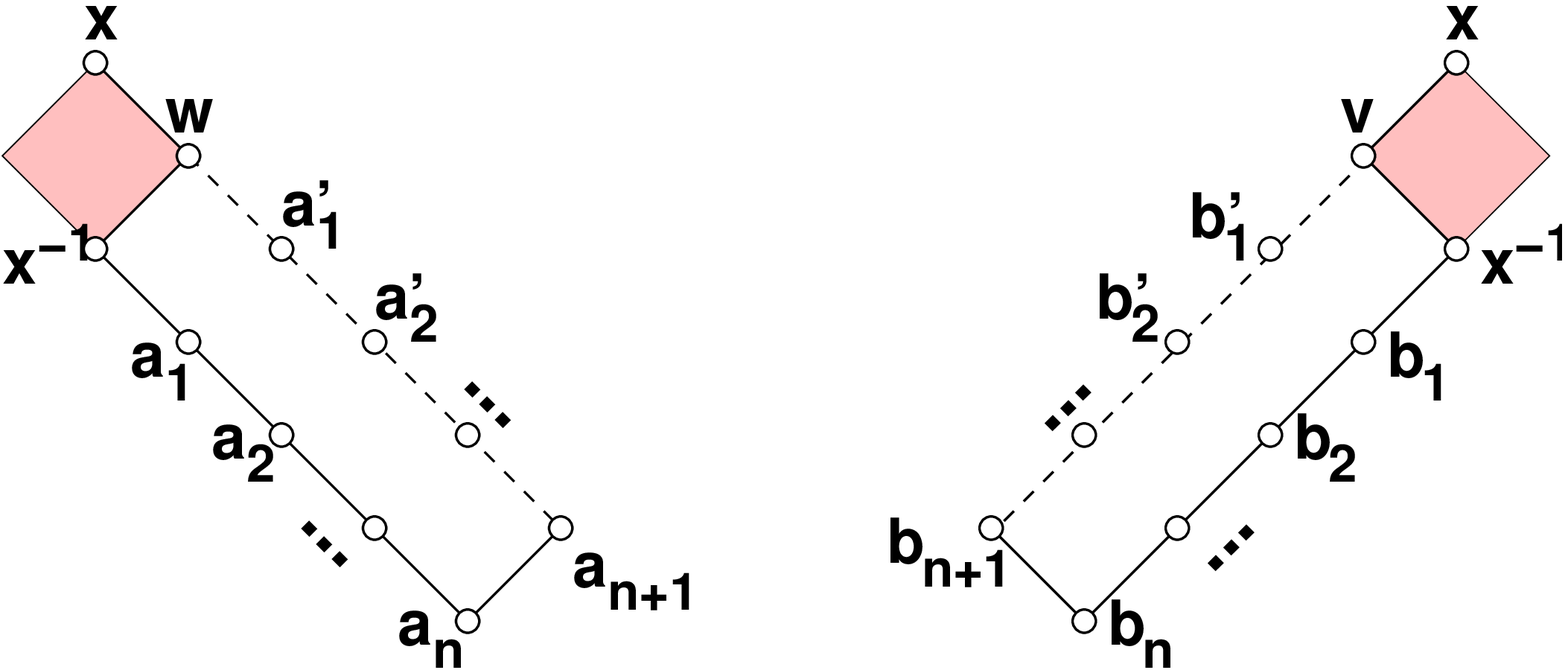}}}$$
leading to matrix products along the bottommost path of the form $...B_{L/R}...$, where $B_{L/R}$ is the only term
containing possibly negative or non-Laurent entries for the situation on the left/right, with respectively:
\begin{eqnarray*}
B_L&=& V(x,w)U(x^{-1},w)^{-1} V(x^{-1},a_1)V(a_1,a_2)...V(a_{n-1},a_n)U(a_n,a_{n+1})\\
B_R&=& V(b_{n+1},b_n)U(b_n,b_{n-1})U(b_{n-1},b_{n-2})...U(b_1,x^{-1})V(x^{-1},v)^{-1}U(v,x)
\end{eqnarray*}
Applying the flat connection condition to the domain made of the $n$ diamonds, we may express respectively:
\begin{eqnarray}
B_L&=& V(x,w)V(w,a_1')V(a_1',a_2')...V(a_{n-1}',a_{n+1})\label{prodone} \\
B_R&=& U(b_{n+1},b_{n-1}')U(b_{n-1}',b_{n-2}')...U(b_1',v)U(v,x)\label{prodtwo}
\end{eqnarray}
where for $a_n'=a_{n+1}$, $a_0'=w$, $a_0=x^{-1}$ and $b_n'=b_{n+1}$, $b_0'=v$, $b_0=x^{-1}$
we have the following descending recursion relations:
\begin{eqnarray*}
a_m'&=& \frac{a_{m+1}'a_m+1}{a_{m+1}} \qquad (m=n-1,n-2,...,1,0)\\
b_m'&=& \frac{b_{m+1}'b_m+1}{b_{m+1}} \qquad (m=n-1,n-2,...,1,0)
\end{eqnarray*}
It is straightforward to see that the $a_m',b_m'$ $m=1,2,...,n-1$, are all positive Laurent polynomials
of respectively $(x^{-1},a_1,a_2,...,a_{n+1})$ and $(x^{-1},b_1,b_2,...,b_{n+1})$. 
Moreover possible denominators involving $v$ or $w$ (which are not initial data) are suppressed in the products
\eqref{prodone} and \eqref{prodtwo}.
As a consequence
both products \eqref{prodone} and \eqref{prodtwo} have entries that are non-negative Laurent polynomials
of the initial data. As the remainder of the matrix products along the bottommost path only involve
matrices with non-negative Laurent monomial entries, the positivity follows.
\end{proof}

We may consider more general cases where $S$ is made of possibly several chains of diamonds
attached by their north/south vertices. We refer to such chains as ``cuts".  We observed that to guarantee the Laurent
property, the N/S vertex assignments of diamonds along these chains must alternate between a value
and its inverse $x,x^{-1},x,x^{-1},...$. Let us consider a general situation as depicted below of arbitrary paths
of initial data (empty circles) separated by chains as above, with initial data $x=x^{-1}=1$ (filled circles):
$$  \raisebox{-1.cm}{\hbox{\epsfxsize=16.cm \epsfbox{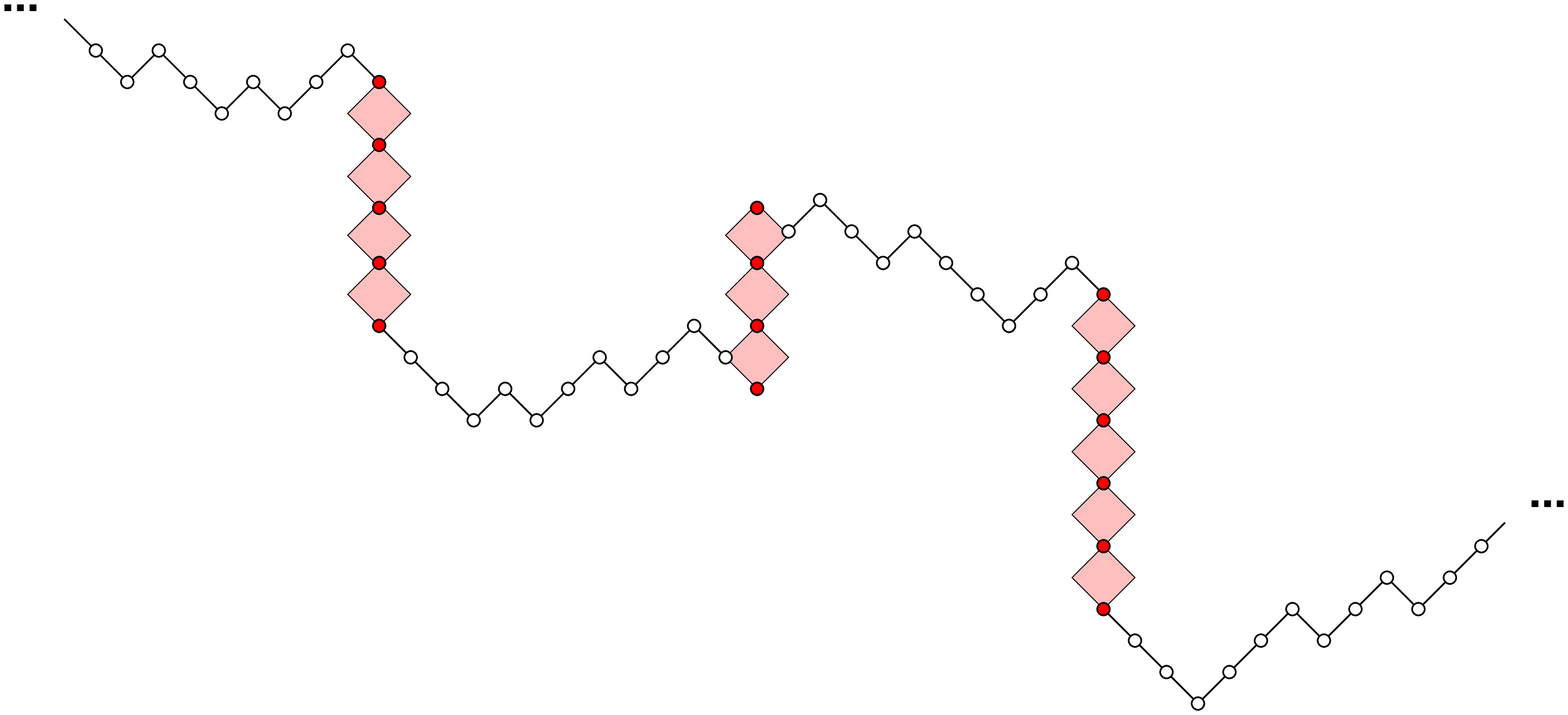}}}  $$
Then we conjecture that the solution to the $A_1$ $T$-system is a positive Laurent polynomial of the initial 
data at the empty circles.

In the $A_r$ case, we expect the above to generalize analogously, namely that a general situation with pieces
of stepped surfaces
separated by pieces of walls of singular octahedra attached by their vertices $(i,j,k)$ with the same 
value of $j$, at which the
assigned value of $T_{i,j,k}$ is $1$, leads to a solution that is a positive Laurent polynomial of the initial data along
the stepped surfaces.

This generalizes the $\ell$-restricted situation of the present paper, in which we consider two infinite parallel walls of singular tetrahedra
and a finite stepped surface in-between.

\medskip
\noindent{\bf Acknowledgments.} The authors acknowledge support 
by the CNRS PICS program.
PDF received partial support from the ANR Grant GranMa. RK is supported by NSF grant DMS-1100929. 
PDF would like to thank the Mathematical Science Research Institute in Berkeley, CA 
and the organizers of the semester ``Random Spatial Processes" for hospitality
while this work was completed.

\end{document}